\documentclass{amsart}

\usepackage[margin=1.4in]{geometry}
\usepackage{amsmath,amssymb,amsthm}
\usepackage[retainorgcmds]{IEEEtrantools}
\usepackage{enumerate}
\usepackage{pifont,mathtools}
\usepackage{tikz-cd}
\usepackage[square,numbers]{natbib}

\newtheorem{theorem}{Theorem}[section]
\newtheorem{observation}[theorem]{Observation}
\newtheorem{lemma}[theorem]{Lemma}
\newtheorem{proposition}[theorem]{Proposition}
\newtheorem{corollary}[theorem]{Corollary}

\numberwithin{equation}{section}

\theoremstyle{definition}
\newtheorem{definition}[theorem]{Definition}

\theoremstyle{remark}
\newtheorem{remark}[theorem]{Remark}
\newtheorem{example}[theorem]{Example}

\newcommand\R{\mathbb{R}}
\newcommand\C{\mathbb{C}}
\newcommand\Q{\mathbb{Q}}
\newcommand\Z{\mathbb{Z}}
\newcommand\N{\mathbb{N}}

\DeclareMathOperator{\Lip}{Lip}

\DeclareMathOperator{\vol}{vol}

\DeclareMathOperator{\Tr}{Tr}

\DeclareMathOperator{\Span}{Span}

\DeclareMathOperator{\TrP}{TrP}
\DeclareMathOperator{\NCP}{NCP}

\DeclareMathOperator{\Jac}{Jac}
\DeclareMathOperator{\II}{II}

\DeclarePairedDelimiter{\norm}{\lVert}{\rVert}
\DeclarePairedDelimiter{\ip}{\langle}{\rangle}

\def\smallint{\begingroup\textstyle \int\endgroup}

\begin{document}

\title{An Elementary Approach to Free Entropy Theory for Convex Potentials}

\author{David Jekel}
\address{Department of Mathematics, UCLA, Los Angeles, CA 90095}
\email{davidjekel@math.ucla.edu}
\urladdr{www.math.ucla.edu/{$\sim$}davidjekel/}

\begin{abstract}
We present an alternative approach to the theory of free Gibbs states with convex potentials.  Instead of solving SDE's, we combine PDE techniques with a notion of asymptotic approximability by trace polynomials for a sequence of functions on $M_N(\C)_{sa}^m$ to prove the following.  Suppose $\mu_N$ is a probability measure on on $M_N(\C)_{sa}^m$ given by uniformly convex and semi-concave potentials $V_N$, and suppose that the sequence $DV_N$ is asymptotically approximable by trace polynomials.  Then the moments of $\mu_N$ converge to a non-commutative law $\lambda$.  Moreover, the free entropies $\chi(\lambda)$, $\underline{\chi}(\lambda)$, and $\chi^*(\lambda)$ agree and equal the limit of the normalized classical entropies of $\mu_N$.
\end{abstract}

\maketitle

\section{Introduction}

\subsection{Motivation and Main Ideas}

Since Voiculescu introduced free entropy of a non-commutative law in \cite{VoiculescuFE1,VoiculescuFE2,VoiculescuFE5}, a number of open problems have prevented a satisfying unification of the theory (as explained in \cite{Voiculescu2002}).  The free entropy $\chi$ was defined by taking the $\limsup$ as $N \to \infty$ of the normalized log volume of the space of microstates, where the microstates are certain tuples of $N \times N$ self-adjoint matrices having approximately the correct distribution.  It is unclear whether using the $\liminf$ instead of $\limsup$ would yield the same quantity.  Voiculescu also defined a non-microstates free entropy $\chi^*$ by integrating the free Fisher information of $X + t^{1/2} S$ where $S$ is a free semicirulcar family free from $X$, and conjectured that $\chi = \chi^*$.

Biane, Capitaine, and Guionnet \cite{BCG2003} showed that $\chi \leq \chi^*$ as a consequence of their large deviation principle for the GUE (see also \cite{CDG2001}).  The proof relied on stochastic differential equations relative to Hermitian Brownian motion and analyzed exponential functionals of Brownian motion.  Recent work of Dabrowski \cite{Dabrowski2017} combined these ideas with stochastic control theory and ultraproduct analysis in order to show that $\chi = \chi^*$ for free Gibbs states defined by a convex and sufficiently regular potential.  This resolves this part of the unification problem for a significant class of non-commutative laws.

This paper will prove a similar result to Dabrowski's using deterministic rather than stochastic methods.  We want to argue as directly as possible that the classical entropy and Fisher's information of a sequence of random matrix models converge to their free counterparts.  Let us motivate and sketch the main ideas, beginning with the heuristics behind Voiculescu's non-microstates entropy $\chi^*$.

Consider a non-commutative law $\lambda$ of an $m$-tuple and suppose $\lambda$ is the limit of a sequence of random $N \times N$ matrix distributions $\mu_N$ given by convex, semi-concave potentials $V_N: M_N(\C)_{sa}^m \to \R$.  Let $\sigma_{t,N}$ be the distribution of $m$ independent GUE matrices which each have normalized variance $t$, and let $\sigma_t$ be the non-commutative law of $m$ free semicircular variables which each have variance $t$.  Let $V_{N,t}$ be the potential corresponding to the convolution $\mu_N * \sigma_{t,N}$.  The classical Fisher information $\mathcal{I}$ satisfies
\[
\frac{d}{dt} \frac{1}{N^2} h(\mu_N * \sigma_{t,N}) = \frac{1}{N^3} \mathcal{I}(\mu_N * \sigma_{t,N}) = \int \norm{DV_{N,t}(x)}_2^2\,d(\mu_N * \sigma_{t,N})(x),
\]
and from this we deduce that
\[
\frac{1}{N^2} h(\mu_N) + \frac{m}{2} \log N = \frac{1}{2} \int \left( \frac{m}{1 + t} - \frac{1}{N^3} \mathcal{I}(\mu_N * \sigma_{t,N}) \right)\,dt + \frac{m}{2} \log 2 \pi e.
\]
As $N \to \infty$, we expect the left hand side to converge to the microstates free entropy $\chi(\lambda)$ because the distribution $\mu_N$ should be concentrated on the microstate spaces of the law $\lambda$.  On the other hand, we expect the right hand side to converge to the Voiculescu's non-microstates free entropy $\chi^*(\lambda)$ defined by
\[
\chi^*(\lambda) = \frac{1}{2} \int \left( \frac{m}{1 + t} - \Phi^*(\lambda \boxplus \sigma_t) \right)\,dt + \frac{m}{2} \log 2 \pi e,
\]
where $\Phi^*$ is the free Fisher information and $\lambda \boxplus \sigma_t$ is the free convolution \cite{VoiculescuFE5}.

Under suitable assumptions on $V_N$, the microstates free entropy $\chi(\lambda)$ is the $\limsup$ of normalized classical entropies of $\mu_N$.  On the right hand side, we want to show that $N^{-3} \mathcal{I}(\mu_N * \sigma_{t,N}) \to \Phi^*(\lambda \boxplus \sigma_t)$ for all $t \geq 0$.  Since the Fisher information is the $L^2(\mu_N)$ norm squared of the score function or (classical) conjugate variable $DV_{N,t}(x)$, we want to prove that the classical conjugate variables $DV_{N,t}(x)$ behave asymptotically like the free conjugate variables for $\lambda \boxplus \sigma_t$ for all $t$.

This would not be surprising because classical objects associated to invariant random matrix ensembles often behave asymptotically like their free counterparts.  For instance, Biane showed that the entrywise Segal-Bargman transform of non-commutative functions evaluated on $N \times N$ matrices can be approximated by the free Segal-Bargman transform computed through analytic functional calculus \cite{Biane1997}.  Similarly, Guionnet and Shlyakhtenko showed that classical monotone transport maps for certain random matrix models approximate the free monotone transport \cite[Theorem 4.7]{GS2014}.  Moreover, Dabrowski's approach to prove $\chi = \chi^*$ involved constructing solutions to free SDE as ultraproducts of the solutions to classical SDE \cite{Dabrowski2017}.

In section \ref{subsec:defasymptoticapproximation}, we make precise the idea that a sequence of functions on $M_N(\C)_{sa}^m$ has a ``well-defined, non-commutative asymptotic behavior'' by defining \emph{asymptotic approximability by trace polynomials} (Definition \ref{def:AATP}).  We assume that $DV_N$ at time zero has the approximation property and must show that the same is true for $DV_{N,t}$ for all $t$.

First, we show that this property is preserved under several operations on sequences, including composition and convolution with the Gaussian law $\sigma_{N,t}$ (see \S \ref{subsec:defasymptoticapproximation}).  Then in \S \ref{sec:evolution} we analyze the PDE that describes the evolution of $V_{N,t}$.  We show that for all $t$ the solution can be approximated in a dimension-independent way by applying a sequence of simpler operations, each of which preserves asymptotic approximability by trace polynomials.  We conclude that if the initial data $DV_N$ is asymptotically approximable by trace polynomials, then so is $DV_{N,t}$, and hence we obtain convergence of the classical Fisher information to the free Fisher information.

This proves the equality $\chi(\lambda) = \chi^*(\lambda)$ whenever a sequence of log-concave random matrix models $\mu_N$ converges to $\lambda$ in an appropriate sense (Theorem \ref{thm:mainthm}).  Another result (Theorem \ref{thm:convergenceandconcentration}), proved by similar techniques, establishes sufficient conditions for a sequence of log-concave random matrix models $\mu_N$ to converge in moments to a non-commutative law $\lambda$, so that Theorem \ref{thm:mainthm} can be applied.  As a consequence, we show that $\chi = \chi^*$ for a class of free Gibbs states.

\subsection{Main Results} \label{subsec:mainresults}

To fix notation, let $M_N(\C)_{sa}^m$ be space of $m$-tuples $x = (x_1,\dots,x_m)$ of self-adjoint $N \times N$ matrices and let $\norm{x}_2 = (\sum_j \tau_N(x_j^2))^{1/2}$, where $\tau_N = (1/N) \Tr$.  We denote by $\norm{x}_\infty$ the maximum of the operator norms $\norm{x_j}$.  Recall that a trace polynomial $f(x_1,\dots,x_m)$ is a linear combination of terms of the form
\[
p(x) \prod_{j=1}^n \tau(p_j(x)),
\]
where $p$ and $p_j$ are non-commutative polynomials in $x_1$, \dots, $x_m$.

Consider a sequence of potentials $V_N: M_N(\C)_{sa}^m \to \R$ such that $V_N(x) - (c/2) \norm{x}_2^2$ is convex and $V_N(x) - (C/2) \norm{x}_2^2$ is concave for some $0 < c < C$.  Define the associated probability measure $\mu_N$ by
\[
d\mu_N(x) = \frac{1}{Z_N} e^{-N^2 V_N(x)}\,dx, \qquad Z_N = \int_{M_N(\C)_{sa}^m} e^{-N^2 V_N(x)}\,dx.
\]
Assume that the sequence of normalized gradients $DV_N(x) = N \nabla V_N(x)$ is asymptoticallly approximable by trace polynomials in the sense that for every $\epsilon > 0$ and $R > 0$, there exists a trace polynomial $f(x)$ such that
\[
\limsup_{N \to \infty} \sup_{\norm{x}_\infty \leq R} \norm{DV_N(x) - f(x)}_2 \leq \epsilon,
\]
where $\norm{x}_\infty$ denotes the maximum of the operator norms of the $x_j$'s.  Also, assume that $\int (x - \tau_N(x))\,d\mu_N(x)$ is bounded in operator norm as $N \to \infty$ (it will be zero if $\mu_N$ is unitarily invariant or has expectation zero).  In this case, we have the following.  (To clarify the larger picture, we include statements of the concentration estimates (1) and (3), although these are standard and not proved in this paper.)

\begin{enumerate}
	\item There exists a constant $R_0$ such that $\mu_N(\norm{x}_\infty \geq R_0 + \delta) \leq m e^{-cN \delta^2/2}$ for $\delta > 0$.
	\item There exists a non-commutative law $\lambda$ such that
	\[
	\lim_{N \to \infty} \int \tau_N(p(x))\,d\mu_N(x) = \lambda(p)
	\]
	for every non-commutative polynomial $p$.
	\item The measures $\mu_N$ exhibit exponential concentration around $\lambda$, in the sense that
	\[
	\lim_{N \to \infty} \frac{1}{N^2} \log \mu_N(\norm{x}_\infty \leq R, |\tau_N(p(x)) - \lambda(p)| \geq \delta) < 0
	\]
	for every $R > 0$ and every non-commutative polynomial $p$.
	\item The law $\lambda$ has finite free entropy and we have
	\[
	\chi(\lambda) = \underline{\chi}(\lambda) = \chi^*(\lambda) = \lim_{N \to \infty} \frac{1}{N^2} \left( h(\mu_N) + \frac{m}{2} \log N \right),
	\]
	where $\chi$ and $\underline{\chi}$ are respectively the $\limsup$ and $\liminf$ versions of microstates free entropy, $\chi^*$ is the non-microstates free entropy, and $h$ is the classical entropy.
	\item The same holds for $\mu_N * \sigma_{t,N}$ and $\lambda \boxplus \sigma_t$, where $\sigma_{t,N}$ is the law of $m$ independent GUE matrices with variance $t$ and $\sigma_t$ is the law of $m$ free semicircular variables with variance $t$.
	\item The law $\lambda$ has finite free Fisher information.  If $\mathcal{I}$ is the classical Fisher information and $\Phi^*$ is the free Fisher information, then
	\[
	\lim_{N \to \infty} \frac{1}{N^3} \mathcal{I}(\mu_N * \sigma_{t,N}) = \Phi^*(\lambda \boxplus \sigma_t).
	\]
	\item The functions $t \mapsto \frac{1}{N^3} \mathcal{I}(\mu_N * \sigma_{t,N})$ and $t \mapsto \Phi^*(\lambda \boxplus \sigma_t)$ are decreasing and Lipschitz in $t$ with the absolute value of the derivative bounded by $C^2 m(1 + Ct)^{-2}$.
\end{enumerate}

Here claims (1) through (3) come from Theorem \ref{thm:convergenceandconcentration}, which is similar to the earlier results \cite[Theorem 4.4]{GS2009}, \cite[Proposition 50 and Theorem 51]{DGS2016}, \cite[Theorem 4.4]{Dabrowski2017} plus standard results on concentration of measure.  Claims (4) through (7) come from Theorem \ref{thm:mainthm}, which is similar to \cite[Theorem A]{Dabrowski2017}.

In particular, we recover Dabrowski's result \cite[Theorem A]{Dabrowski2017} that $\chi(\lambda) = \underline{\chi}(\lambda) = \chi^*(\lambda)$ when the law $\lambda$ is a free Gibbs state given by a sufficiently regular convex non-commutative potential $V(X)$, because taking $V_N = V$ will define a sequence of random matrix models $\mu_N$ which concentrate around the non-commutative law $\lambda$.

Unlike Dabrowski, we do not provide an explicit formula for $(d/dt) \Phi(\lambda \boxplus \sigma_t)$.  However, we are able to prove that $\Phi(\lambda \boxplus \sigma_t)$ is Lipschitz in $t$ rather than merely having a derivative in $L^2(dt)$ (and hence being H\"older $1/2$ continuous) as shown by Dabrowski.  Our results also allow slightly more flexibility in the choice of random matrix models, so that we do not have to assume that $V_N$ is given by exactly the same formula for every $N$ or that $V_N$ is exactly unitarily invariant.

\subsection{Organization of Paper}

Section \ref{sec:preliminaries} establishes notation and reviews basic facts from non-commutative probability and random matrix theory.

Section \ref{sec:tracepolynomials} defines the algebra of trace polynomials and describes how they behave under differentiation and convolution with Gaussians.  We then introduce the notion that a sequence $\{\phi_N\}$ of functions $M_N(\C)_{sa}^m \to M_N(\C)_{sa}^m$ or $\C$ is asymptotically approximable by trace polynomials.  We show that this approximation property is preserved under several operations including composition and Gaussian convolution.

Section \ref{sec:convergenceandconcentration} proves Theorem \ref{thm:convergenceandconcentration} concerning the convergence of moments for the measure $\mu_N$  (claims (1) - (3) of \S \ref{subsec:mainresults}).  We evaluate $\int u\,d\mu_N$ for a Lipschitz function $u$ as $\lim_{t \to \infty} T_t^{V_N}u$, where $T_t^{V_N}$ is the semigroup such that $u_t = T_t^{V_N} u$ solves the equation $\partial_t u_t = (1/2N) \Delta u_t - (N/2) \nabla V \cdot \nabla u_t$.  We approximate $T_t^{V_N}$ by iterating simpler operations in order to show that if $N \nabla V_n$ and $u_N$ are asymptotically approximable by trace polynomials, then so is $T_t^{V_N} u_N$, and hence that $\lim_{N \to \infty} \int u_N \,d\mu_N$ exists.

Section \ref{sec:entropy} reviews the defintions of free entropy and Fisher's information.  We also show that the microstates free entropies $\chi(\lambda)$ and $\underline{\chi}(\lambda)$ are the $\limsup$ and $\liminf$ of normalized classical entropies of $\mu_N$, provided that $\mu_N$ concentrates around $\lambda$ and satisfies some mild operator norm tail bounds, and that $\{V_N\}$ is asymptotically approximable by trace polynomials.  Similarly, if $\{N \nabla V_N\}$ is asymptotically approximable by trace polynomials, then the normalized classical Fisher information converges to the free Fisher information.

Section \ref{sec:evolution} considers the evolution of the potential $V_N(x,t)$ corresponding to $\mu_N * \sigma_{t,N}$, where $\sigma_{t,N}$ is the law of $m$ independent GUE of variance $t$.  Our goal is to show that if $N \nabla V_N(x,0)$ is asymptotically approximable by trace polynomials, then so is $N \nabla V_N(x,t)$ for all $t > 0$, so that we can apply our previous result that the classical Fisher information converges to the free Fisher information.  As in \S \ref{sec:convergenceandconcentration}, we construct the semigroup $R_t$ which solves the PDE as a limit of iterates of simpler operations which are known to preserve asymptotic approximation by trace polynomials.

Section \ref{sec:mainthm} concludes the proof of our main theorem on free entropy and Fisher's information (Theorem \ref{thm:mainthm}), which establishes claims (4) - (7) of \S \ref{subsec:mainresults}, assuming a weakened version of the hypothesis and conclusion of Theorem \ref{thm:convergenceandconcentration}.

In section 8, we characterize the limiting laws $\lambda$ which arise in Theorem \ref{thm:convergenceandconcentration} as the free Gibbs states for a certain class of potentials.  In particular, we apply Theorem \ref{thm:mainthm} to show that $\chi = \chi^*$ for several types of free Gibbs states considered in previous literature.

\subsection{Acknowledgements}

I thank Timothy Austin, Guillaume C{\'e}bron, Yoann Dabrowski, Alice Guionnet, Benjamin Hayes, Dimitri Shlyakhtenko, Terence Tao, Yoshimichi Ueda, and Dan Voiculescu for various useful conversations.  I especially thank Shlyakhtenko for his mentorship and ongoing conversations about free entropy, and Dabrowski for detailed discussions of his own results and other recent literature.  I acknowledge the support of the NSF grants DMS-1344970 and DMS-1500035.  Part of this research was conducted at the Institute for Pure and Applied Mathematics (IPAM) during the long program on Quantitative Linear Algebra in Spring 2018.  IPAM provided hospitality and a stimulating research environment where many of the conversations mentioned above took place.  I also thank the referee for detailed feedback that improved the clarity and correctness of the paper throughout.

\section{Preliminaries} \label{sec:preliminaries}

The first subsection \ref{subsec:matrixnotation} fixes certain notations which will be used throughout the paper.  The other subsections of \S \ref{sec:preliminaries} discuss background that they reader may refer to as needed.

\subsection{Notation for Matrix Algebras} \label{subsec:matrixnotation}

Let $M_N(\C)$ denote the $N \times N$ matrices over $\C$ and let $M_N(\C)_{sa}$ be the self-adjoint elements.  Note that $M_N(\C)_{sa}^m$ is a real inner product space with the inner product $\ip{x,y}_{\Tr} := \sum_{j=1}^m \Tr(x_jy_j)$ for $x = (x_1,\dots,x_m)$ and $y = (y_1,\dots,y_m)$.  Moreover, $M_N(\C)$ can be canonically identified with the complexification $\C \otimes_{\R} M_N(\C)_{sa}$ by decomposing a matrix into its self-adjoint and anti-self-adjoint parts.

Being a real-inner product space, $M_N(\C)_{sa}$ is isomorphic to $\R^{mN^2}$.  An explicit choice of coordinates can be made using the following orthonormal basis for $M_N(\C)_{sa}$:
\begin{equation} \label{eq:basis}
\mathcal{B}_N = \{E_{k,k}\}_{k=1}^N \cup \left\{\frac{1}{\sqrt{2}} E_{k,\ell} + \frac{1}{\sqrt{2}} E_{\ell,k} \right\}_{k < \ell} \cup \left\{\frac{i}{\sqrt{2}} E_{k,\ell} - \frac{i}{\sqrt{2}} E_{\ell,k} \right\}_{k < \ell}.
\end{equation}
This basis has the property that for any $x, y, z \in M_N(\C)$, we have
\begin{equation} \label{eq:basissummation}
\sum_{b \in \mathcal{B}_N} xbybz = xz \Tr(y),
\end{equation}
which follows from an elementary computation.

We denote the norm corresponding to $\Tr$ by $| \cdot |$ (essentially the Euclidean norm).  We denote the normalized trace by $\tau_N = \frac{1}{N} \Tr$.  We denote the corresponding inner product by $\ip{x,y}_2 = \sum_{j=1}^m \tau_N(x_jy_j)$ and the norm by $\norm{\cdot}_2$.  For $x \in M_N(\C)$, we denote the operator norm by $\norm{x}$.  Similarly, if $x = (x_1,\dots,x_m) \in M_N(\C)^m$, we denote $\norm{x}_\infty = \max_j \norm{x_j}$.

The symbols $\nabla$ and $\Delta$ will respresent the gradient and Laplacian operators with respect to the coordinates of $M_N(\C)_{sa}$ in the non-normalized inner product $\ip{\cdot,\cdot}_{\Tr}$.  The symbols $D$ and $L_N$ will denote the normalized versions $N \nabla$ and $(1/N) \Delta$ respectively, as well as the corresponding linear transformations on the algebra of trace polynomials.  This normalization and notation will be explained and justified in \S \ref{subsec:TPdiff}.

\subsection{Non-Commutative Probability Spaces and Laws}

The following are standard definitions and facts in non-commutative probability.  For further background, see \cite{VDN1992,NS2006}.

\begin{definition}
A \emph{von Neumann algebra} is a unital $\C$-algebra $\mathcal{M}$ of bounded operators on a Hilbert space $\mathcal{H}$ which is closed under adjoints and closed in the weak operator topology.
\end{definition}

\begin{definition}
A \emph{tracial von Neumann algebra} or \emph{non-commutative probability space} is a von Neumann algebra $\mathcal{M}$ together with a bounded linear map $\tau: \mathcal{M} \to \C$ which is continuous in the weak operator topology and satisfies $\tau(1) = 1$, $\tau(xy) = \tau(yx)$, and $\tau(x^*x) \geq 0$.  The map $\tau$ is called a \emph{trace}.
\end{definition}

\begin{definition}
For $m \geq 1$, we denote by $\NCP_m = \C\ip{X_1,\dots,X_m}$ the algebra of non-commutative polynomials in $X_1$, \dots, $X_m$.  A \emph{non-commutative law} (for an $m$-tuple) is a map $\lambda: \NCP_m \to \C$ such that
\begin{enumerate}
	\item $\lambda$ is linear,
	\item $\lambda$ is unital (that is, $\lambda(1) = 1$),
	\item $\lambda$ is completely positive, that is, for any matrix $P$ with entries in $\C\ip{X_1,\dots,X_m}$, the matrix $\lambda(P^*P)$ is positive semi-definite.
	\item $\lambda$ is tracial, that is, $\lambda(p(X) q(X)) = \lambda(q(X) p(X))$.
\end{enumerate}
We denote by $\Sigma_m$ the space of non-commutative laws equipped with the topology of pointwise convergence on $\C\ip{X_1,\dots,X_m}$, that is, convergence in non-commutative moments.
\end{definition}

\begin{definition}
We say that a non-commutative law $\lambda$ is \emph{bounded by $R$} if we have
\[
|\lambda(X_{i_1}, \dots, X_{i_n})| \leq R^n.
\]
We denote the space of such laws by $\Sigma_{m,R}$.
\end{definition}

\begin{definition} \label{def:lawofX}
Suppose that $x_1$, \dots $x_m$ are bounded self-adjoint elements of a tracial von Neumann algebra $(\mathcal{M},\tau)$.  Then the \emph{law of $x = (x_1,\dots,x_m)$} is the map
\[
\lambda_x: \C\ip{X_1,\dots,X_n} \to \C: p(X) \mapsto \tau(p(x)).
\]
\end{definition}

\begin{definition}
Let $M_N(\C)$ be the algebra of $N \times N$ matrices over $\C$.  Let $\tau_N = \frac{1}{N} \Tr$ be the normalized trace.  Then $(M_N(\C), \tau_N)$ is a tracial von Neumann algebra, and hence for any tuple of self-adjoint matrices $x = (x_1,\dots,x_m)$, the law $\lambda_x$ is defined by Definition \ref{def:lawofX}.
\end{definition}

\begin{proposition}
The space $\Sigma_{m,R}$ is compact, separable, and metrizable.  Moreover, every $\mu \in \Sigma_{m,R}$ can be realized as $\lambda_x$ for some tuple $x = (x_1,\dots,x_m)$ of self-adjoint elements of a tracial von Neumann algebra $(\mathcal{M},\tau)$ with $\norm{x}_\infty \leq R$.
\end{proposition}

\subsection{Non-commutative $L^\alpha$-norms} \label{subsec:NCLP}

On several occasions, we will need to use the non-commutative $L^\alpha$ norms for $\alpha \in [1,+\infty]$.  (Here we use $\alpha$ rather than $p$ since the letter $p$ will often be used for a polynomial.)  If $y$ is any element of a tracial von Neumann algebra $(\mathcal{M},\tau)$, then we define $|y| = (y^*y)^{1/2}$ defined using continuous functional calculus.  For $\alpha \in [0,+\infty)$, we define $\norm{y}_\alpha = \tau(|y|^\alpha)^{1/\alpha}$.  We also define $\norm{y}_\infty$ to be the operator norm.

\begin{proposition} \label{prop:NCHolder}
If $(\mathcal{M},\tau)$ is a tracial von Neumann algebra and $\alpha \in [1,+\infty]$, then $\norm{\cdot}_\alpha$ defines a norm.  Moreover, we have the non-commutative H\"older inequality
\[
\norm{x_1 \dots x_n}_\alpha \leq \norm{x_1}_{\alpha_1} \dots \norm{x_n}_{\alpha_n}
\]
whenever
\[
\frac{1}{\alpha_1} + \dots + \frac{1}{\alpha_n} = \frac{1}{\alpha}.
\]
Moreover, we have $|\tau(y)| \leq \norm{y}_1$.
\end{proposition}

A standard proof of the H\"older inequality uses polar decomposition, complex interpolation, and the three lines lemma.  We will in fact only need this inequality for the trace $\tau_N$ on $M_N(\C)$.   Modulo renormalization of the trace, the inequality for matrices follows from the treatment of trace-class operators in \cite{Simon2005}; see especially Thm.\ 1.15 and Thm.\ 2.8, as well as the references cited on p.\ 31.  For the setting of von Neumann algebras, a convenient proof can be found in \cite[Thm.\ 2.4 - 2.6]{daSilva2018}; for an overview and further history see \cite[\S 2]{PX2003}.

\begin{remark}
One can define the non-commutative $L^\alpha$ norm for a tuple $(y_1,\dots,y_m)$ as
\[
\norm{(y_1,\dots,y_m)}_\alpha = \begin{cases} \tau(|y_1|^\alpha + \dots + |y_m|^\alpha)^{1/\alpha}, & \alpha \in [1,+\infty) \\ \max_j \norm{y_j}, & \alpha = +\infty. \end{cases}
\]
However, for tuples, we will only need to use the $2$ and $\infty$ norms.
\end{remark}

\subsection{Free Independence, Semicircular Law, and GUE}

We will use the following standard definitions and facts from free probability.  For further background, refer to \cite{Voiculescu1986,Voiculescu1991,VDN1992,NS2006,AGZ2009}.

Let $(\mathcal{M},\tau)$ be a tracial von Neumann algebra, and let $A_1$, \dots, $A_n$ be unital $*$-subalgebras of $\mathcal{M}$.  Then we say that $A_1$, \dots, $A_m$ are \emph{freely independent} if given $a_1$, \dots, $a_k$ with $a_j \in A_{i_j}$ and $i_j \neq i_{j+1}$ and $\tau(a_j) = 0$ for each $j$, we have also $\tau(a_1 \dots a_k) = 0$.

In particular, if $S_1$, \dots, $S_n$ are \emph{subsets} of $\mathcal{M}$, then we say that they are freely independent if the unit $*$-subalgebras they generate are freely independent.  Thus, for instance, self-adjoint elements $x_1$, \dots, $x_m$ of $\mathcal{M}$ are freely independent if given polynomials $f_1$, \dots, $f_k$ and indices $i_1$, \dots, $i_k$ with $i_j \neq i_{j+1}$ such that $\tau(f_j(X_{i_j})) = 0$, we have also $\tau(f_1(X_{i_1}) \dots f_k(X_{i_k})) = 0$.

The \emph{free convolution} of two non-commutative laws $\mu$ and $\nu$ (of self-adjoint $m$-tuples) is defined as the non-commutative law of $(x_1+y_1,\dots,x_m+y_m)$, given that $\{x_1,\dots,x_m\}$ and $\{y_1,\dots,y_m\}$ are freely independent and the non-commutative law of $(x_1,\dots,x_m)$ is $\mu$, and the non-commutative law of $(y_1,\dots,y_m)$ is $\nu$.  Then $\boxplus$ is well-defined, independent of the particular choice of operators that realize the laws $\mu$ and $\nu$.  Moreover, $\boxplus$ commutative and associative.

If $X_1$, \dots, $X_m$ are freely independent, then their joint law is determined by the individual laws of the $X_j$'s, each of which is represented by a compactly supported probability measure on $\R$.  The \emph{semicircle law} (of mean zero and variance $1$) is the probability measure given by density $(1/2\pi) \sqrt{4 - x^2} \mathbf{1}_{[-2,2]}(x)\,dx$.  We denote by $\sigma_t$ the non-commutative law of $m$ freely independent semicircular random variables which each have mean zero and variance $t$ (that is, $\sigma_t(X_j) = 0$ and $\sigma_t(X_j^2) = t$).

These free semicircular families play the role of multivariable Gaussians in free probability.  Moreover, they form a semigroup under free convolution, that is, $\sigma_s \boxplus \sigma_t = \sigma_{s+t}$.

We denote by $\sigma_{t,N}$ the probability distribution on $M_N(\C)_{sa}^m$ for $m$ independent GUE matrices of normalized variance $t$, that is,
\[
d\sigma_{t,N}(x) = \frac{1}{Z_{N,t}} \exp \left(-N \sum_{j=1}^m \Tr(x_j^2) / 2t \right)\,dx,
\]
where $Z_{N,t}$ is chosen so that $\sigma_{t,N}$ is a probability measure.  It is well known that the independent GUE matrices behave in the limit like freely independent semicircular random variables, although we shall directly prove the specific properties we use in \S \ref{sec:tracepolynomials}.

\subsection{Concentration and Operator Norm Tail Bounds} \label{subsec:concentration}

The following is a standard concentration estimate for uniformly log-concave random matrix models.  The best known proof goes through the log-Sobolev inequality and Herbst's argument (see \cite[\S 4.4.2]{AGZ2009}), although it can also be proved directly using the heat semigroup directly as in \cite{Ledoux1992}.  We state the theorem here with free probabilistic normalizations.

\begin{theorem} \label{thm:concentration}
Suppose that $V: M_N(\C)_{sa}^m \to \R$ is a potential such that $V(x) - (c/2) \norm{x}_2^2$ is convex.  Define
\[
d\mu(x) = \frac{1}{Z} \exp(-N^2 V(x))\,dx, \qquad Z = \int \exp(-N^2 V(x))\,dx.
\]
Suppose that $f: M_N(\C)_{sa}^m \to \R$ is $K$-Lipschitz with respect to $\norm{\cdot}_2$.  Then
\[
\mu(x: f(x) - \smallint f\,d\mu \geq \delta) \leq e^{cN^2 \delta^2 / 2K^2},
\]
and since the same estimate can be applied to $-f$, we have also
\[
\mu(x: |f(x) - \smallint f \,d\mu| \geq \delta) \leq 2 e^{-cN^2 \delta^2 / 2K^2}.
\]
\end{theorem}

In particular, this concentration estimate applies to the GUE law $\sigma_{t,N}$ with $c = 1/t$.  In addition to the concentration estimate, we will also use the fact that such uniformly convex random matrix models have subgaussian moments and therefore have good tail bounds on the probability of large operator norm.  The following theorem is a special case of \cite[Theorem 1.1]{Harge2004} and the application to random matrix models is taken from the proof of \cite[Theorem 3.4]{GMS2006}.

\begin{theorem} \label{thm:subgaussian}
Let $V$ and $\mu$ be as in Theorem \ref{thm:concentration}, and suppose that $f: M_N(\C)_{sa}^m \to \R$ is convex.  Let $a = \int x\,d\mu(x)$.  Then
\[
\int f(x - a)\,d\mu(x) \leq \int f(y)\,d\sigma_{c^{-1},N}(y).
\]
In particular, if $\norm{x}_\alpha$ denotes the $L^\alpha$ norm from \S \ref{subsec:NCLP}, then for every $\alpha \in [1,+\infty]$ and $\beta \in [1,+\infty)$, we have
\[
\int \norm{x_j - a_j}_\alpha^\beta\,d\mu(x) \leq \int \norm{y_j}_\alpha^\beta\,d\sigma_{c^{-1},N}(y).
\]
\end{theorem}

\begin{proof}
The convexity assumption on $V$ means that $\mu$ has a log-concave density with respect to the Gaussian measure $\sigma_{c^{-1},N}(y)$.  Therefore, the first claim follows from Harg\'e's theorem \cite[Theorem 1.1]{Harge2004}.  The second claim follows because norms on vector spaces are convex functions, and the function $t \mapsto t^\beta$ is convex for $\beta \geq 1$.
\end{proof}

\begin{corollary} \label{cor:operatornormtail}
Let $V_N: M_N(\C)_{sa} \to \R$ be a function such that $V_N(x) - (c/2) \norm{x}_2^2$ is convex and let $\mu_N$ be the corresponding measure.  Let $a_{N,j} = \int x_j\,d\mu_N(x)$.  Then
\[
\limsup_{N \to \infty} \int \norm{x_j - a_{N,j}} \,d\mu_N(x) \leq 2c^{-1/2},
\]
and
\[
\mu_N(x: \norm{x_j} \geq \smallint \norm{y_j}\,d\mu_N(y_j) + \delta) \leq e^{-c\delta^2 N / 2}.
\]
\end{corollary}

\begin{proof}
In light of Theorem \ref{thm:subgaussian}, for the first claim of the Corollary, it suffices to check the special case $\sigma_{c^{-1},N}$.  This special case is a standard result in random matrix theory; see for instance the proof of \cite[Theorem 2.1.22]{AGZ2009}.  The second claim follows from Theorem \ref{thm:concentration} after we observe that the function on $M_N(\C)_{sa}^N$ given by $x \mapsto \norm{x_j}_\infty$ is $N^{1/2}$-Lipschitz with respect to $\norm{\cdot}_2$.
\end{proof}

\subsection{Semi-convex and Semi-concave functions} \label{subsec:convex}

We recall the following terminology and facts about semi-convex and semi-concave functions.  These facts are typically applied to functions from $\R^n \to \R$, but of course they hold equally well if $\R^n$ is replaced by a finite-dimensional real inner product space.  In particular, we focus on the case of $M_N(\C)_{sa}^m$.

A function $u: M_N(\C)_{sa}^m \to \R$ is \emph{semi-convex} if there exists some $c \in \R$ such that $u(x) - (c/2) \norm{x}_2^2$ is convex.  If this holds for some $c > 0$, then $u$ is said to be \emph{uniformly convex}.  Similarly, $u: M_N(\C)_{sa}^m \to \R$ is said to be \emph{semi-concave} if there exists $C \in \R$ such that $u(x) - (C/2) \norm{x}_2^2$ is concave, and it is \emph{uniformly concave} if this holds for some $C < 0$.

Fix $m$ and $N$.  For $c \leq C$ be real numbers.  Then we define
\[
\mathcal{E}_{m,N}(c,C) = \{u: M_N(\C)_{sa}^m \to \R: u(x) - (c/2) \norm{x}_2^2 \text{ is convex and } u(x) - (C/2) \norm{x}_2^2 \text{ is concave}\}.
\]
We will often suppress $m$ and $N$ in the notation and simply write $\mathcal{E}(c,C)$.  Throughout the paper, we rely on the following basic properties of functions in $\mathcal{E}(c,C)$.

\begin{proposition} \label{prop:convex} ~
\begin{enumerate}
	\item The space $\mathcal{E}(c,C)$ is closed under translation, averaging with respect to probability measures, and pointwise limits.
	\item A function $u$ is in $\mathcal{E}(c,C)$ if and only if for every point $x_0 \in M_N(\C)_{sa}^m$, there exists some $p \in M_N(\C)_{sa}^m$ such that
	\[
	u(x_0) + \ip{p,x-x_0}_2 + \frac{c}{2} \norm{x-x_0}_2^2 \leq u(x) \leq u(x_0) + \ip{p,x-x_0}_2 + \frac{C}{2} \norm{x-x_0}_2^2.
	\]
	\item In particular, if $u \in \mathcal{E}(c,C)$, then $u$ is differentiable everywhere.
	\item If $u \in \mathcal{E}(c,C)$, then the gradient $Du$ is $\max(|c|,|C|)$-Lipschitz with respect to $\norm{\cdot}_2$.
	\item If $u \in \mathcal{E}(c,C)$, then
	\[
	c \norm{x - y}_2^2 \leq \ip{Du(x) - Du(y),x-y}_2 \leq C \norm{x - y}_2^2.
	\]
	\item If $u \in \mathcal{E}(c,C)$ for some $c > 0$, then $u$ is bounded below and achieves a global minimum at its unique critical point.
\end{enumerate}
\end{proposition}

\begin{proof}[Sketch of proof]
(1) One follows from elementary computation and the fact that the same holds for the class of convex functions.

(2), (3) The convex functions $u(x) - (c/2)\norm{x}_2^2$ and $(C/2) \norm{x}_2^2 - u(x)$ must have supporting hyperplanes at $x_0$.  This yields one vector $p$ which satisfies the left inequality of (2) and another vector $p'$ satisfying the right inequality.  Then one checks that $p$ must equal $p'$ and this implies that $u$ is differentiable at $x_0$.

(4), (5) One can check these inequalities for smooth functions in $\mathcal{E}_{c,C}$ directly using Taylor expansions and calculus.  Consider a general $u \in \mathcal{E}_{c,C}$.  Let $u_n = u * \rho_n$, where $\rho_n$ be a smooth probability density supported in the ball of radius $1/2$ around $0$.  Then $u_n$ is smooth and $u_n \to u$ locally uniformly.  Also, $u_n \in \mathcal{E}_{c,C}$ by (1), hence $Du_n$ is $\max(|c|,|C|)$-Lipschitz.  By the Arzela-Ascoli theorem, after passing to a subsequence, we may assume that $Du_n$ converges locally uniformly to some $F$.  It follows from local uniform convergence of $Du_n$ that $F = Du$.  Moreover, since (4) and (5) hold for $Du_n$, they also hold for $Du$.
\end{proof}

\section{Trace Polynomials} \label{sec:tracepolynomials}

In this secion, we consider the algebra of trace polynomials in non-commutative variables $X_1$, \dots, $X_m$, first defined in \cite{Razmyslov1974}, \cite{Razmyslov1987}.  As in \cite{Rains1997}, \cite{Cebron2013}, \cite{DHK2013}, we describe how trace polynomials behave under differentiation (\S \ref{subsec:TPdiff}) and convolution with Gaussian (\S \ref{subsec:TPGaussianconvolution}).  Finally, in \S \ref{subsec:defasymptoticapproximation}, we define the property of asymptotic approximability by trace polynomials for a sequence of functions on $M_N(\C)_{sa}^m$, which is one of the key technical tools in our proof.

\subsection{Definition}

\begin{definition}
We define the \emph{algebra of scalar-valued trace polynomials}, or $\TrP_m^0$, as follows.  Let $\mathcal{V}$ be the vector space $\NCP_m / \Span(pq - qp: p, q \in \NCP_m)$.  We define the vector space
\begin{equation}
\TrP_m^0 = \bigoplus_{n=0}^\infty \mathcal{V}^{\odot n},
\end{equation}
where $\odot$ is the symmetric tensor power over $\C$.  Then $\TrP_m^0$ forms a commutative algebra with the tensor operator $\odot$ as the multiplication.  We denote the element $p_1 \odot \dots \odot p_n$ by $\tau(p_1) \dots \tau(p_n)$ where $\tau$ is a formal symbol.
\end{definition}

To state the definition more suggestively, an element of $\TrP_m^0$ is a linear combination of terms of the form $\tau(p_1(X)) \dots \tau(p_n(X))$ where $p_1$, \dots, $p_n$ are non-commutative polynomials in $X_1,\dots,X_m$ and $\tau$ is a formal symbol thought of as the trace.  By forming a quotient vector space, we identify $\tau(pq)$ with $\tau(qp)$.  The trace polynomials form a commutative $*$-algebra $\TrP_m^0$ over $\C$ where the $*$-operation is
\begin{equation}
(\tau(p_1(X)) \dots \tau(p_n(X)))^* = \tau(p_1(X)^*) \dots \tau(p_n(X)^*)
\end{equation}
and the multiplication operation is the one suggested by the notation.

We define $\TrP_m^k$ to be the vector space
\[
\TrP_m^k := \TrP_m^0 \otimes \C\ip{X_1,\dots,X_m}^{\otimes k}.
\]
We call the elements of $\TrP_m^1$ \emph{operator-valued trace polynomials}.  We use the term \emph{trace polynomials} more generally to describe elements of $\TrP_m^k$ or tuples of elements from $\TrP_m^k$.  Note that $\TrP_m^1$ forms a $*$-algebra because it is the tensor products of two $*$-algebras.

\begin{definition}
Suppose that $\mathcal{M}$ is a von Neumann algebra with trace $\sigma$.  Given $f \in \TrP_m^1$ and a self-adjoint tuple $x = (x_1,\dots,x_m)$ of elements of $\mathcal{M}$, we define $f(x)$ to the element of $\mathcal{M}$ given by formally evaluating $f$ with the formal symbol $X$ repalce by $x$ and the formal symbol $\tau$ replaced with $\sigma$.  For instance if $f(X) = p_0(X) \otimes \tau(p_1(X)) \dots \tau(p_n(X))$ in $\TrP_m^1$, then
\[
f(x) = p_0(x) \sigma(p_1(x)) \dots \sigma(p_n(x)).
\]
In particular, we define $f(x)$ when $x$ is an $m$-tuple of self-adjoint $N \times N$ matrices by setting $\tau = \tau_N$.
\end{definition}

\begin{definition}
If $f \in \TrP_m^0$ and $\lambda$ is a non-commutative law, we define the evaluation $\lambda(f)$ to be the number obtained by replacing the symbol $\tau$ with $\lambda$ everywhere in $f$.  For example, if $f(X_1,X_2,X_3) = \tau(X_1) \tau(X_2 X_3) + \tau(X_2^2)$, then we define
\[
\lambda(f) = \lambda(f(X_1,\dots,X_m)) = \lambda(X_1) \lambda(X_2 X_3) + \lambda(X_2^2).
\]
\end{definition}

\begin{definition}
We define the \emph{degree} for elements of $\NCP_m$ and $\TrP_m^k$ as follows.  If $p \in \NCP_m$ is a monomial $p(X_1,\dots,X_m) = X_{i_1} \dots X_{i_d}$, then we define $\deg'(p) = d$.  If $p_1$, \dots, $p_\ell$ and $q_1$, \dots, $q_k$ are non-commutative monomials, then consider the element $\tau(p_1) \dots \tau(p_\ell) q_1 \otimes \dots \otimes q_k \in \TrP_m^k$, and define
\[
\deg'(\tau(p_1) \dots \tau(p_\ell) q_1 \otimes \dots \otimes q_k) = \deg'(p_1) \dots \deg'(p_\ell) \deg'(q_1) \dots \deg'(q_k).
\]
For general $f \in \TrP_m^k$, we define the \emph{degree} $\deg(f)$, as the infimum of $\max(\deg'(f_1),\dots,\deg'(f_\ell))$, where $f = f_1 + \dots + f_\ell$ and each $f_j$ is a product of non-commutative monomials and traces of non-commutative monomials as above.  Similarly, for general $f \in \NCP_m$, we define $\deg(f)$ as the infimum of $\max(\deg'(f_1),\dots,\deg'(f_\ell))$ where $f = f_1 + \dots + f_\ell$ and each $f_j$ is a non-commutative monomial.
\end{definition}

\begin{remark}
One can check that if $f$ is a product of monomials as above, then $\deg(f) = \deg'(f)$.  Moreover, the degree makes $\TrP_m^0$ and $\TrP_m^1$ into graded algebras.  Finally, we observe that $f \in \TrP_m^0$ or $\TrP_m^1$, then the function on $M_N(\C)_{sa}^m$ defined by $x \mapsto f(x)$ is a polynomial in the entries of $x_1$,\dots, $x_m$, and the degree of $x \mapsto f(x)$ with respect to the entries is bounded above by the degree of $f$ in $\TrP_m^0$ or $\TrP_m^1$.  None of these facts will be used in what follows, so we omit the proofs.
\end{remark}

We also observe that there is a composition operation $(\TrP_m^1)^m \times (\TrP_m^1)^m \to (\TrP_m^1)^m$ defined just as one would expect from manipulations in $M_N(\C)$.  If $f, g \in (\TrP_m^1)^m$, we define $f(g(x))$ by substituting $g_j(x)$ as the $j$-th argument of $f$.  Then we multiply elements out by treating the terms of the form $\tau(p)$ like scalars.  For instance, if $f(Y_1,Y_2) = (\tau(Y_1Y_2) Y_2, Y_1 + \tau(Y_1^2) Y_2)$ and $g(X_1,X_2) = (\tau(X_1) X_2 + X_1, X_1)$, then $f \circ g(X_1,X_2) = (Z_1,Z_2)$, where
\[
Z_1 =  (\tau[\tau(X_1) X_2 + X_1]X_1) X_1 = \tau(X_1) \tau(X_2 X_1) X_1 + \tau(X_1^2) X_1
\]
and
\begin{align*}
Z_2 &= \tau(X_1) X_2 + X_1 + X_1 \tau[(\tau(X_1)X_2 + X_1)^2] \\
&= \tau(X_1) X_2 + X_1 + \tau[\tau(X_1)^2 X_2^2 + \tau(X_1) X_2 X_1 + \tau(X_1) X_1 X_2 + X_1^2]X_1 \\
&= \tau(X_1) X_2 + X_1 + [\tau(X_1)^2 \tau(X_2^2) + \tau(X_1) \tau(X_2 X_1) + \tau(X_1) \tau(X_1 X_2) + \tau(X_1^2)] X_1 \\
&= \tau(X_1) X_2 + X_1 + \tau(X_1)^2 \tau(X_2^2) X_1 + 2 \tau(X_1) \tau(X_2 X_1) X_1 + \tau(X_1^2) X_1. 
\end{align*}
One can check that composition on $(\TrP_m^1)^m$ is well-defined.  Moreover, if $f$ and $g$ are self-adjoint elements of $(\TrP_m^1)^m$, then they define functions $M_N(\C)_{sa}^m \to M_N(\C)_{sa}^m$, and the element $f \circ g \in (\TrP_m^1)^m$ defined abstractly will product a function $M_N(\C)_{sa}^m \to M_N(\C)_{sa}^m$ which is the composition of the corresponding functions for $f$ and $g$.

\subsection{Differentiation of Trace Polynomials} \label{subsec:TPdiff}

In this section, we give explicit formulas for the gradient and Laplacian of trace polynomials and in particular show that these operations have a well-defined limit as $N \to \infty$ (see \cite{Rains1997}, \cite{Cebron2013}, \cite[\S 3]{DHK2013}).  We first recall the free difference quotients of Voiculescu \cite{VoiculescuFE5}.

\begin{definition} \label{def:NCderivative}
We define the \emph{free difference quotient} (or simply \emph{non-commutative derivative}) $\mathcal{D}_j: \NCP_m \to \NCP_m \otimes \NCP_m$ by
\[
\mathcal{D}_j[X_{i_1} \dots X_{i_n}] = \sum_{k: i_k = j} X_{i_1} \dots X_{i_{k-1}} \otimes X_{i_{k+1}} \dots X_{i_n}.
\]
We also define $\mathcal{D}_j: \NCP_m^{\otimes n} \to \NCP_m^{\otimes n+1}$ by
\[
\mathcal{D}_j[p_1 \otimes \dots \otimes p_n] = \sum_{k=1}^n p_1 \otimes \dots \otimes p_{k-1} \otimes \mathcal{D}_j p_k \otimes p_{k+1} \otimes \dots \otimes p_n.
\]
Then of course $\mathcal{D}_j^k$ is a well-defined map $\NCP_m^{\otimes n} \to \NCP_m^{\otimes n + k}$.
\end{definition}

\begin{remark}
We caution the reader that the normalization for $\mathcal{D}_j^n f$ here differs from that of \cite{VoiculescuFE5} by a factor of $n!$.
\end{remark}

\begin{definition}
We define the \emph{cyclic derivative} $\mathcal{D}_j^\circ: \NCP_m \to \NCP_m$ as the linear map given by
\[
\mathcal{D}_j^\circ[X_{i_1} \dots X_{i_n}] = \sum_{k: i_k = j} X_{i_{k+1}} \dots X_{i_n} X_{i_1} \dots X_{i_{k-1}}.
\]
\end{definition}

\begin{definition}
Given an algebra $\mathcal{A}$ (e.g. $\NCP_m$), we define the \emph{hash operation} as the bilinear map $(\mathcal{A} \otimes \mathcal{A}) \times \mathcal{A}$ given by $(a_1 \otimes a_2) \# b = a_1 b a_2$.
\end{definition}

\begin{example}
Let $X = (X_1,X_2,X_3)$ and define $f(X) = X_1 X_2 X_1^2 X_3 X_2$.  Then
\begin{align*}
\mathcal{D}_1 f(X) &= 1 \otimes X_2 X_1^2 X_3 X_2 + X_1 X_2 \otimes X_1 X_3 X_2 + X_1 X_2 X_1 \otimes X_3 X_2 \\
\mathcal{D}_1^\circ f(X) &= X_2 X_1^2 X_3 X_2 + X_1 X_3 X_2 X_1 X_2 + X_3 X_2 X_1 X_2 X_1, \\
\mathcal{D}_1 f(X) \# Y &= YX_2 X_1^2 X_3 X_2 + X_1 X_2  Y X_1 X_3 X_2 + X_1 X_2 X_1 Y X_3 X_2.
\end{align*}
To compute $\mathcal{D}_1^2 f(X) = \mathcal{D}_1[\mathcal{D}_1 f(X)]$, we would add together the three terms
\begin{align*}
\mathcal{D}_1[1 \otimes X_2 X_1^2 X_3 X_2] &= 1 \otimes X_2 \otimes X_1 X_3 X_2 + 1 \otimes X_2 X_1 \times X_3 X_2 \\
\mathcal{D}_1[X_1 X_2 \otimes X_1 X_3 X_2] &= 1 \otimes X_2 \otimes X_1 X_3 X_2 + X_1 X_2 \otimes 1 \otimes X_3 X_2 \\
\mathcal{D}_1[X_1 X_2 X_1 \otimes X_3 X_2] &= 1 \otimes X_2 X_1 \otimes X_3 X_2 + X_1 X_2 \otimes 1 \otimes X_3 X_2.
\end{align*}
\end{example}

Now we can define derivatives for scalar-valued and non-commutative trace polynomials that correspond with differentiation with respect to the standard coordinates on $M_N(\C)_{sa}^m$.  We begin with the gradient.

To fix notation, recall that in \S \ref{subsec:matrixnotation} we gave a canonical orthonormal basis for $M_N(\C)_{sa}$ with respect to the inner product $\ip{x,y} = \Tr(x^*y)$.  Using these coordinates, we may identify $M_N(\C)_{sa}$ with $\R^{N^2}$ and hence identify $M_N(\C)_{sa}^m$ with $\R^{mN^2}$.  Similarly, we identify the complexification $\C \otimes M_N(\C)_{sa}^m$ with $M_N(\C)^m$ and with $\C^{mN^2}$.  For $f: M_N(\C)_{sa}^m \to \C$ and $x = (x_1,\dots,x_m) \in M_N(\C)_{sa}^m$, we denote by $\nabla f(x) \in M_N(\C)^m$ the gradient computed in these coordinates; similarly, we denote by $\nabla_j f(x) \in M_N(\C)$ the gradient with respect to $x_j$ computed in these coordinates.

\begin{definition}
Define the \emph{$j$th gradient operator} $\TrP_m^0 \to \TrP_m^1$ by
\begin{equation}
D_j \left[ \prod_{k=1}^n \tau(p_k) \right] = \sum_{k=1}^n \mathcal{D}_j^\circ p_k \prod_{\ell \neq k} \tau(p_\ell).
\end{equation}
Note that $D_j$ is defined so as to obey the Leibniz rule (that is, it is a derivation).
\end{definition}

\begin{lemma}
If $f \in \TrP_m^0$ is viewed as a function $M_N(\C)_{sa}^m \to M_N(\C)^m = \C$, then we have
\begin{equation}
\nabla_j[f(x)] = \frac{1}{N} [D_jf](x).  \label{eq:evaluategradient}
\end{equation}
Similarly, for $F: M_N(\C)_{sa}^m \to M_N(\C)^m$, let $J_jF$ denote the Jacobian linear transformation (a.k.a.\ Fr{\'e}chet derivative) with respect to $x_j$.  Then for a non-commutative polynomial $p$, we have
\begin{equation} \label{eq:noncommutativedifferentiation}
[J_jp(x)](y) = [\mathcal{D}_jp](x) \# y,
\end{equation}
and hence by the product rule for $p \in \NCP_m$ and $f \in \TrP_m^0$, we have
\begin{equation}
[J_j(pf)(x)](y) = ([\mathcal{D}_j p](x) \# y) f(x) + p(x) \tau_N([D_jf](x) y). \label{eq:evaluateJacobian}
\end{equation}
\end{lemma}

\begin{proof}
By standard computations, for a non-commutative polynomial $p$ and $y \in M_N(\C)_{sa}$, we have
\begin{align*}
[J_j p(x)](y) &= [\mathcal{D}_j p](x) \# y \\
\nabla_j[\tau_N(p)](x) &= \frac{1}{N} [\mathcal{D}_j^\circ p](x).
\end{align*}
The claims \eqref{eq:evaluategradient} and \eqref{eq:evaluateJacobian} now follow from the product rule.
\end{proof}

Next, we can define the algebraic Laplacian operators on $\TrP_m^0$ and $\TrP_m^1$, which correspond to computing the Laplacian on scalar-valued or vector-valued functions on $M_N(\C)_{sa}^m$, still using the coordinates given in \S \ref{subsec:matrixnotation}.

For $f: M_N(\C)_{sa}^m \to \C$, let $\Delta_j f$ be the Laplacian with respect to the coordinates of the $j$-th matrix $x_j$.  Note that $\Delta f = \sum_{j=1}^m \Delta_j f$.  Similarly, if $f: M_N(\C)_{sa}^m \to M_N(\C)$ is an operator-valued function, we define $\Delta_j f$ and $\Delta f$ by applying $\Delta_j$ and $\Delta$ entrywise (as is standard notation for the Laplacian of a vector-valued function).

Motivated by \eqref{eq:basissummation} and the computation in Lemma \ref{lem:evaluateLaplacian} below, we define the map $\eta: \NCP_m^{\otimes 3} \to \TrP_m^1$ by
\[
\eta(p_1 \otimes p_2 \otimes p_3) = p_1 p_3 \tau(p_2).
\]

\begin{definition} \label{def:scalarLaplacian}
We define $L_j$ and $L_{N,j}: \TrP_m^0 \to \TrP_m^0$ to be the unique linear operators such that
\begin{equation} \label{eq:scalarLaplacianspecialcase}
L_j[\tau(p)] = L_{N,j}[\tau(p)] = \tau \circ \eta[\mathcal{D}_j^2 p] \text{ for } p \in \NCP_m.
\end{equation}
and such that the following product rule is satisfied:
\begin{align}
L_j[f \cdot g] &= L_j[f] \cdot g + f \cdot L_j[g] \\
L_{N,j}[f \cdot g] &= L_{N,j}[f] \cdot g + f \cdot L_{N,j}[g] + \frac{2}{N^2} \tau(D_j f \cdot D_jg).
\end{align}
Then we define $L = \sum_{j=1}^m L_j$ and $L_N = \sum_{j=1}^m L_{N,j}$.
\end{definition}

\begin{remark}
To show the existence of operators $L_{N,j}$ and $L_j$ satisfying \eqref{eq:scalarLaplacianspecialcase} and the product rule, one can define $L_{N,j}$ more explicitly as the linear operator $\TrP_m^0 \to \TrP_m^0$ given by
\[
L_{N,j}[\tau(p_1) \dots \tau(p_n)] = \sum_{k=1}^n \tau \circ \eta[\mathcal{D}^2 p_k] \cdot \prod_{i \neq k} \tau(p_i) + \frac{1}{N^2} \sum_{k=1}^n \sum_{\ell \neq k} \tau(\mathcal{D}_j^\circ p_k \cdot \mathcal{D}_j^\circ p_\ell) \prod_{i \neq k, \ell} \tau(p_i),
\]
and check that this operator satisfies the product rule.  Moreover, the uniqueness of the operator $L_{N,j}$ satisfying \eqref{eq:scalarLaplacianspecialcase} and the product rule follows from the fact that $\TrP_m^0$ is spanned by products of terms of the form $\tau(p)$ for $p \in \NCP_m$.  The argument for the existence and uniqueness of $L_j$ is the same.
\end{remark}

\begin{example}
Let $X = (X_1,X_2)$.  Consider $f(X) = \tau(f_1(X)) \tau(f_2(X))$ where $f_1(X) = X_1 X_2 X_1 X_3$ and $f_2(X) = X_2^2 X_1$.  Then
\begin{align*}
D_1[\tau(f_1)] &= \mathcal{D}_1^\circ f_1 = X_2 X_1 X_3 + X_3 X_1 X_2 \\
D_1[\tau(f_2)] &= \mathcal{D}_2^\circ f_2 = X_2^2,
\end{align*}
and
\begin{align*}
L_1[\tau(f_1)] &= L_{N,1}[\tau(f_1)] = \tau \circ \eta[\mathcal{D}_1^2 f_1] = \tau[\eta[1 \otimes X_2 \otimes X_3]] = \tau[1 \cdot X_3] \cdot \tau[X_2] \\
L_1[\tau(f_2)] &= L_{N,1}[\tau(f_2)] = 0.
\end{align*}
Therefore, we have
\begin{align*}
L_1[f] &= L_1[\tau(f_1)] \tau(f_2) + \tau(f_1) L_1[\tau(f_2)] = \tau(X_3) \tau(X_2) \tau(X_2^2 X_1) + 0 \\
L_{N,1}[f] &= L_{N,1}[\tau(f_1)] \tau(f_2) + \tau(f_1) L_{N,1}[\tau(f_2)] + \frac{2}{N^2} \tau[ \mathcal{D}_1^\circ f_1 \mathcal{D}_1^\circ f_2] \\
&= \tau(X_3) \tau(X_2) \tau(X_2^2 X_1) + \frac{2}{N^2} \tau[(X_2 X_1 X_3 + X_3 X_1 X_2) X_2^2].
\end{align*}
One can carry out a similar computation for $L_2[f]$ and $L_{N,2}[f]$ and thus find $L[f]$ and $L_{N,2}[f]$.
\end{example}

Since we will also deal with the Laplacians of matrix-valued functions on matrices, we also need to define the algebraic Laplacian on \emph{operator-valued} trace polynomials.

\begin{definition} \label{def:matrixLaplacian}
We also define $L_j$ and $L_{N,j}: \TrP_m^1 \to \TrP_m^1$ to be the unique linear operators on the operator-valued trace polynomials such that
\begin{equation}
L_j[p] = L_{N,j}[p] = \eta[\mathcal{D}_j^2 p] \text{ for } p \in \NCP_m
\end{equation}
and the following product rule is satisfied for $p \in \NCP_m$ and $f \in \TrP_m^0$:
\begin{align}
L_j[p \cdot f] &= L_j[p] \cdot f + p \cdot L_j[f] \\
L_{N,j}[p \cdot f] &= L_{N,j}[p] \cdot f + p \cdot L_{N,j}[f] + \frac{2}{N^2} \mathcal{D}_j p \# D_jf,
\end{align}
where $L_j[f]$ and $L_{N,j}[f]$ are given by Definition \ref{def:scalarLaplacian}.  Then we define $L = \sum_{j=1}^m L_j$ and $L_N = \sum_{j=1}^m L_{N,j}$.
\end{definition}

\begin{remark}
The argument for the existence and uniqueness of the operators $L_j$ and $L_{N,j}$ on $\TrP_m^1$ is similar to the argument for $\TrP_m^0$, only that it relies on the previous scalar-valued case since the scalar-valued case was used in the product rule.
\end{remark}

\begin{lemma} \label{lem:evaluateLaplacian}
Let $f \in \TrP_m^0$.  Viewing $f$ is a function $M_N(\C)_{sa}^m \to \C$, we have
\begin{equation} \label{eq:evaluateLaplacian}
\Delta_j f(x) = N [L_{N,j} f](x) \qquad \Delta f(x) = N [L_N f](x).
\end{equation}
The same formula holds if $f \in \TrP_m^1$ and $f$ is viewed as a function $M_N(\C)_{sa}^m \to M_N(\C)$.
\end{lemma}

\begin{proof}
We begin with the special case of computing the Laplacian of $p \in \NCP_m$ (as a \emph{matrix-valued} function).  To differentiate, we use the basis $\mathcal{B}_N$ given by \eqref{eq:basis}.  Note that
\begin{align*}
\Delta_j p(x) &= \sum_{b \in \mathcal{B}_N} \frac{d^2}{dt^2} \Bigr|_{t=0} f(x_1, \dots,x_{j-1},x_j + tb, x_{j+1}, \dots,x_m) \\
&= \sum_{b \in \mathcal{B}_N} \mathcal{D}_j^2 p(x) \# (b \otimes b) \\
&= N [\eta(\mathcal{D}_j^2 p)](x), \\
&= [L_{N,j}p](x)
\end{align*}
where the second-to-last equality follows from \eqref{eq:basissummation}.

Next, we consider the case of computing the Laplacian of $\tau_N(p)$ (as a \emph{scalar-valued} function) for $p \in \NCP_m$.  Since $\tau_N$ is a linear map $M_N(\C) \to \C$, we have
\[
\Delta_j[\tau_N(p(x))] = \tau_N(\Delta_jp(x)),
\]
where the Laplacian $\Delta_j$ on the left hand side is applied to a scalar-valued function and on the right hand side it is applied to a matrix-valued function.  Therefore, it follows from the previous computation that
\[
\Delta_j[\tau_N(p(x))] = N \tau_N([\eta(\mathcal{D}_j^2 p)](x)) = [L_{N,j}[\tau(p)]](x).
\]

For the general case of scalar-valued trace polynomials, recall that the vector space of trace polynomials is spanned by elements of the form $f = \tau(p_1) \dots \tau(p_N)$ where $p_j \in \NCP_m^0$.  Let $f_j = \tau(p_j) \in \TrP_m^0$.  The Laplacian $\Delta_j$ of a product of a functions can be computed using the product rule of differentiation as
\[
\Delta_j f(x) = \sum_{k=1}^n \Delta_j[f(x)] \prod_{i \neq k} f_i(x) + \sum_{k=1}^n \sum_{\ell \neq k} \Tr(\nabla_j f_k(x) \nabla_j f_\ell(x)) \prod_{i \neq k, \ell} f_i(x).
\]
The special cased proved above shows that $\Delta_j[f_k(x)] = N[L_{N,j}f](x)$.  Moreover, by \eqref{eq:evaluategradient}, we have $\nabla_j[f_k(x))] = \frac{1}{N} [D_jf_k](x)$.  Thus, we have
\[
\Delta_j f(x) = \sum_{k=1}^n N[L_{N,j} f_k](x) \prod_{i \neq k} f_i(x) + \frac{1}{N} \sum_{k=1}^n \sum_{\ell \neq k} \tau_N([D_j f_k](x) [D_j f_\ell](x)) \prod_{i \neq k, \ell} f_i(x).
\]
Because of the product rule in the definition of $L_{N,j}$, the right hand side equals $N[L_{N,j} f](x)$.  This completes the proof of \eqref{eq:evaluateLaplacian} in the scalar-valued case.  The proof for the operator-valued case is similar, using the cases proved above, as well as \eqref{eq:evaluategradient} and \eqref{eq:evaluateJacobian}.
\end{proof}

\begin{corollary} \label{cor:Laplacianlimit}
Let $f \in \TrP_m^0$ or $\TrP_m^1$.  If we view $f$ as a function on $M_N(\C)_{sa}^m$, then $\Delta f$ is a trace polynomial of lower degree than $f$, and we have coefficient-wise
\[
\lim_{N \to \infty} \frac{1}{N} \Delta f(x) = \lim_{N \to \infty} L_Nf(x) = Lf(x).
\]
\end{corollary}

\begin{remark} \label{rem:normalizedderivatives}
We have shown that if $f$ is a scalar-valued trace polynomial, then viewed as a map $M_N(\C)_{sa}^m \to \C$, we have
\[
Du = N \nabla f, \qquad L_N f = \frac{1}{N} \Delta f.
\]
Therefore, in the rest of the paper, we will freely write $Df$ and $L_Nf$ for $N \nabla f$ and $(1/N) \Delta f$ for general functions $f: M_N(\C)_{sa}^m \to \C$.  The same considerations apply to the Laplacian for operator-valued trace polynomials, viewed as maps $M_N(\C)_{sa}^m \to M_N(\C)$.
\end{remark}

\subsection{Convolution of Trace Polynomials and Gaussians} \label{subsec:TPGaussianconvolution}

Let $f \in \TrP_m^0$ or $f \in \TrP_m^1$.  Then viewing $f$ is a function defined on $M_N(\C)_{sa}^m$, we may define the convolution of $f$ with the probability measure $\sigma_{t,N}$ (the law of an $m$-tuple of independent GUE).  This is equivalent to the classical convolution of $f$ with the function $M_N(\C)_{sa}^m \to \R$ giving the density of the measure $\sigma_{t,N}$.  Moreover, $f_t = f * \sigma_{t,N}$ is the solution to the heat equation with initial condition $f$, or more precisely
\[
\partial_t f_t = \frac{1}{2N} \Delta f_t.
\]
(The integral formula for the solution to the heat equation with the Laplacian $\Delta$ is well-known \cite[\S 2.3]{Evans}, and to solve the equation with $(1/2N)\Delta$ one renormalizes time by a factor of $1/2N$, and this corresponds precisely to our normalizations in the definition of $\sigma_{N,t}$.  We leave this computation to the reader.)

We showed in the last subsection that $L_N = \frac{1}{N} \Delta$ on trace polynomials is given by a purely algebraic computation.  Moreover, examining the construction of $L_N$, one can see that it maps trace polynomials of degree $\leq d$ to trace polynomials of degree $\leq d$.  We can view $L_N$ and $L$ as linear transformations on the finite-dimensional vector space of trace polynomials of degree $\leq d$ and define $\exp(tL_N / 2)$ and $\exp(tL / 2)$ by the matrix exponential.

Because this holds for any $d$, we know that $\exp(tL_N/2)$ and $\exp(tL/2)$ define linear transformations $\TrP_m^0 \to \TrP_m^0$ and $\TrP_m^1 \to \TrP_m^1$.  Moreover, a standard computation shows that $f_t = \exp(tL_N/2) f$ satisfies the heat equation $\partial_t f_t = (1/2) L_N f_t$.  These observations, together with Corollary \ref{cor:Laplacianlimit} yield the following.

\begin{lemma} \label{lem:Laplacianlimit}
Let $f$ be a trace polynomial in $\TrP_m^0$ or $\TrP_m^1$.  Then we have
\begin{equation}
\sigma_{t,N} * f(x) = [\exp(t L_N/2) f](x),
\end{equation}
with $\deg(\exp(tL_N/2) f) \leq \deg(f)$, and we have
\begin{equation}
\lim_{N \to \infty} \exp(t L_N / 2) f = \exp(t L / 2)f \text{ coefficient-wise.}
\end{equation}
\end{lemma}

\begin{example} \label{ex:fourthpower}
Let $X = (X_1,\dots,X_m)$ and define $f(X) = \sum_{j=1}^m X_j^2$.  Note that $\mathcal{D}_j^2[f(X)] = 2(1 \otimes 1 \otimes 1)$ for each $j$, and hence $L[\tau(f)] = 2m = L_N[\tau(f)]$.  We also have $\mathcal{D}_j^\circ f = 2 X_j$.  Hence,
\begin{align*}
L[\tau(f)^2] &= 2 L[\tau(f)] \tau(f) = 4m \tau(f) \\
L_N[\tau(f)^2] &= 2 L[\tau(f) \tau(f) + 2 \sum_{j=1}^m \tau(\mathcal{D}_j^\circ f \cdot \mathcal{D}_j^\circ f) \\
&= 4m \tau(f) + \frac{8m}{N^2} \tau(f).
\end{align*}
Therefore, $(L/2)[\tau(f)^2] = 2m \tau(f)$ and $(L/2)[\tau(f)] = m$.  Thus, the span of $\tau(f)^2$, $\tau(f)$, and $1$ is invariant under the operator $L/2$, and $L/2$ is given by a nilpotent matrix on this subspace.  Direct computation then shows that
\[
e^{-tL/2}[\tau(f)^2] = \tau(f)^2 + 2mt \tau(f) + m^2 t^2.
\]
A similar computation shows that
\[
e^{-tL_N/2}[\tau(f)^2] = \tau(f)^2 + 2m(1 + 2/N^2) t \tau(f) + 2m^2(1 + 2/N^2) t^2 / 2.
\]
Thus, as $N \to +\infty$, we have $e^{-tL_N/2}[\tau(f)^2] \to e^{-tL/2}[\tau(f)^2]$.
\end{example}

The probabilistic interpretation of $f * \sigma_{t,N} = \exp(t L_N/2) f$, which follows from a standard computation, is that $\sigma_{t,N} * f(x)$ is the expectation of $f(x + t^{1/2} Y)$, where $Y$ is an $m$-tuple of independent GUE of variance $1$.  Moreover, for any probability measure $\mu$ on $M_N(\C)_{sa}^m$ with finite moments, we have
\begin{equation}
\int f(x)\,d(\mu * \sigma_{t,N})(x) = \int (\sigma_{t,N} * f)(x)\,d\mu(x) = \int [\exp(tL_N/2) f](x)\,d\mu(x).
\end{equation}
In the free setting, the operator $\exp(tL/2)$ has a similar relationship with the free convolution with $\sigma_t$.  This fact is standard in free probability, but because we need it for Lemmas \ref{lem:convolutionapproximation} and \ref{lem:concentrationpreserved} below, we include a sketch of the proof here.

\begin{lemma} \label{lem:freeconvolutionevaluation}
Let $\lambda \in \Sigma_{n,R}$ be a non-commutative law.  Then for any trace polynomial $f \in \TrP_m^0$, we have
\begin{equation}
\lambda \boxplus \sigma_t(f) = \lambda(\exp(tL/2) f).
\end{equation}
\end{lemma}

\begin{proof}
Because free convolution with $\sigma_t$ forms a semigroup and $\exp(tL/2)$ is also a semigroup, it suffices to prove that
\[
\frac{d}{dt} \Bigr|_{t=0} \lambda \boxplus \sigma_t(f) = \frac{1}{2} \lambda(Lf).
\]
By the product rule, it suffices to handle the case of $f = \tau(p)$ for $p \in \NCP_m$ by showing that
\[
\frac{d}{dt} \Bigr|_{t=0} \lambda \boxplus \sigma_t(p) = \frac{1}{2} \lambda(\eta(\mathcal{D}_j^2 p)).
\]
Let $X = (X_1,\dots,X_m)$ be a random variable with law $\lambda$ and let $S = (S_1,\dots,S_m)$ be a freely independent tuple of semi-circulars realized together in a von Neumann algebra $(\mathcal{M},\tau)$.  We want to compute $\frac{d}{dt} \Bigr|_{t=0} \tau(p(X + t^{1/2}S))$.  But note that
\[
p(X + t^{1/2}S) = p(X) + t^{1/2} \sum_{j=1}^m \mathcal{D}_j p(X) \# S_j + \frac{1}{2} t \sum_{j,k=1}^m \mathcal{D}_j \mathcal{D}_k p(X) \# (S_j \otimes S_k) + O(t^{3/2})
\]
A moment computation with free independence shows that the terms of order $t^{1/2}$ have expectation zero, and so do the terms of order $t$ with $j \neq k$.  We are left with
\[
\frac{d}{dt} \Bigr|_{t=0} \tau(p(X + t^{1/2}S)) = \frac{1}{2} \sum_{j=1}^n \tau(\mathcal{D}_j^2 p(X) \# (S_j \otimes S_j)),
\]
which using freeness evaluates to $(1/2) \sum_{j=1}^n \tau(\eta \mathcal{D}_j^2 p(X)) = \tau((1/2)L p(X))$.
\end{proof}

\subsection{Asymptotic Approximation by Trace Polynomials} \label{subsec:defasymptoticapproximation}

Now we are ready to define the approximation property which captures the asymptotic behavior of functions on $M_N(\C)_{sa}^m$.

\begin{definition} \label{def:AATP}
A sequence of functions $\phi_N: M_N(\C)_{sa}^m \to M_N(\C)^m$ is said to be \emph{asymptotically approximable by trace polynomials} if for every $\epsilon > 0$ and $R > 0$, there exists some $f \in (\TrP_m^1)^m$ (an $m$-tuple of operator-valued trace polynomials) such that
\[
\limsup_{N \to \infty} \sup_{\norm{x}_\infty \leq R} \norm{\phi_N(x) - f(x)}_2 \leq \epsilon.
\]
In this case, we call $f$ an \emph{$(\epsilon,R)$-approximation of $\{\phi_N\}$}.  We make the same definitions for functions $\phi_N: M_N(\C)_{sa}^m \to \C$, except that we use scalar-valued trace polynomials (elements of $\TrP_m^0$) and apply the absolute value rather than the $2$-norm.
\end{definition}

\begin{observation}
If $f \in (\TrP_m^1)^m$ and if $f_N$ denotes the map $M_N(\C)_{sa}^m \to M_N(\C)^m$ given by $x \mapsto f(x)$, then $f_N$ is asymptotically approximable by trace polynomials.  Also, asymptotically approximable sequences form a vector space over $\C$.
\end{observation}

\begin{observation} \label{lem:convergenceapproximation}
Let $\{\phi_N^{(\ell)}\}_{N,\ell \in \N}$ be a sequence of functions where $\phi_N^{(\ell)}: M_N(\C)_{sa}^m \to M_N(\C)^m$.  Suppose that $\{\phi_N\}$ is another sequence such that for every $R > 0$,
\[
\lim_{\ell \to \infty} \limsup_{N \to \infty} \sup_{\norm{x}_\infty \leq R} \norm{\phi_N^{(\ell)}(x) - \phi_N(x)}_2 = 0.
\]
If $\{\phi_N^{(\ell)}\}_{N \in \N}$ is asymptotically approximable by trace polynomials for each $\ell$, then so is $\{\phi_N\}_{N \in \N}$.  The same holds in the case of scalar-valued functions and scalar-valued trace polynomials.
\end{observation}

\begin{lemma} \label{lem:compositionapproximation}
Let $\phi_N, \psi_N: M_N(\C)_{sa}^m \to M_N(\C)_{sa}^m$.  Suppose that $\{\phi_N\}$ and $\{\psi_N\}$ are both asymptotically approximable by trace polynomials, and furthermore suppose that $\{\phi_N\}_{N \in \N}$ is uniformly Lipschitz in $\norm{\cdot}_2$, that is, for some $L > 0$,
\[
\norm{\phi_N(x) - \phi_N(y)}_2 \leq K \norm{x - y}_2 \text{ for all } x, y, \text{ for all } N.
\]
Then $\{\phi_N \circ \psi_N\}$ is asymptotically approximable by trace polynomials.
\end{lemma}

\begin{proof}
Choose $\epsilon > 0$ and $R > 0$.  Choose a trace polynomial $g$ which is an $(\epsilon / 2K, R)$-approximation of $\{\psi_N\}$.  Since $g$ is a trace polynomial, there exists some $R' > 0$ such that for any tuple $x$ of self-adjoint matrices of any size, we have
\[
\norm{x}_\infty \leq R \implies \norm{g(x)}_\infty \leq R'.
\]
Now because $\phi_N$ is asymptotically approximable by trace polynomials, we can choose polynomial $f$ which is an $(\epsilon / 2, R')$-approximation of $\{\phi_N\}$.  Now we observe that when $\norm{x}_\infty \leq R$ (hence $\norm{g(x)}_\infty \leq R'$), we have
\begin{align*}
\norm{\phi_N \circ \psi_N(x) - f \circ g(x)}_2 &\leq \norm{\phi_N \circ \psi_N(x) - \phi_N \circ g(x)}_2 + \norm{\phi_N \circ g(x) - f \circ g(x)}_2 \\
&\leq K \sup_{\norm{x}_\infty \leq R} \norm{\psi_N(x) - g(x)}_2 + \sup_{\norm{y}_\infty \leq R'} \norm{\phi_N(y) - f(y)}_2.
\end{align*}
Therefore,
\[
\limsup_{N \to \infty} \sup_{\norm{x}_\infty \leq R} \norm{\phi_N \circ \psi_N(x) - f \circ g(x)}_2 \leq K \cdot \frac{\epsilon}{2K} + \frac{\epsilon}{2} = \epsilon. \qedhere
\]
\end{proof}

\begin{lemma} \label{lem:convolutionapproximation}
Suppose that $\phi_N: M_N(\C)_{sa}^m \to M_N(\C)_{sa}^m$ is asymptotically approximable by trace polynomials and that
\begin{equation} \label{eq:growthrate}
\norm{\phi_N(x)}_2 \leq A\left( 1 +  \sum_j \tau_N(x_j^{2n}) \right)
\end{equation}
for some $A  > 0$ and some integer $n \geq 0$.  If $\{\phi_N\}$ is asymptotically approximable by trace polynomials, then so is $\{\phi_N * \sigma_{t,N}\}$.
\end{lemma}

\begin{proof}
Fix $R > 0$ and $\epsilon > 0$.  Choose a trace polynomial $f$ which is an $(\epsilon, R + 3t^{1/2})$ approximation for $\{\phi_N\}$.  Now for $x$ with $\norm{x}_\infty \leq R$, we estimate
\[
\norm{\sigma_{t,N} * \phi_N(x) - \sigma_{t,N} * f(x)}_2 \leq \int \norm{\phi_N(x+y) - f(x+y)}_2\,d\sigma_{t,N}(y).
\]
We break this integral into two pieces:  The integral over the region where $\norm{y}_\infty \leq 3t^{1/2}$ is bounded by $\epsilon$ as $N \to \infty$ by our choice of $f$.  Furthermore, we claim that the integral over the region where $\norm{y}_\infty > 3t^{1/2}$ vanishes as $N \to \infty$.  Using assumption \eqref{eq:growthrate} and the fact that $f$ is a trace polynomial, we see that there exists a $C > 0$ and integer $d > 0$, depending only on $R$, $A$, $n$, and $f$, such that
\[
\sup_{\norm{x}_\infty \leq R} [\norm{\phi_N(x + y)}_2 + \norm{f(x + y)}_2] \leq C\left(1 + \sum_j \tau_N(y_j^{2d}) \right).
\]
Therefore, we have
\[
\int_{\norm{y}_\infty \geq 3t^{1/2}} \norm{\phi_N(x+y) - f(x+y)}_2\,d\sigma_{t,N}(y) \leq  C \int_{\norm{y}_\infty \geq 3t^{1/2}} \left(1 + \sum_j \tau_N(y_j^{2d}) \right) \,d\sigma_{t,N}(y).
\]
This vanishes as $N \to \infty$ by Corollary \ref{cor:operatornormtail} applied to the GUE.  Therefore, we have
\[
\limsup_{N \to \infty} \sup_{\norm{x}_\infty \leq R} \norm{\sigma_{t,N} * \phi_N(x) - \sigma_{t,N} * f(x)}_2 \leq \epsilon.
\]
On the other hand, by Lemma \ref{lem:Laplacianlimit}, we have $\sigma_{t,N} * f = \exp(tL_N/2) f \to \exp(tL/2) f$ coefficient-wise, and therefore,
\[
\limsup_{N \to \infty} \sup_{\norm{x}_\infty \leq R} \norm{\sigma_{t,N} * f(x) - [\exp(tL/2) f](x)}_2 = 0,
\]
so that
\[
\limsup_{N \to \infty} \sup_{\norm{x}_\infty \leq R} \norm{\sigma_{t,N} * \phi_N(x) - [\exp(tL/2) f](x)}_2 \leq \epsilon.  \qedhere
\]
\end{proof}

\begin{lemma} \label{lem:gradientapproximation}
Suppose that $\phi_N: M_N(\C)_{sa}^m \to \C$ and suppose that $\{D\phi_N\} = \{N \nabla \phi_N\}$ is asymptotically approximable by trace polynomials and that $\phi_N(0) = 0$.  Then $\{\phi_N\}$ is asymptotically approximable by trace polynomials.
\end{lemma}

\begin{proof}
Given a trace polynomial $F \in (\TrP_m^1)^m$, we can define
\[
f(X) = \int_0^1 \tau(F(tX)X)\,dt
\]
in $\TrP_m^0$.  Then we have
\begin{align*}
\sup_{\norm{x}_\infty \leq R} |\phi_N(x) - f(x)| &= \sup_{\norm{x}_\infty \leq R} \left| \int_0^1 \ip{D \phi_N(tx) - F(tx), x}_2\,dt \right| \\
&\leq R \sup_{\norm{y}_\infty \leq R} \norm{N \nabla \phi_N(y) - F(y)}_2.  \qedhere
\end{align*}
\end{proof}

\section{Convergence of Moments} \label{sec:convergenceandconcentration}

Our goal in this section is prove the following theorem.  The convergence of moments is related to \cite[Theorem 4.4]{GS2009}, \cite[Proposition 50 and Theorem 51]{DGS2016}, \cite[Theorem 4.4]{Dabrowski2017}, and we include versions of standard concentration estimates (not proved in this paper) in the statement.

\begin{theorem} \label{thm:convergenceandconcentration}
Let $V_N: M_N(\C)_{sa}^m \to \R$ be a sequence of potentials such that $V_N(x) - (c/2) \norm{x}_2^2$ is convex and $V_N(x) - (C/2) \norm{x}_2^2$ is concave.  Let $\mu_N$ be the associated measure.  Suppose that the sequence $\{DV_N\}$ is asymptotically approximable by trace polynomials, and assume that
\begin{equation} \label{eq:operatornormhypothesis}
M = \limsup_{N \to \infty} \max_j \norm*{ \int(x_j - \tau_N(x_j)1)\,d\mu_N(x) } < +\infty,
\end{equation}
where $1$ denotes the $N \times N$ identity matrix.
\begin{enumerate}
	\item We have the following bounds on the operator norm.  If $R_N = \max_j \int \norm{x_j}\,d\mu_N(x)$, then
	\begin{align*}
	\limsup_{N \to \infty} R_N &\leq \frac{2}{c^{1/2}} + \frac{1}{c} \limsup_{N \to \infty} \max_j \left| \int \tau_N(x_j)\,d\mu_N(x) \right| + M \\
	&\leq \frac{2}{c^{1/2}} + \frac{1}{c} \limsup_{N \to \infty} \norm{DV_N(0)}_2 + \frac{C - c}{2c^{3/2}} + M,
	\end{align*}
	and as a consequence of concentration we have for each $j$ that
	\[
	\mu_N(\norm{x_j} \geq R_N + \delta) \leq e^{-cN \delta / 2}.
	\]
	\item There exists a non-commutative law $\lambda \in \Sigma_{m,R_*}$, where $R_* = \limsup_{N \to \infty} R_N$, such that for every non-commutative polynomial $p$,
	\[
	\lim_{N \to \infty} \int \tau_N(p(x))\,d\mu_N(x) = \lambda(p).
	\]
	\item The sequence $\{\mu_N\}$ exhibits exponential concentration around $\lambda$ in the sense that for every $R > 0$, and every neighborhood $\mathcal{U}$ of $\lambda$ in $\Sigma_m$,
	\[
	\limsup_{N \to \infty} \frac{1}{N^2} \log \mu_N\left(\norm{x}_\infty \leq R, \lambda_x \not \in \mathcal{U}\right) < 0.
	\]
\end{enumerate}
\end{theorem}

\begin{remark}
The rather artificial hypothesis that $\limsup_{N \to \infty} \max_j \norm*{ \int(x_j - \tau_N(x_j))\,d\mu_k(x) } < +\infty$ is trivially satisfied if either (1) $\mu_N$ has expectation zero or (2) $\mu_N$ is invariant under unitary conjugation and hence $\int x_j \,d\mu_N(x)$ is equal to $\int \tau_N(x_j)\,d\mu_N(x)$ times the identity matrix.
\end{remark}

We have already seen in \S \ref{subsec:concentration} that concentration estimates and operator norm tail bounds are standard.  To prove that the moments converge, something more is needed; indeed, the only assumption relating the measures $\mu_N$ for different values of $N$ is the fact that $DV_N$ is asymptotically approximable by trace polynomials.  But even if $DV_N$ is given by the same ``trace analytic-function'' for different values of $N$, it is not immediate that the measure would concentrate in the same regions for different size matrices.

To prove convergence of moments, we want to express $\int u\,d\mu_N$ in terms of $DV_N$ for a Lipschitz function $u$.  One of the standard techniques is to show $\mu_N$ is the unique stationary distribution for a process $X_t$ that satisfies the SDE
\[
dX_t = dY_t - \frac{1}{2} DV_N(X_t)\,dt,
\]
where $Y_t$ is a GUE Brownian motion.  This machinery lies behind the log-Sobolev inequality and concentration results, as well as earlier theorems about convergence of moments for general convex potentials.

Specifically, Dabrowski, Guionnet, and Shylakhtenko used the free version of this SDE to show that for a non-commutative potential $V$ satisfying certain convexity assumptions, there exists a free Gibbs law for $V$ which is the unique stationary distribution \cite[Proposition 5]{DGS2016}.  As an application, they show convergence of moments for random matrix models given by $V_N = V$ \cite[Proposition 50 and Theorem 51]{DGS2016}, essentially a special case of our Theorem \ref{thm:convergenceandconcentration}.

Dabrowski was able to show convergence of moments under weaker convexity assumptions by constructing the solution to the free SDE as an ultralimit of the finite-dimensional solutions \cite[Theorem 4.4]{Dabrowski2017}.  Our theorem has similar convexity assumptions to Dabrowski's, but we consider a more general sequence of potentials $V_N$.  We also perform most of our analysis in the finite-dimensional setting, but we use deterministic rather than stochastic methods.

Instead of the solving the SDE, we study the associated semigroup $T_t^{V_N}$, acting on Lipschitz functions $u$, given by
\[
T_t^{V_N} u(x) = E_x[u(X_t)],
\]
where $X_t$ is the process solving the SDE with initial condition $x$.  The semigroup provides the solution to a certain PDE, that is, if $u(x,t) = T_t u_0(x)$, then we have
\[
\partial_t u = \frac{1}{2N} \Delta u - \frac{N}{2} \nabla V_N \cdot \nabla u = \frac{1}{2} L_N u - \frac{1}{2} \ip{DV_N,Du}_2.
\]
The semigroup $T_t^{V_N}$ will decrease the Lipschitz norms of functions and thus, if $u$ is Lipschitz, then $T_t^{V_N} u$ will converge to $\int u \,d\mu_N$ as $t \to \infty$.

Solving the differential equation and taking $t \to \infty$ provides a way to evaluate $\int u\,d\mu_N$ in terms of $DV_N$.  We will describe a construction of the semigroup $T_t^V$ through iterating simpler operations (\S \ref{subsec:iteration0}), and then we will show (Lemma \ref{lem:asymptoticapproximation0}) that the iteration procedure preserves approximability by trace polynomials and hence conclude that $\lim_{N \to \infty} \int u\,d\mu_N$ exists.

\subsection{Iterative Construction of the Semigroup} \label{subsec:iteration0}

To simplify notation in this section, we fix $N$ and fix a potential $V: M_N(\C)_{sa}^m \to \R$ such that $V(x) - (c/2) \norm{x}_2^2$ is convex and $V(x) - (C/2) \norm{x}_2^2$ is concave, and we write $T_t$ rather than $T_t^V$.

We will construct $T_t$ by combining two simpler semigroups corresponding to the stochastic and deterministic terms of $dY_t - (1/2) DV(X_t)\,dt$.  Recall that the solution to the heat equation $\partial_t u = (1/2N) \Delta u$ with initial data $u_0$ is given by the heat semigroup:
\[
P_t u_0(x) = \int u_0(x + y)\,d\sigma_{t,N}(y).
\]
Meanwhile, the solution to $\partial_t u = -(1/2) \ip{DV,Du}_2$ with initial data $u_0$ is given by
\[
S_t u_0(x) = u_0(W(x,t)),
\]
where $W(x,t)$ is the solution to the ODE
\begin{equation} \label{eq:ODE}
\partial_t W(x,t) = -\frac{1}{2} DV(W(x,t)) \qquad W(x,0) = x.
\end{equation}
We want to define $T_t = \lim_{n \to \infty} (P_{t/n} S_{t/n})^n$.  This is motivated by Trotter's product formula which asserts that $e^{t(A+B)} = \lim_{n \to \infty} (e^{tA/n} e^{tB/n})^n$ for nice enough self-adjoint operators $A$ and $B$ (see \cite{Trotter1959}, \cite{Kato1978}, \cite[{pp.\ 4 - 6}]{Simon1979}).  But of course, we must show that $(P_{t/n} S_{t/n})^n$ converges and derive dimension-independent error bounds.

We use the following basic properties of the semigroups $P_t$ and $S_t$.  Here if $u: M_N(\C)_{sa}^m \to \C$, then $\norm{u}_{\Lip}$ denotes the Lipschitz norm with respect to the \emph{normalized} $L^2$ metric $\norm{\cdot}_2$ on $M_N(\C)_{sa}^m$ and $\norm{u}_{L^\infty}$ denotes the standard $L^\infty$ norm.  We are only concerned with Lipschitz functions, so in the following estimates, the reader may always assume $u$ is Lipschitz, but of course $\norm{u}_{L^\infty}$ may be infinite for Lipschitz functions.

\begin{lemma} \label{lem:heatsemigroup} ~
\begin{enumerate}
	\item $\norm{P_tu}_{L^\infty} \leq \norm{u}_{L^\infty}$.
	\item $\norm{P_tu}_{\Lip} \leq \norm{u}_{\Lip}$.
	\item $\norm{P_tu - u}_{L^\infty} \leq m^{1/2} t^{1/2} \norm{u}_{\Lip}$.
\end{enumerate}
\end{lemma}

\begin{proof}
(1) and (2) follow from the fact that $P_t u$ is $u$ convolved with a probability measure.  To prove (3), suppose $\norm{u}_{\Lip} < +\infty$.  Then
\begin{align*}
|P_tu(x) - u(x)| &= \left| \int (u(x + y) - u(x))\,d\sigma_{t,N}(y) \right| \\
&\leq \int |u(x+y) - u(x)|\,d\sigma_{t,N}(y) \\
&\leq \norm{u}_{\Lip} \int \norm{y}_2\,d\sigma_{t,N}(y).
\end{align*}
Meanwhile,
\[
\int \norm{y}_2\,d\sigma_{t,N}(y) \leq \left( \int 1\,d\sigma_{t,N}(y) \right) \left( \int \norm{y}_2^2 \,d\sigma_{t,N}(y) \right)^{1/2} = (mt)^{1/2},
\]
since $y$ an $m$-tuple $(y_1,\dots,y_m)$ and $\int \tau_N(y_j^2)\,d\sigma_{t,N}(y) = t$ for each $j$.
\end{proof}

\begin{lemma} \label{lem:deterministicsemigroup} ~
\begin{enumerate}
	\item The solution $W(x,t)$ to \eqref{eq:ODE} exists for all $t$.
	\item $\norm{W(x,t) - W(y,t)}_2 \leq e^{-ct/2} \norm{x - y}_2$.
	\item $\norm{W(x,t) - x}_2 \leq (t/2) \norm{DV(x)}_2$.
	\item $\norm{(W(x,t) - x) - (W(y,t) - y)}_2 \leq \frac{C}{c}(1 - e^{-ct/2}) \norm{x - y}_2$.
	\item $\norm{S_t u}_{\Lip} \leq e^{-ct/2} \norm{u}_{\Lip}$.
	\item $\norm{S_t u}_{L^\infty} \leq \norm{u}_{L^\infty}$.
\end{enumerate}
\end{lemma}

\begin{proof}
(1) The convexity and semi-concavity assumptions on $V$ imply that $DV$ is $C$-Lipschitz and therefore global existence of the solution follows from the Picard--Lindel\"of Theorem.

(2) Let $\tilde{V}(x) = V(x) - (c/2) \norm{x}_2^2$.  Because $\tilde{V}$ is convex, we have
\[
\ip{D\tilde{V}(x) - D\tilde{V}(y), x - y}_2 \geq 0,
\]
which translates to
\[
\ip{DV(x) - DV(y), x - y}_2 \geq c \norm{x - y}_2^2.
\]
Now observe that
\begin{align*}
\frac{d}{dt} \norm{W(x,t) - W(y,t)}_2^2 &= -\ip{DV(W(x,t)) - DV(W(y,t)), W(x,t) - W(y,t)}_2 \\
&\leq -c \norm{W(x,t) - W(y,t)}_2^2,
\end{align*}
and hence by Gr{\"o}nwall's inequality, $\norm{W(x,t) - W(y,t)}_2^2 \leq e^{-ct} \norm{W(x,0) - W(y,0)}_2^2 = e^{-ct} \norm{x - y}_2^2$.

(3) Note that
\begin{align*}
\frac{d}{dt} \norm{W(x,t) - x}_2^2 = -\ip{DV(W(x,t)), W(x,t) - x}_2 \\
= -\ip{DV(W(x,t)) - DV(x), W(x,t) - x}_2 - \ip{DV(x), W(x,t) - x}_2 \\
\leq \norm{DV(x)}_2 \norm{W(x,t) - x}_2.
\end{align*}
Meanwhile, $\norm{W(x,t) - x}_2$ is Lipschitz in $t$ and hence differentiable almost everywhere and we have
\[
\frac{d}{dt} \norm{W(x,t) - x}_2^2 = 2 \norm{W(x,t) - x}_2 \frac{d}{dt} \norm{W(x,t) - x}_2.
\]
Thus, we have
\[
\frac{d}{dt} \norm{W(x,t) - x}_2 \leq \frac{1}{2} \norm{DV(x)}_2,
\]
which proves (3).

(4) We observe that
\begin{align*}
\norm{(W(x,t) - x) - (W(y,t) - y)}_2 &\leq \frac{1}{2} \int_0^t \norm{DV(W(x,s)) - DV(W(y,s))}_2\,ds \\
&\leq \frac{C}{2} \int_0^t \norm{W(x,s) - W(y,s)}_2\,ds \\
&\leq \frac{C}{2} \int_0^t e^{-cs/2} \norm{x - y}_2\,ds \\
&= \frac{C}{c}(1 - e^{-ct/2}) \norm{x - y}_2.
\end{align*}

(5) follows from (2).

(6) is immediate because $S_t u$ is $u$ precomposed with another function.
\end{proof}

Now we combine $P_t$ and $S_t$ as in Trotter's formula, except that for technical convenience we define our approximations using dyadic time intervals rather than subdividing $[0,t]$ into intervals of size $t/n$.

\begin{lemma} \label{lem:convergence0}
For dyadic $t \in 2^{-\ell} \N$, define
\[
T_{t,\ell} u = (P_{2^{-\ell}} S_{2^{-\ell}})^{2^\ell t} u.
\]
Then $T_t u := \lim_{\ell \to \infty} T_{t,\ell}u$ exists and we have
\[
\norm{T_{t,\ell}u - T_t u}_{L^\infty} \leq \frac{Cm^{1/2}}{c(2 - 2^{1/2})} 2^{-\ell/2} \norm{u}_{\Lip}.
\]
We also have $\norm{T_tu}_{\Lip} \leq e^{-ct/2} \norm{u}_{\Lip}$.
\end{lemma}

\begin{proof}
We want to show that the sequence $\{T_{t,\ell} u\}_{\ell}$ is Cauchy by estimating the difference between consecutive terms.  Suppose that $t \in 2^{-\ell}\N$ and write $t = n / 2^\ell$ and $\delta = 2^{-\ell-1}$.  Note the telescoping series identity
\[
T_{t,\ell+1}u - T_{t,\ell} u = \sum_{j=0}^{n-1} (P_\delta S_\delta)^{2j}P_\delta(S_\delta P_\delta - P_\delta S_\delta) S_\delta (P_{2\delta} S_{2\delta})^{n-1-j}u.
\]
Thus, we want to estimate $S_\delta P_\delta - P_\delta S_\delta$ and then control the propagation of the errors through the applications of the other operators.  Note that for a Lipschitz function $v$, we have
\begin{align*}
|S_\delta P_\delta v(x) - P_\delta S_\delta v(x)| &\leq \int |v(W(x,\delta) + y) - v(W(x+y,\delta))|\,d\sigma_{\delta,N}(y) \\
&\leq \norm{v}_{\Lip} \int \norm{(W(x,\delta) - x) - (W(x+y,\delta) - (x + y))}_2\,d\sigma_{\delta,N}(y) \\
&\leq \norm{v}_{\Lip} \frac{C}{c} (1 - e^{-c\delta/2}) \int \norm{y}_2 \,d\sigma_{\delta,N}(y) \\
&\leq \norm{v}_{\Lip} \frac{C}{c} (1 - e^{-c\delta/2}) (m\delta)^{1/2},
\end{align*}
where the last inequality follows by the same reasoning as Lemma \ref{lem:heatsemigroup} (3).  Therefore,
\[
\norm{S_\delta P_\delta v - P_\delta S_\delta v}_{L^\infty} \leq \frac{C}{c} m^{1/2} \delta^{1/2} (1 - e^{-c\delta/2})  \norm{v}_{\Lip}.
\]
Therefore, we can estimate a single term in the telescoping series identity by
\begin{align*}
\norm{(P_\delta S_\delta)^{2j}P_\delta(S_\delta P_\delta - P_\delta S_\delta) S_\delta (P_{2\delta} S_{2\delta})^{n-1-j}u}_{L^\infty} &\leq \norm{(S_\delta P_\delta - P_\delta S_\delta) S_\delta (P_{2\delta} S_{2\delta})^{n-1-j}u}_{L^\infty} \\
&\leq \frac{C}{c} m^{1/2} \delta^{1/2} (1 - e^{-c\delta/2}) \norm{S_\delta (P_{2\delta} S_{2\delta})^{n-1-j}u}_{\Lip} \\
&\leq \frac{C}{c} m^{1/2} \delta^{1/2} (1 - e^{-c\delta/2}) e^{-c\delta/2} e^{-c\delta(n-j-1)/2} \norm{u}_{\Lip}.
\end{align*}
Here we have first applied the fact that $P_\delta$ and $S_\delta$ are contractions with respect to the $L^\infty$ norm from Lemma \ref{lem:heatsemigroup} (1) and Lemma \ref{lem:deterministicsemigroup} (6); second, we used the above estimate on $S_\delta P_\delta - P_\delta S_\delta$; and third we used the estimates $\norm{P_\delta u}_{\Lip} \leq \norm{u}_{\Lip}$ and $\norm{S_\delta u}_{\Lip} \leq e^{-c\delta/2} \norm{u}_{\Lip}$ found in Lemma \ref{lem:heatsemigroup} (2) and Lemma \ref{lem:deterministicsemigroup} (5).  Now summing up the telescoping series, we get
\begin{align*}
\norm{T_{t,\ell+1}u - T_{t,\ell} u}_{L^\infty} &\leq \sum_{j=0}^{n-1} \frac{C}{c} m^{1/2} \delta^{1/2} (1 - e^{-c\delta/2}) e^{-c\delta/2} e^{-c\delta(n-j-1)/2} \norm{u}_{\Lip} \\
&\leq \frac{C}{c} m^{1/2} \delta^{1/2} (1 - e^{-c\delta/2}) e^{-c\delta/2} \frac{1}{1 - e^{-c\delta/2}} \norm{u}_{\Lip} \\
&= \frac{C}{c} m^{1/2} \delta^{1/2} e^{-c\delta/2} \norm{u}_{\Lip} \leq \frac{C}{2c} m^{1/2} \delta^{1/2} \norm{u}_{\Lip}.
\end{align*}
In other words, we have $\norm{T_{t,\ell+1}u - T_{t,\ell} u}_{L^\infty} \leq \frac{Cm^{1/2}}{2c} 2^{-(\ell+1)/2} \norm{u}_{\Lip}$.  It follows that the sequence is Cauchy with respect to $\norm{\cdot}_{L^\infty}$ and we have the desired estimate on $\norm{T_{t,\ell}u - T_tu}_{L^\infty}$ from summing the geometric series.

The estimate $\norm{T_{t,\ell}u}_{\Lip} \leq e^{-ct/2} \norm{u}_{\Lip}$ follows from Lemma \ref{lem:heatsemigroup} (2) and Lemma \ref{lem:deterministicsemigroup} (5), and then by taking the limit as $\ell \to +\infty$, we obtain $\norm{T_t u}_{\Lip} \leq e^{-ct/2} \norm{u}_{\Lip}$.
\end{proof}

\begin{lemma} \label{lem:continuityofsemigroup0}
The semigroup $T_t$ defined above extends to a semigroup defined for positive $t$ such that for $s \leq t$,
\[
|T_tu(x) - T_su(x)| \leq e^{-cs/2} \left( \frac{C}{c} (6 + 5 \sqrt{2}) (t - s)^{1/2} + \norm{DV(x)}_2 (t - s) \right) \norm{u}_{\Lip},
\]
and $\norm{T_t u}_{\Lip} \leq e^{-ct/2} \norm{u}_{\Lip}$.
\end{lemma}

\begin{proof}
We first prove the estimate on $|T_t u - T_s u|$ for dyadic values of $s$ and $t$.  First, consider the case where $t = 2^{-\ell}$ and $s = 0$.  Note that
\[
(T_t - 1)u = (T_t - P_t S_t) u + (P_t - 1) S_t u + (S_t - 1)u.
\]
The first term can be estimated by Lemma \ref{lem:convergence0} with $\ell = 1$, the second term can be estimated by Lemma \ref{lem:heatsemigroup} (3) and Lemma \ref{lem:deterministicsemigroup} (5) as
\[
\norm{(P_t - 1) S_t u}_{L^\infty} \leq m^{1/2} t^{1/2} \norm{S_t u}_{\Lip} \leq m^{1/2} t^{1/2} \norm{u}_{\Lip},
\]
and the third term can be estimated by Lemma \ref{lem:deterministicsemigroup} (3).  Altogether, we obtain
\[
|T_tu(x) - u(x)| \leq \left( \frac{Cm^{1/2}}{c(2 - 2^{1/2})} t^{1/2} + m^{1/2} t^{1/2} + \frac{t}{2} \norm{DV(x)}_2 \right) \norm{u}_{\Lip}.
\]
In the case of general dyadic $s$ and $t$, suppose $t > s$ and write $t - s$ in a binary expansion to obtain
\[
t = s + \sum_{j=n+1}^\infty a_j 2^{-j},
\]
where $a_j \in \{0,1\}$ and $a_{n+1} = 1$.  Note that $2^{-n-1} \leq |s - t| \leq 2^{-n}$.  Let $t_k = s + \sum_{j=n+1}^k a_j 2^{-j}$.  Then
\begin{align*}
|T_t u(x) - T_s u(x)| &\leq \sum_{j=n+1}^\infty |T_{t_j}u(x) - T_{t_{j-1}} u(x)| \\
&\leq \sum_{j=n+1}^\infty \left( \frac{Cm^{1/2}}{c(2 - 2^{1/2})} 2^{-j/2} + m^{1/2} 2^{-j/2} + \frac{2^{-j}}{2} \norm{DV(x)}_2 \right) \norm{T_{t_{j-1}} u}_{\Lip} \\
&\leq \left( \left( \frac{Cm^{1/2}}{c(2 - 2^{1/2})} + 1 \right) \frac{1}{1 - 2^{-1/2}} \cdot 2^{-n/2} + \norm{DV(x)}_2 \cdot 2^{-n-1} \right) \norm{T_s u}_{\Lip} \\
&\leq \left( \left( \frac{Cm^{1/2}}{c(2 - 2^{1/2})} + 1 \right) \frac{2^{1/2}}{1 - 2^{-1/2}} (t - s)^{1/2} + \norm{DV(x)}_2 \cdot 2^{-n-1} \right) e^{-cs/2} \norm{u}_{\Lip} \\
&\leq e^{-cs/2} \left( \frac{Cm^{1/2}}{c} (6 + 5 \sqrt{2}) (t - s)^{1/2} + \norm{DV(x)}_2 (t - s) \right) \norm{u}_{\Lip},
\end{align*}
where we used the crude estimate that $1 \leq Cm^{1/2} / c$ to combine the first two terms.  Because this continuity estimate holds for dyadic values of $s$ and $t$, we can extend the definition of $T_t u$ to all positive $t$.  Furthermore, because $\norm{T_tu}_{\Lip} \leq e^{-ct/2} \norm{u}_{\Lip}$ for dyadic $t$, the same must hold for real values of $t$.

Now let us verify that $T_s T_t = T_{s+t}$ for all real $t$.  Choose dyadic $s_n \searrow s$ and $t_n \searrow t$ and let $u$ be a Lipschitz function.  We know that  $T_{s_n} T_{t_n} u = T_{s_n+t_n} u$ and that $T_{s_n+t_n} u \to T_{s+t}u$ locally uniformly, so it suffices to show that $T_{s_n} T_{t_n} u \to T_s T_t u$.  Observe that
\[
|T_{s_n} T_{t_n} u - T_s T_t u| \leq |(T_{s_n} - T_s)T_{t_n}u| + |T_s(T_{t_n} - T_t)u|.
\]
The first term can be estimated by
\[
|(T_{s_n} - T_s)T_{t_n}u(x)| \leq e^{-s/2} \left( \frac{C}{c} (6 + 5 \sqrt{2}) (s_n - s)^{1/2} + \norm{DV(x)}_2 (s_n - s) \right) \norm{T_{t_n}u}_{\Lip},
\]
which goes to zero as $n \to \infty$.  For the second term, we first note that
\[
|(T_{t_n} - T_t)u(x)| \leq e^{-t/2} \left( \frac{C}{c} (6 + 5 \sqrt{2}) (t_n - t)^{1/2} + \norm{DV(x)}_2 (t_n - t) \right) \norm{u}_{\Lip}
\]
Let $h_n(x)$ be the right hand side.  Note that $u \leq v$ implies that  $T_s u \leq T_s v$ because this holds for the operators $P_s$ and $S_s$ (since $P_s$ is given by convolution and $S_s$ is given by composition).  Therefore,
\[
|T_s(T_{t_n} - T_t)u(x)| \leq T_s|(T_{t_n} - T_t)u|(x) \leq T_s h_n(x).
\]
Because $DV$ is $C$-Lipschitz, we know that $h_n$ is a $e^{-t/2}(t_n - t) C \norm{u}_{\Lip}$-Lipschitz function and hence
\begin{align*}
|T_s h_n(x)| &\leq h_n(x) + |(T_s - 1)h_n(x)| \\
&\leq h_n(x) + e^{-t/2}(t_n - t) C \norm{u}_{\Lip} \left( \frac{C}{c} (6 + 5 \sqrt{2}) s^{1/2} + \norm{DV(x)}_2 s \right),
\end{align*}
which goes to zero as $n \to \infty$.
\end{proof}

\begin{lemma} \label{lem:weaksolution}
Let $u(x)$ be Lipschitz.  Then $T_t u$ is a weak solution of the equation
\[
\partial_t T_t u = \frac{1}{2N} \Delta(T_tu) - \frac{N}{2} \nabla V \cdot \nabla (T_tu)
\]
in the sense that for $\phi \in C_c^\infty(M_N(\C)_{sa}^m)$, we have
\[
\int_{M_N(\C)_{sa}^m} [T_{t_1}u\, \phi - T_{t_0} u\,\phi] =  \int_{t_0}^{t_1} \int_{M_N(\C)_{sa}^m} \left[ -\frac{1}{2N} \nabla (T_s u) \cdot \nabla \phi - \frac{N}{2} (\nabla V \cdot \nabla (T_su)) \phi \right]\,ds.
\]
\end{lemma}

\begin{proof}
Recall that by Rademacher's theorem if $u$ is Lipschitz, then $\nabla u$ exists almost everywhere and it is in $L^\infty$.  Moreover, because $\nabla V$ is Lipschitz, we also know that the second derivatives of $V$ exist almost everywhere and are in $L^\infty$.

We begin by considering $\int (S_\delta P_\delta - 1) u \cdot \phi$ for a Lipschitz $u: M_N(\C)_{sa}^m \to \R$ and a $\phi \in C_c^\infty(M_N(\C)_{sa}^m)$ and $\delta > 0$.  Note that
\[
(S_\delta P_\delta - 1) u = (S_\delta - 1) P_\delta u + (P_\delta - 1) u. 
\]
Now $P_\delta u$ is the convolution of $u$ with the Gaussian and so $\nabla (P_\delta u) = P_\delta (\nabla u)$.  Because the gradient of the Gaussian is $O(\delta^{-1/2})$, we see that the first derivatives of $P_\delta(\nabla u)$ are $O(\delta^{-1/2})$ in $L^\infty$.  (Here our estimates may depend on $N$.)
\[
P_\delta u(y) - P_\delta u(x) = \nabla P_\delta u(x) \cdot (x - y) + O(\delta^{-1/2} \norm{x - y}_2^2).
\]
Now using equation \eqref{eq:ODE} and Lemma \ref{lem:deterministicsemigroup} (3), we have $W(x,\delta) - x = \frac{N \delta}{2} \nabla V(x) + O(\delta^2)$ uniformly on any compact set $K$.  Therefore,
\[
(S_\delta - 1)P_\delta u(x) = P_\delta u(W(x,\delta)) - P_\delta u(x) = -\frac{N\delta}{2} \nabla (P_\delta u)(x) \cdot \nabla V(x) + O(\delta^{3/2}).
\]
Now we have
\begin{align*}
\int (S_\delta P_\delta - 1)u \cdot \phi &= \int (S_\delta - 1) P_\delta u \, \phi + \int (P_\delta - 1) u \, \phi \\
&= -\frac{N \delta}{2} \int [\nabla(P_\delta u) \cdot \nabla V] \phi + \int u\, (P_\delta - 1)\phi + O(\delta^{3/2}) \\
&= -\frac{N \delta}{2} \int P_\delta u [(\Delta V) \phi + \nabla V \cdot \nabla \phi] + \int u \frac{\delta}{2N} \Delta \phi + O(\delta^{3/2}) \\
&= - \frac{N\delta}{2} \int u \, P_\delta[(\Delta V) \phi + \nabla V \cdot \nabla \phi]  + \frac{\delta}{2N} \int u \, \Delta \phi + O(\delta^{3/2}),
\end{align*}
where the error estimates depend only on $C$, $N$, $\norm{u}_{\Lip}$, the support of $\phi$, and the $L^\infty$ norms of its derivatives.  We also know from the proof of Lemma \ref{lem:convergence0} that $(S_\delta P_\delta - P_\delta S_\delta) u$ is bounded by $\norm{u}_{\Lip} (Cm^{1/2}/c) (1 - e^{-c\delta}) \delta^{1/2}$ which is $O(\delta^{3/2})$.  Therefore,
\[
\int (P_\delta S_\delta - 1)u \cdot \phi = -\frac{N\delta}{2} \int u \, P_\delta[\Delta V \phi + \nabla V \cdot \nabla \phi]  + \frac{\delta}{2N} \int u \, \Delta \phi + O(\delta^{3/2}).
\]

Now suppose that $t$ is a dyadic rational and write $t = n \delta$ where $\delta = 2^{-\ell}$ for some integer $\ell$.  Recall that $T_{t,\ell} = (P_\delta S_\delta)^n$.  Then by a telescoping series argument
\[
\int (T_{t,\ell} - 1)u \cdot \phi = \sum_{j=0}^{n-1} \left( -\frac{N\delta}{2} \int T_{j\delta,\ell} u \, P_\delta[(\Delta V)\phi + \nabla V \cdot \nabla \phi]  + \frac{\delta}{2N} \int T_{j\delta,\ell} u \, \Delta \phi \right) + O(\delta^{1/2}).
\]
We fix a dyadic $t$ and take $\ell \to \infty$ (and hence $\delta \to 0$).  The above sum over $j$ may be viewed as a Riemann sum for an integral from $0$ to $t$ where $\delta$ is the mesh size.  Using Lemma \ref{lem:continuityofsemigroup0}, we know that $T_t u$ is H\"older continuous in $t$.  Also, by Lebesgue differentiation theory, $P_\delta[(\Delta V) \phi + \nabla V \cdot \nabla \phi] \to (\Delta V) \phi + \nabla V \cdot \nabla \phi$ in $L_{\text{loc}}^1$.  Thus, in the limit, we obtain
\[
\int (T_t - 1)u \cdot \phi\,dx = \int_0^t \int \left( -\frac{N}{2} T_s u[ (\Delta V) \phi + \nabla V \cdot \nabla \phi] + \frac{1}{2N} T_s u\, \Delta \phi \right)\,dx\,ds.
\]
We pass from dyadic $t$ to all positive $t$ using Lemma \ref{lem:continuityofsemigroup0}.  Finally, after another integration by parts (which is justified by approximation by smooth functions in the appropriate Sobolev spaces), we have
\[
\int (T_t - 1)u \cdot \phi\,dx = \int_0^t \int \left( -\frac{N}{2} (\nabla T_s u \cdot \nabla V)\phi - \frac{1}{2N} \nabla T_s u \cdot \nabla \phi \right)\,dx\,ds.
\]
The asserted formula then follows by replacing $u$ with $T_{t_0}u$ and $t$ with $t_1 - t_0$.
\end{proof}

\begin{lemma}
If $\mu$ is the measure given by the potential $V$ and if $u$ is Lipschitz, then we have $\int T_t u\,d\mu = \int u\,d\mu$.
\end{lemma}

\begin{proof}
By applying Lemma \ref{lem:weaksolution} and approximating $(1/Z) \exp(-N^2 V(x))$ by compactly supported smooth functions, we see that
\begin{multline*}
\int T_t u\,d\mu - \int u \,d\mu \\
= \frac{1}{Z} \int \int_0^t \left[ -\frac{1}{2N} \nabla (T_s u) \cdot \nabla[e^{-N^2 V}] - \frac{N}{2} (\nabla V \cdot \nabla (T_su)) e^{-N^2 V}\right]\,ds\,dx
= 0.  \qedhere
\end{multline*}
\end{proof}

\begin{lemma} \label{lem:limitastgoestoinfinity0}
We have $T_t u(x) \to \int u\,d\mu$ as $t \to \infty$ and more precisely
\[
\left|T_t u(x) - \int u\,d\mu \right| \leq e^{-ct/4} \left( \frac{4Cm^{1/2}}{c^2} (6 + 5 \sqrt{2}) t^{-1/2} + \frac{2}{c} \norm{V(x)}_2 \right) \norm{u}_{\Lip}.
\]
\end{lemma}

\begin{proof}
Fix $t$ and fix $r \geq t$.  Let $n$ be an integer.  Then
\begin{align*}
|T_{t+r} u(x) - T_t u(x)| &\leq \sum_{j=0}^{n-1} |T_{t+r(j+1)/n}u(x) - T_{t+rj/n} u(x)| \\
&\leq \sum_{j=0}^{n-1} e^{-ct/2} e^{-crj/2n} \left( \frac{Cm^{1/2}}{c} (6 + 5 \sqrt{2}) (r/n)^{1/2} + \norm{V(x)}_2 (r/n) \right) \norm{u}_{\Lip} \\
&\leq e^{-ct/2} e^{cr/2n} \frac{2n}{cr} \left( \frac{Cm^{1/2}}{c} (6 + 5 \sqrt{2}) (r/n)^{1/2} + \norm{V(x)}_2 (r/n) \right) \norm{u}_{\Lip}
\end{align*}
Since $r \geq t$, we can choose $n$ such that $t / 4 \leq r/n \leq t/2$.  Then we have
\[
|T_{t+r} u(x) - T_t u(x)| \leq e^{-ct/4} \left( \frac{4Cm^{1/2}}{c^2} (6 + 5 \sqrt{2}) t^{-1/2} + \frac{2}{c} \norm{V(x)}_2 \right) \norm{u}_{\Lip}.
\]
Because this holds for all sufficiently large $r$, this shows that $\lim_{t \to \infty} T_t u(x)$ exists.  Because $\norm{T_tu}_{\Lip} \leq e^{-ct/2} \norm{u}_{\Lip}$, the limit must be constant and therefore equals $\int u\,d\mu$.  Moreover, we have the asserted rate of convergence by taking $r \to \infty$ in the above estimate.
\end{proof}

\subsection{Approximability and Convergence of Moments}

Now we are ready to show that the semigroup $T_t^{V_N}$ associated to a sequence of potentials $V_N$ will preserve asymptotic approximability by trace polynomials and as a consequence we will show that the moments of the associated measures $\mu_N$ converge.

\begin{lemma} \label{lem:asymptoticapproximation0}
Let $V_N: M_N(\C)_{sa}^m \to \R$ be a sequence of potentials such that $V_N(x) - (c/2) \norm{x}_2^2$ is convex and $V_N(x) - (C/2) \norm{x}_2^2$ is concave.  For each $N$, let $\mu_N$ be the associated measure.  Let $S_t^{V_N}$ and $T_t^{V_N}$ denote the semigroups defined in the previous section.  Suppose that the sequence $\{DV_N\}$ is asymptotically approximable by trace polynomials.  Suppose that $\{u_N\}$ is a sequence of scalar-valued $K$-Lipschitz functions which is asymptotically approximable by (scalar-valued) trace polynomials.  Then
\begin{enumerate}
	\item $\{S_t^{V_N} u_N\}$ is asymptotically approximable by trace polynomials for each $t \geq 0$.
	\item $\{T_t^{V_N} u_N\}$ is asymptotically approximable by trace polynomials for each $t \geq 0$.
	\item $\lim_{N \to \infty} \int u_N\,d\mu_N$ exists.
\end{enumerate}
\end{lemma}

\begin{proof}
(1) Recall that $S_t^{V_N}u_N = u_N(W_N(x,t))$, where $W_N$ is the solution to \eqref{eq:ODE}.  Thus, by Lemma \ref{lem:compositionapproximation}, it suffices to show that $W_N(x,t)$ is asymptotically approximable by trace polynomials for each $t$.  To this end, we write $W_N(x,t)$ as the limit as $\ell \to \infty$ of Picard iterates $W_{N,\ell}$ given by
\[
W_{N,0}(x,t) = x, \qquad W_{N,\ell+1}(x,t) = x - \frac{1}{2} \int_0^t DV_k(W_N(x,s))\,ds.
\]
Because $DV_N$ is $C$-Lipschitz, the standard Picard-Lindel\"of arguments show that
\[
\norm{W_{N,\ell}(x,t) - W_N(x,t)}_2 \leq \sum_{n=\ell+1}^\infty \frac{C^{n-1} t^n}{2^n n!} \norm{DV_N(x)}_2.
\]
Because $DV_N$ is asymptotically approximable by trace polynomials, we know that $\norm{DV_N(x)}_2$ is uniformly bounded on $\norm{x} \leq R$ for any given $R > 0$, and therefore for each $T$ and $R > 0$, the convergence of $W_{N,\ell}$ to $W_N$ as $\ell \to \infty$ is uniform for all $\norm{x} \leq R$ and $t \leq T$ and $N \in \N$.  Thus, by Observation \ref{lem:convergenceapproximation}, it suffices to show that each Picard iterate $\{W_{N,\ell}(x,t)\}_N$ is asymptotically approximable by trace polynomials.

Fix $T > 0$.  We claim that for every $\ell$, for every $R > 0$ and $\epsilon > 0$, there exists a trace polynomial $f(X,t)$ with coefficients that are polynomial functions of $t$, such that
\[
\limsup_{N \to \infty} \sup_{t \in [0,T]} \sup_{\norm{x}_\infty \leq R} \norm{W_{N,\ell}(x,t) - f(x,t)}_2 \leq \epsilon.
\]
We proceed by induction on $\ell$, with the base case $\ell = 0$ being trivial.  For the inductive step, fix $\epsilon$ and $R$, and choose a trace polynomial $f(X,t)$ which provides a $(\epsilon / CT, R)$ approximation for $W_{N,\ell}$ for all $t \leq T$.  Let
\[
R' = \sup_{t \in [0,T]} \sup_N \sup_{x \in M_N(\C)_{sa}^m: \norm{x}_\infty \leq R} \norm{f(x,t)} < +\infty.
\]
Choose another trace polynomial $g(X)$ which is an $(\epsilon / T, R')$ approximation for $\{DV_N\}$, and let $h(X,t) = X - \frac{1}{2} \int_0^t g(f(X,s))\,ds$.  Then arguing as in Lemma \ref{lem:compositionapproximation}, we have for $\norm{x} \leq R$ and $t \in [0,T]$ that
\begin{align*}
\norm{W_{N,\ell+1}(x,t) - h(x,t)} &\leq \frac{1}{2} \int_0^t \norm{DV_N(W_{N,\ell}(x,s)) - g(f(x,s))}_2 \,ds \\
&\leq \frac{t}{2} \sup_{\norm{y} \leq R'} \norm{DV_N(y) - g(y)}_2 + \frac{Ct}{2} \sup_{s \in [0,T]} \sup_{\norm{x} \leq R} \norm{W_{N,\ell}(x,s) - f(x,s)}_2.
\end{align*}
Taking $N \to \infty$, we see that $h(x,t)$ is an $(\epsilon,R)$ approximation for $\{W_{N,\ell}(x,t)\}_N$ for all $t \leq T$.

(2) We have shown that $S_t^{V_k}$ preserves asymptotic approximability.  Moreover, if the sequence $u_N: M_N(\C)_{sa}^m \to \C$ is asymptotically approximable by trace polynomials and $u_N$ is $K$-Lipschitz, then the sequence $P_t u_N$ is also asymptotically approximable by trace polynomials by Lemma \ref{lem:convolutionapproximation} (the hypothesis \eqref{eq:growthrate} is satisfied since $|u_N(x)| \leq |u_N(0)| + L \norm{x}_2$ and $|u_N(0)|$ is bounded as $N \to +\infty$ because $u_N$ is asymptotically approximable by trace polynomials).  Therefore, the iterated operator $T_{t,\ell}^{V_k} = (P_{2^{-\ell}} S_{2^{-\ell}}^{V_k})^{2^\ell t}$ preserves asymptotic approximability for dyadic values of $t$.  Taking $\ell \to \infty$, we see by Observation \ref{lem:convergenceapproximation} and Lemma \ref{lem:convergence0} that $T_t^{V_k}$ preserves asymptotic approximability for dyadic values of $t$.  Finally, we extend the approximability property to $T_t^{V_k}$ for all real $t$ using Observation \ref{lem:convergenceapproximation} and Lemma \ref{lem:continuityofsemigroup0}.

(3) We know by Lemma \ref{lem:limitastgoestoinfinity0} that $T_t^{V_N} u_N(x) \to \int u_N\,d\mu_N$ as $t \to \infty$ with estimates that are independent of $N$.  It follows by Observation \ref{lem:convergenceapproximation} that the sequence of constant functions $\{\int u_N\,d\mu_N\}$ is asymptotically approximable by trace polynomials.  But since these functions are constant, this simply means that the limit of $\int u_N\,d\mu_N$ as $N \to \infty$ exists.
\end{proof}

\begin{proof}[Proof of Theorem \ref{thm:convergenceandconcentration}]
(1)  Let $a_N = \int x\,d\mu_N(x)$ and $a_{N,j} = \int x_j\,d\mu_N(x)$.  Note that
\[
R_N \leq \max_j \int \norm{x_j - a_{N,j}}\,d\mu_N(x) + \max_j \left| \int \tau_N(x_j)\,d\mu_N(x) \right| + \max_j \norm*{ \int (x_j - \tau_N(x_j))\,d\mu(x) }.
\]
When we take the $\limsup$ as $N \to \infty$, the first term is bounded by $2 / c^{1/2}$ by Corollary \ref{cor:operatornormtail} while the last term is bounded by $M$.  It remains to estimate $\int \tau_N(x_j)\,d\mu_N(x)$.

Using integration by parts, we see that
\[
\int DV_N(x)\,d\mu_N(x) = 0
\]
On the other hand, we may estimate $\norm*{ DV_N(x) - \left( DV_N(0) + \frac{C + c}{2} x \right) }_2$ as follows.  We assumed that $V_N(x) - (c/2) \norm{x}_2^2$ is convex and $V_N(x) - (C/2) \norm{x}_2^2$ is concave.  Let $\tilde{V}_N(x) = V_N(x) - \frac{C+c}{2} \norm{x}_2^2$.  Then $\tilde{V}_N(x) + \frac{1}{2}(C-c) \norm{x}_2^2$ is convex and $\tilde{V}_N(x) - \frac{1}{2}(C - c) \norm{x}_2^2$ is concave.  Therefore, $D\tilde{V}_N$ is $(C - c)/2$-Lipschitz with respect to $\norm{\cdot}_2$.  It follows that
\[
\norm*{ DV_N(x) - \left( DV_N(0) + \frac{C + c}{2} x \right) }_2 = \norm{D\tilde{V}_N(x) - D\tilde{V}_N(0)}_2 \leq \frac{C - c}{2} \norm{x}_2.
\]
Therefore,
\begin{align*}
\norm*{ DV_N(0) + \frac{C+c}{2} a_N }_2 &\leq \frac{C - c}{2} \int \norm{x}_2 \,d\mu_N(x) \\
&\leq \frac{C - c}{2} \left( \norm{a_N}_2 + \left( \int \norm{x - a_N}_2^2\,d\mu(x) \right)^{1/2} \right) \\
&\leq \frac{C - c}{2} \left( \norm{a_N}_2 + c^{-1/2} \right),
\end{align*}
where the last step follows from Theorem \ref{thm:subgaussian}.  Altogether,
\[
\frac{C + c}{2} \norm{a_N}_2 \leq \frac{C - c}{2} \norm{a_N}_2 + \norm{DV_N(0)}_2 + \frac{C - c}{2c^{1/2}}.
\]
Then we move $(C - c)/ 2 \cdot \norm{a_N}_2$ to the left hand side and divide the equation by $c$ to obtain
\[
\left| \int \tau_N(x_j)\,d\mu_N(x) \right| \leq \norm{a_N}_2 \leq \frac{1}{c} \norm{DV_N(0)}_2 + \frac{C - c}{2c^{3/2}},
\]
which proves the asserted estimate on $R_N$.  The tail estimate on $\mu_N(\norm{x_j} \geq R_N + \delta)$ follows from Corollary \ref{cor:operatornormtail}.

(2) Fix a non-commutative polynomial $p$.  Let $R_* = \limsup_{N \in \N} R_N$ which we know is finite because of (1) and suppose that $R' > R_*$.  Let $\psi \in C_c^\infty(\R)$ be such that $\psi(t) = t$ for $|t| \leq R'$, and define $\Psi(x_1,\dots,x_m) = (\psi(x_1),\dots,\psi(x_m))$, where $\psi(x_j)$ is defined through the continuous functional calculus for self-adjoint operators.  Now $x \mapsto \psi(x)$ is Lipschitz in $\norm{\cdot}_2$ for $x \in M_N(\C)_{sa}$ with constants independent of $N$ (see for instance Proposition \ref{prop:Peller} below).  It follows that $p(\Psi(x))$ is globally Lipschitz in $\norm{\cdot}_2$ and it equals $p(x)$ when $\norm{x} \leq R'$.

Furthermore, we claim that the sequence $\tau_N(p(\Psi(x)))$ is asymptotically approximable by trace polynomials.  To see this, choose some radius $r$ and $\delta > 0$.  By the Weierstrass approximation theorem, there exists a polynomial $\widehat{\psi}(t)$ such that $|\psi(t) - \widehat{\psi}(t)| \leq \delta$ for $t \in [-r,r]$.  By the spectral mapping theorem, we have $\norm{\psi(y) - \widehat{\psi}(y)} \leq \delta$ if $y \in M_N(\C)_{sa}$ and $\norm{y} \leq r$.  In particular, if we let $\widehat{\Psi}(x) = (\widehat{\psi}(x_1),\dots,\widehat{\psi}(x_m))$ for $x \in M_N(\C)_{sa}^m$, then we have $\norm{\Psi(x) - \widehat{\Psi}(x)} \leq \delta$ when $\norm{x}_\infty \leq r$.  Given $\epsilon > 0$, we may choose $\delta$ small enough to guarantee that $|\tau_N(p(\Psi(x))) - \tau_N(p(\widehat{\Psi}(x)))| \leq \epsilon$ for $\norm{x}_\infty \leq r$, and clearly $\tau_N(p(\widehat{\Psi}(x)))$ is a trace polynomial.  Thus, $\tau_N(p(\Psi(x))$ is asymptotically approximable by trace polynomials.

Therefore, by Lemma \ref{lem:asymptoticapproximation0}, the limit
\[
\lambda(p) = \lim_{N \to \infty} \int \tau_N(p(\Psi(x))\,d\mu_N(x)
\]
exists.  Clearly, $\lambda$ satisfies all the conditions to be a non-commutative law.  Furthermore, because of the operator norm bounds (1), we know that $\int_{\norm{x} \geq R'} \tau_N(p(x))\,d\mu_N(x)$ is finite and approaches zero as $N \to \infty$ and the same holds for the integral of $\tau_N(p(\Psi(x)))$.  Therefore,
\[
\lim_{N \to \infty} \int \tau_N(p(x))\,d\mu_N(x) = \lim_{N \to \infty} \int \tau_N(p(\Psi(x))\,d\mu_N(x) = \lambda(p).
\]
Also, we have $\lambda(p) = \lim_{N \to \infty} \int_{\norm{x} \leq R'} \tau_N(p(x))\,d\mu_N(x)$ and hence $\lambda \in \Sigma_{m,R'}$.  But since this holds for every $R' > R_*$, we have $\lambda \in \Sigma_{m,R_*}$.

(3) It suffices to prove the concentration claim (3) for sufficiently large $R$, say $R > 2R_*$.  Because the topology of $\Sigma_{m,R}$ is generated by the functions $\lambda \mapsto \lambda(p)$ for non-commutative polynomials $p$, it suffices to consider the case where $\mathcal{U} = \{\lambda': |\lambda'(p) - \lambda(p)| < \epsilon\}$ for some non-commutative polynomial $p$.  Choose a function $\psi \in C_c^\infty(\R)$ with $\psi(t) = t$ for $|t| \leq R$, and let $\Psi$ be as above.  Then by Theorem \ref{thm:concentration},
\[
\mu_N\left(\left| \tau_N(p(\Psi(x)) - \int \tau_N (p \circ \Psi)\,d\mu_N \right| \geq \epsilon / 2 \right) \leq 2e^{-N^2\epsilon^2 / 8 \norm{\tau_N(p \circ \Psi)}_{\Lip}^2}.
\]
But by the same reasoning as in part (2), we know that large enough $N$, we have
\[
\left| \int \tau_N(p \circ \Psi)\,d\mu_N - \lambda(p) \right| \leq \frac{\epsilon}{2},
\]
and hence
\[
\limsup_{N \to \infty} \frac{1}{N^2} \log \mu_N\left(\norm{x} \leq R, \, \left| \tau_N(p(x)) - \lambda(p) \right| \geq \epsilon \right) < 0. \qedhere
\]
\end{proof}

\section{Entropy and Fisher's Information} \label{sec:entropy}

\subsection{Classical Entropy}

In this section, we will state sufficient conditions for the microstates free entropies $\chi$ and $\underline{\chi}$ to be evaluated as the $\limsup$ and $\liminf$ of renormalized classical entropies.  Recall that the (classical, continuous) entropy of a measure $d\mu(x) = \rho(x)\,dx$ on $\R^n$ is defined as
\[
h(\mu) := \int_{\R^n} -\rho \log \rho,
\]
whenever the integral makes sense.  We will later use the following basic facts about the classical entropy, so for convenience we provide a proof.

\begin{lemma} ~ \label{lem:classicalentropy}
\begin{enumerate}
	\item If $\mu$ has a density $\rho$ and $\int |x|^2\,d\mu(x) < +\infty$, then the positive part of $-\rho \log \rho$ has finite integral and hence $\int -\rho \log \rho$ is well-defined in $[-\infty,+\infty)$.
	\item In fact, we have $h(\mu) \leq (n/2) \log 2\pi ae$, where $a = \int |x|^2\,d\mu(x) / n$, and equality is achieved in the case of a centered Gaussian with covariance matrix $aI$.
	\item Suppose $\{\mu_k\}$ is a sequence of probability measures with density $\rho_k$, suppose $\rho_k \to \rho$ pointwise almost everywhere, and suppose that $\int |x|^2\,d\mu_k(x) \to \int |x|^2\,d\mu(x) < +\infty$.  Then $\limsup_{k \to \infty} h(\mu_k) \leq h(\mu)$.
	\item If $\mu$ and $\nu$ have finite second moments, then $h(\mu * \nu) \geq \max(h(\mu), h(\nu))$.
\end{enumerate}
\end{lemma}

\begin{proof}
(1) Let $a = \int |x|^2\,d\mu(x) / n$.  Let $g(x) = (2\pi a)^{-n/2} e^{-|x|^2/2a}$ be the Gaussian of variance $a$ and let $\gamma$ be the corresponding Gaussian measure.  Let $\tilde{\rho} = \rho / f$ be the density of $\mu$ relative to the Gaussian.  We write
\begin{align*}
-\rho(x) \log \rho(x) &= -\tilde{\rho}(x) \log \tilde{\rho}(x) \cdot g(x) - \tilde{\rho}(x) \log g(x) \cdot g(x) \\
&= -\tilde{\rho}(x) \log \tilde{\rho}(x) \cdot g(x) + \left(\frac{1}{2a} |x|^2 + \frac{n}{2} \log 2\pi a \right)\rho(x) .
\end{align*}
The second term has a finite integral by assumption.  The function $-t \log t$ is bounded above for $t \in \R$, and $g(x)$ is a probability density; thus, the positive part of $-\tilde{\rho} \log \tilde{\rho} \cdot g$ has finite integral.  Hence, $\int -\rho \log \rho$ is well-defined.

(2) The function $-t \log t$ is concave and its tangent line at $t = 0$ is $1 - t$ and hence $-t \log t \leq 1 - t$.  Thus,
\[
\int -\tilde{\rho} \log \tilde{\rho}\,d\gamma \leq \int (1 - \tilde{\rho})\,d\gamma = 0,
\]
so
\[
h(\mu) \leq \int \left(\frac{1}{2a} |x|^2 + \frac{n}{2} \log 2\pi a \right)\rho(x)\,dx = \frac{n}{2} + \frac{n}{2} \log 2 \pi a = \frac{n}{2} \log 2\pi e.
\]
In the case where $\mu = \gamma$, we have $\tilde{\rho} = 1$ and hence $\int - \tilde{\rho} \log \tilde{\rho} = 0$.

(3) Let $\gamma$ be the Gaussian of covariance matrix $I$ and $g$ its density.  Let $\tilde{\rho}_k = \rho_k / g$.  As before,
\[
h(\mu_k) = \int -\tilde{\rho}_k \log \tilde{\rho}_k\,d\gamma + \int \left( \frac{1}{2} |x|^2 + \frac{n}{2} \log 2\pi \right)\,d\mu_k.
\]
By assumption, the second term converges to $\int (\frac{1}{2} |x|^2 + \frac{n}{2} \log 2\pi)\,d\mu$.  Since the function $-t \log t$ is bounded above and $\gamma$ is a probability measure, the integral of the positive part of $-\tilde{\rho}_k \log \tilde{\rho}_k$ converges to the corresponding quantity for $\rho$.  For the negative part, we can apply Fatou's lemma.  This yields $\limsup_{k \to \infty} h(\mu_k) = h(\mu)$.

(4) We can assume without loss of generality that one of the measures, say $\mu$, has finite entropy and in particular has a density $\rho$.  Then $\mu * \nu$ has a density given almost everywhere by $\tilde{\rho}(x) = \int \rho(x-y)\,d\nu(y)$.  Since $-t\log t$ is concave, Jensen's inequality implies that
\[
-\tilde{\rho}(x) \log \tilde{\rho}(x) \geq \int -\rho(x-y) \log \rho(x-y)\,d\nu(y).
\]
The right hand side is $\int \int -\rho(x - y) \log \rho(x - y)\,d\nu(y)\,dx = \int \int -\rho(x - y) \log \rho(x - y)\,dx\,d\nu(y) = h(\mu)$, where the exchange of order is justified because we know that $-\rho \log \rho$ is integrable since $h(\mu) > -\infty$.  Therefore, $h(\mu * \nu) = \int -\tilde{\rho} \log \tilde{\rho} \geq h(\mu)$.
\end{proof}

\subsection{Microstates Free Entropy}

Because there is no integral formula known for free entropy of multiple non-commuting variables as in the classical case, Voiculescu defined the free analogue of entropy \cite{VoiculescuFE1,VoiculescuFE2} using Boltzmann's microstates viewpoint on entropy.

\begin{definition}
For $\mathcal{U} \subseteq \Sigma_m$, we define the \emph{microstate space}
\begin{align*}
\Gamma_N(\mathcal{U}) &= \{x \in M_N(\C)_{sa}^m: \lambda_x \in \mathcal{U}\} \\
\Gamma_{N,R}(\mathcal{U}) &= \{x \in M_N(\C)_{sa}^m: \lambda_x \in \mathcal{U}, \norm{x}_\infty \leq R\}.
\end{align*}
The \emph{microstates free entropy} of a non-commutative law $\lambda$ is defined as
\begin{align*}
\chi_R(\lambda) &= \inf_{\mathcal{U} \ni \lambda} \limsup_{N \to \infty} \left( \frac{1}{N^2} \log \vol \Gamma_{N,R}(\mathcal{U}) + \frac{m}{2} \log N \right) \\
\chi(\lambda) &= \sup_{R > 0} \chi_R(\lambda).
\end{align*}
Here $\mathcal{U}$ ranges over all open neighborhoods of $\lambda$ in $\Sigma_m$.  Similarly, we denote
\begin{align*}
\underline{\chi}_R(\lambda) &= \inf_{\mathcal{U} \ni \lambda} \liminf_{N \to \infty} \left( \frac{1}{N^2} \log \vol \Gamma_{N,R}(\mathcal{U}) + \frac{m}{2} \log N \right) \\
\underline{\chi}(\lambda) &= \sup_{R > 0} \chi_R(\lambda).
\end{align*}
\end{definition}

\begin{definition}
A sequence of probability measures $\mu_N$ on $M_N(\C)_{sa}^m$ is said to \emph{concentrate around the non-commutative law $\lambda$} if $\lambda_x \to \lambda$ in probability when $x$ is chosen according to $\mu_N$, that is, for any neighborhood $\mathcal{U}$ of $\lambda$ in $\Sigma_m$, we have
\[
\lim_{k \to \infty} \mu_N(x \in \Gamma_N(\mathcal{U})) = 1.
\]
\end{definition}

\begin{proposition} \label{prop:convergenceofentropy}
Let $V_N: M_N(\C)_{sa}^m \to \R$ be a potential with $\int \exp(-N^2 V_N(x))\,dx < +\infty$ and let $\mu_N$ be the associated measure.  Assume: 
\begin{enumerate}[(A)]
	\item The sequence $\{\mu_N\}$ concentrates around a non-commutative law $\lambda$.
	\item The sequence $\{V_N\}$ is asymptotically approximable by scalar-valued trace polynomials.
	\item For some $n \geq 1$ and $a, b > 0$ we have $|V_N| \leq a + b \sum_{j=1}^m \tau_N(x_j^{2n})$.
	\item There exists $R_0 > 0$ such that
	\[
	\lim_{N \to \infty} \int_{\norm{x}_\infty \geq R_0} \left(1 + \sum_{j=1}^m \tau_N(x_j^{2n}) \right) \,d\mu_N(x) = 0,
	\]
	where $n$ is the same number as in (C).
\end{enumerate}
Then $\lambda$ can be realized as the law of non-commutative random variables $X = (X_1,\dots,X_m)$ in a von Neumann algebra $(\mathcal{M},\tau)$ with $\norm{X_j} \leq R_0$.  Moreover, we have
\begin{align}
\chi(\lambda) &= \chi_{R_0}(\lambda) = \limsup_{N \to \infty} \left( \frac{1}{N^2} h(\mu_N) + \frac{m}{2} \log N \right) \label{eq:limsup} \\
\underline{\chi}(\lambda) &= \underline{\chi}_{R_0}(\lambda) = \liminf_{N \to \infty} \left( \frac{1}{N^2} h(\mu_N) + \frac{m}{2} \log N \right). \label{eq:liminf}
\end{align}
\end{proposition}

\begin{proof}
It follows from assumptions (A) and (D) that for every non-commutative polynomial $p$,
\[
\lim_{N \to \infty} \int_{\norm{x}_\infty \leq R_0} \tau_N(p(x))\,d\mu_N(x) = \lambda(p).
\]
It follows that $\lambda(X_j^{2n}) \leq R_0^{2n}$ for any $n > 0$.  From here it is a standard fact that $\lambda$ can be realized by self-adjoint random variables in a tracial von Neumann algebra which have norm $\leq R_0$.  

Now let us evaluate $\chi_R$ and $\underline{\chi}_R$ for $R \geq R_0$.  Recall that
\[
d\mu_N(x) = \frac{1}{Z_N} \exp(-N^2 V_N(x))\,dx, \qquad  Z_N = \int \exp(-N^2 V_N(x))\,dx,
\]
and note that
\[
h(\mu_N) = N^2 \int V_N(x)\,d\mu_N(x) + \log Z_N.
\]
The assumptions (C) and (D) imply that
\[
\lim_{N \to \infty} \int_{\norm{x}_\infty \geq R} |V_N(x)| \,d\mu_N(x) = 0 \text{ and } \lim_{N \to \infty} \mu_N(x: \norm{x}_\infty \geq R) = 0.
\]
Therefore, if we let
\[
d\mu_{N,R}(x) = \frac{1}{Z_{N,R}} \mathbf{1}_{\norm{x}_\infty \leq R} \exp(-N^2 V_N(x))\,dx, \qquad Z_{N,R} = \int_{\norm{x}_\infty \leq R} \exp(-N^2 V_N(x))\,dx.
\]
then as $N \to \infty$, we have
\[
\int V_N\,d\mu_N - \int V_N \,d\mu_{N,R} \to 0, \qquad \log Z_N - \log Z_{N,R} \to 0,
\]
and hence
\[
\frac{1}{N^2} h(\mu_N) - \frac{1}{N^2} h(\mu_{N,R}) \to 0.
\]

Fix $\epsilon > 0$.  By assumption (B), there is scalar-valued trace polynomial $f$ such that $|V_N(x) - f(x)| \leq \epsilon/2$ for $\norm{x}_\infty \leq R$ and for sufficiently large $N$.  Now because the trace polynomial $f$ is continuous with respect to convergence in non-commutative moments, the set $\mathcal{U} = \{\lambda': |\lambda'(f) - \lambda(f)| < \epsilon / 2\}$ is open.  Now suppose that $\mathcal{V} \subseteq \mathcal{U}$ is a neighborhood of $\lambda$.  Note that
\[
\lim_{N \to \infty} \mu_{N,R}(\Gamma_{N,R}(\mathcal{V})) = \lim_{N \to \infty} \frac{Z_N}{Z_{N,R}} \mu_N(\Gamma_N(\mathcal{V}) \cap \{x: \norm{x}_\infty \leq R\}) = 1,
\]
where we have used that $Z_N / Z_{N,R} \to 1$ as shown above, that $\mu_N(\Gamma_N(\mathcal{V})) \to 1$ by assumption (A), and that $\mu_N(\norm{x}_\infty \leq R) \to 1$ by assumption (D).  Moreover, by our choice of $f$ and $\mathcal{U}$, we have
\[
x \in \Gamma_{N,R}(\mathcal{V}) \implies |V_N(x) - \lambda(f)| \leq \epsilon.
\]
Therefore,
\begin{align*}
Z_{N,R} \mu_{N,R}(\Gamma_{N,R}(\mathcal{V})) &= \int_{\Gamma_{N,R}(\mathcal{V})} \exp(-N^2 V_N(x))\,dx \\
&= \vol \Gamma_{N,R}(\mathcal{V}) \exp(-N^2(\lambda(f) + O(\epsilon))).
\end{align*}
Thus,
\[
\log Z_{N,R} + \log \mu_{N,R}(\Gamma_{N,R}(\mathcal{V})) = \log \vol \Gamma_{N,R}(\mathcal{V}) - N^2 (\lambda(f) + O(\epsilon)).
\]
Meanwhile, note that $f$ is bounded by some constant $K$ whenever $\norm{x}_\infty \leq R$.  Therefore,
\begin{align*}
\int V_N\,d\mu_{N,R} &= \int_{\Gamma_{N,R}(\mathcal{V})} V_N\,d\mu_{N,R} + \int_{\Gamma_{N,R}(\mathcal{V}^c)} V_N\,d\mu_{N,R} \\
&= \int_{\Gamma_{N,R}(\mathcal{V})} \lambda[f] \,d\mu_{N,R} + \int_{\Gamma_{N,R}(\mathcal{V}^c)} \lambda_x[f] \,d\mu_{N,R} + O(\epsilon) \\
&= \lambda(f) \mu_{N,R}(\Gamma_{N,R}(\mathcal{V})) + O(\epsilon) + O\bigl( K \, \mu_N(\Gamma_{N,R}(\mathcal{V}^c)) \bigr).
\end{align*}
Altogether,
\begin{align*}
\frac{1}{N^2} h(\mu_{N,R}) =& \int V_N\,d\mu_{N,R} + \frac{1}{N^2} \log Z_{N,R} \\
=& \lambda(f) (\mu_{N,R}(\Gamma_{N,R}(\mathcal{V})) - 1) + \frac{1}{N^2} \log \vol \Gamma_{N,R}(\mathcal{V}) \\
&+ O(\epsilon) + O\bigl( K\, \mu_N(\Gamma_{N,R}(\mathcal{V}^c)) \bigr) - \frac{1}{N^2} \log \mu_{N,R}(\Gamma_{N,R}(\mathcal{V})).
\end{align*}
Now we apply the fact that $\mu_{N,R}(\Gamma_{N,R}(\mathcal{V})) \to 1$ to obtain
\[
\limsup_{N \to \infty} \frac{1}{N^2} |h(\mu_{N,R}) - \log \vol \Gamma_{N,R}(\mathcal{V})| = O(\epsilon).
\]
Because this holds for all sufficiently small neighborhoods $\mathcal{V}$ with the error $O(\epsilon)$ only depending on $\mathcal{U}$, we have
\begin{align*}
\chi_R(\lambda) &= \limsup_{N \to \infty} \left( \frac{1}{N^2} h(\mu_{N,R}) + \frac{m}{2} \log N \right) + O(\epsilon) \\
&=  \limsup_{N \to \infty} \left( \frac{1}{N^2} h(\mu_N) + \frac{m}{2} \log N \right) + O(\epsilon).
\end{align*}
Next, we take $\epsilon \to 0$ and obtain $\chi_R(\lambda) = \limsup_{N \to \infty} (N^{-2} \log h(\mu_N) + (m/2) \log N)$ for $R \geq R_0$.  Now $\chi(\lambda) = \sup_R \chi_R(\lambda)$ and $\chi_R(\lambda)$ is an increasing function of $R$.  Since our claim about $\chi_R(\lambda)$ holds for sufficiently large $R$, it also holds for $\chi(\lambda)$, so \eqref{eq:limsup} is proved.  The proof of \eqref{eq:liminf} is identical.
\end{proof}

\subsection{Classical Fisher Information}

The classical Fisher information of a probability measure $\mu$ on $\R^n$ describes how the entropy changes when $\mu$ is convolved with a Gaussian.  Suppose $\mu$ is given by the smooth density $\rho > 0$ on $\R^n$, and let $\gamma_t$ be the multivariable Gaussian measure on $\R^n$ with covariance matrix $tI$.  Then the density $\rho_t$ for $\mu_t = \mu * \gamma_t$ evolves according to the heat equation $\partial_t \rho_t = (1/2) \Delta \rho_t$.  Integration by parts shows that $\partial_t h(\mu_t) = (1/2) \int |\nabla \rho_t / \rho_t|^2 d\mu_t$ (which we justify in more detail below).

The \emph{Fisher information} of $\mu$ represents the derivative at time zero and it is defined as
\[
\mathcal{I}(\mu) := \int \left| \frac{\nabla \rho}{\rho} \right|^2 \,d\mu.
\]
The Fisher information is the $L^2(\mu)$ norm of the function $-\nabla \rho(x) / \rho(x)$, which is known as the \emph{score function}.  If $X$ is a random variable with smooth density $\rho$, then the $\R^n$-valued random variable $\Xi = -\nabla \rho(X) / \rho(X)$ satisfies the integration-by-parts relation
\begin{equation} \label{eq:classicalintegrationbyparts}
E[\Xi \cdot f(X)] = -\int \frac{\nabla \rho(x)}{\rho(x)} f(x) \rho(x)\,dx = \int \rho(x) \nabla f(x)\,dx = E[\nabla f(X)] \text{ for } f \in C_c^\infty(\R^n),
\end{equation}
or equivalently $E[\Xi_j f(X)] = E[\partial_j f(X)]$ for each $j$.

In fact, the integration-by-parts relation $E[\Xi \cdot f(X)] = E[\nabla f(X)]$ makes sense even if we do not assume that $X$ has a smooth density.  Following the terminology used by Voiculescu in the free case, if $X$ is an $\R^n$-valued random variable on the probability space $(\Omega,P)$, we say that an $\R^n$-valued random variable $\Xi \in L^2(\Omega,P)$ is a \emph{(classical) conjugate variable for $X$} if $E[\Xi \cdot f(X)] = E[\nabla f(X)]$ and if each $\Xi_j$ is in the closure of $\{f(X): f \in C_c^\infty(\R^n)\}$ in $L^2(\Omega,P)$.

In other words, this means that $\Xi$ is a function of $X$ (up to almost sure equivalence) and satisfies the integration-by-parts relation.  Since the integration-by-parts relation uniquely determines the $L^2(\Omega,P)$ inner product of $\Xi_j$ and $f(X)$ for all $f \in C_c^\infty(\R^n)$, it follows that the conjugate variable is unique (up to almost sure equivalence), and it is also independent of $(\Omega,P)$ and only depends on the law of $X$.  Thus, we may unambiguously define the \emph{Fisher information} $\mathcal{I}(\mu) = E[|\Xi|^2]$ if $\Xi$ is a conjugate variable to $X$ and $\mathcal{I}(\mu) = +\infty$ if no conjugate variable exists.

The probabilistic viewpoint is useful because it enables us to produce conjugate variables and estimate Fisher information using conditional expectation.  (See \cite[Proposition 3.7]{VoiculescuFE5} for the free case.)

\begin{lemma} \label{lem:conditionalexpectation}
Suppose that $X$ and $Y$ are independent $\R^n$-valued random variables with $X \sim \mu$ and $Y \sim \nu$.  If $\Xi$ is a conjugate variable for $X$, then $E[\Xi | X + Y]$ is a conjugate variable for $X + Y$.  In particular,
\[
\mathcal{I}(\mu * \nu) \leq \min(\mathcal{I}(\mu), \mathcal{I}(\nu)).
\]
\end{lemma}

\begin{proof}
Because $X$ and $Y$ are independent, we have for $g \in C_c^\infty(\R^n \times \R^n)$ that $E[\Xi_j g(X,Y)] = E[\partial_{X_j} g(X,Y)]$.  In particular, if $f \in C_c^\infty(\R^n)$, then
\[
E[\Xi_j f(X+Y)] = E[\partial_{X_j}(f(X+Y))] = E[(\partial_j f)(X+Y)].
\]
But $E[\Xi_j | X + Y]$ is the orthogonal projection onto the closed span of $\{f(X+Y): f \in C_c^\infty(\R^n)\}$ and hence
\[
E\left[ E[\Xi_j | X + Y] f(X + Y) \right] = E[\partial_j f(X + Y)].
\]
So $\mathcal{I}(\mu * \nu) = E[|E[\Xi|X+Y]|^2] \leq E[|\Xi|^2] = \mathcal{I}(\mu)$.  By symmetry, $\mathcal{I}(\mu * \nu) \leq \mathcal{I}(\nu)$.
\end{proof}

The entropy of a measure $\mu$ can be recovered by integrating the Fisher information of $\mu * \gamma_t$.  The following integral formula was the motivation for Voiculescu's definition of non-microstates free entropy $\chi^*$.  For the reader's convenience, we include a statement and proof in the random matrix setting with free probabilistic normalizations.  See also \cite[Lemma 1]{Barron1996} and \cite[Proposition 7.6]{VoiculescuFE5}.  Recall that we identify $M_N(\C)_{sa}^m$ with $\R^{mN^2}$ using the orthonormal basis given in \S \ref{subsec:matrixnotation} rather than entrywise coordinates (since some entries are real and some are complex).

\begin{lemma} \label{lem:integrateFisher}
Let $\mu$ be a probability measure on $M_N(\C)_{sa}^m$ with finite variance and with density $\rho$, and let $\sigma_{t,N}$ be the law of $m$ independent GUE's of normalized variance $t$.  If $a = (1/m) \int \norm{x}_2^2\,d\mu(x) = (1/mN) \int |x|^2\,d\mu(x)$, then we have for $t \geq 0$ that
\begin{equation} \label{eq:Fisherinfoestimate}
\frac{m}{a+t} \leq \frac{1}{N^3} \mathcal{I}(\mu * \sigma_{t,N}) \leq \min\left( \frac{m}{t}, \frac{1}{N^3} \mathcal{I}(\mu) \right).
\end{equation}
Moreover,
\begin{equation} \label{eq:integrateFisher}
\frac{1}{N^2} h(\mu * \sigma_{t,N}) - \frac{1}{N^2} h(\mu) = \frac{1}{2} \int_0^t \frac{1}{N^3} \mathcal{I}(\mu * \sigma_{s,N}) \,ds
\end{equation}
and
\begin{equation} \label{eq:integrateFisher2}
\frac{1}{N^2} h(\mu) + \frac{m}{2} \log N = \frac{1}{2} \int_0^\infty \left( \frac{m}{1+s} - \frac{1}{N^3} \mathcal{I}(\mu * \sigma_{s,N}) \right) ds + \frac{m}{2} \log 2\pi e.
\end{equation}
\end{lemma}

\begin{proof}
To prove \eqref{eq:Fisherinfoestimate}, suppose $t \geq 0$ and let $X$ and $Y$ be random variables with the laws $\mu$ and $\sigma_{t,N}$ respectively.  The lower bound is trivial if $\mathcal{I}(\mu * \sigma_{t,N}) = +\infty$, so suppose that $X + Y$ has a conjugate variable $\Xi$.  Then after some computation, the integration-by-parts relation shows that $E \ip{\Xi, X + Y}_{\Tr} = mN^2$.  Thus,
\[
E[|\Xi|^2] \geq \frac{|E\ip{\Xi,X+Y}_{\Tr}|^2}{E|X+Y|^2} = \frac{(mN^2)^2}{N(ma + mt)} = \frac{N^3}{a + t}
\]
since the variance of $Y$ with respect to the non-normalized inner product is $Nmt$ and the variance of $X$ is $Na$.  The upper bound is trivial in the case where $t = 0$.  If $t > 0$, then by the previous lemma $\mathcal{I}(\mu * \sigma_{t,N}) \leq \min(\mathcal{I}(\mu), \mathcal{I}(\sigma_{t,N})$.  Moreover, a direct computation shows that if $Y \sim \sigma_{t,N}$, then the conjugate variable is $(N/t) Y$ and the Fisher information is $mN^3 / t$.

Next, to prove \eqref{eq:integrateFisher}, let $\mu_t := \mu * \sigma_{t,N}$.  By basic properties of convolving positive functions with the Gaussian, $\mu_t$ has a smooth density $\rho_t$.  We claim that if $0 < \delta < t$, then
\[
h(\mu_t) - h(\mu_\delta) = \frac{1}{2N} \int_\delta^t \mathcal{I}(\mu_s)\,ds = \frac{1}{2N} \int_\delta^t \int_{M_N(\C)_{sa}^m} \frac{|\nabla \rho_s(x)|^2}{\rho_s(x)}\,dx\,ds.
\]
This will follow from integration by parts, but to give a complete justification, we first introduce a smooth compactly supported ``cutoff'' function $\psi_R: M_N(\C)_{sa}^m \to \R$ such that $0 \leq \psi_R \leq 1$ and $\psi_R(x) = 1$ when $|x| \leq R$ and $\psi_R(x) = 0$ when $|x| \geq 2R$.  Because of scale-invariance, we can arrange that $\norm{\nabla \psi_R(x)}_2 \leq C / R$.  Because $\partial_s \rho_s = (1/2N) \Delta \rho_s$, we have
\begin{align*}
\frac{d}{dt} \left[ -\int \psi_R \rho_s \log \rho_s \right] &= -\frac{1}{2N} \int \psi_R \cdot (\Delta \rho_s \log \rho_s + \Delta \rho_s) \\
&= \frac{1}{2N} \int \psi_R \frac{|\nabla \rho_s|^2}{\rho_s} + \frac{1}{2N} \int \nabla \psi_R \cdot \nabla \rho_s \cdot (1 + \log \rho_s),
\end{align*}
where all the integrals are taken over $M_N(\C)_{sa}^m$ with respect to $dx$.  This is equal to
\[
\frac{1}{2N} \int \psi_R |\nabla \rho_s / \rho_s|^2\,d\mu_s - \frac{1}{2N} \int (\nabla \psi_R \cdot \nabla \rho_s / \rho_s)(1 + \log \rho_s)\,d\mu_s.
\]
Of course, by the monotone convergence theorem
\[
\lim_{R \to +\infty} \int_\delta^t \int \psi_R |\nabla \rho_s / \rho_s|^2\,d\mu_s\,ds = \int_\delta^t \mathcal{I}(\mu_s)\,ds.
\]

The other term is an error which can be estimated as follows:  Note that $\mu_s = \mu * \sigma_{s,N}$ and that $\sigma_{s,N}$ has a density that is bounded uniformly for $s \in [\delta,t]$ and $x \in M_N(\C)_{sa}^m$.  Therefore, $\rho_s$ is uniformly bounded for $s \in [\delta,t]$ and $x \in M_N(\C)_{sa}^m$ and hence $\log \rho_s$ is uniformly bounded above.  To obtain a lower bound on $\log \rho_s$, first note that there is a $K > 0$ such that
\[
\mu(x: |x| \leq K) \geq 1/2.
\]
Now if $x \in M_N(\C)_{sa}^m$ and $|y| \leq K$, then $|x - y| \geq |x| - K$ and hence $|x - y|^2 \leq |x|^2 - 2K|x| + K^2 \geq 2|x|^2 + 2K^2$, where the last inequality follows because $2K|x| \leq (1/2)|x|^2 + 2K^2$ by the arithmetic geometric mean inequality.  Therefore, letting $Z$ be the normalizing constant for $\sigma_{t,N}$, we have
\begin{align*}
\rho_s(x) &= \frac{1}{Z} \int e^{-(N/2t)|x - y|^2} \,d\mu(y) \\
&\geq \frac{1}{Z} \int_{|y|\leq K} \int e^{-(N/2t)|x - y|^2} \,d\mu(y) \\
&\geq \frac{1}{Z} \int_{|y| \leq K} e^{-(N/t)(|x|^2 + K^2)}\,d\mu(y) \\
&\geq \frac{e^{-NK^2/t}}{2Z} e^{-(N/t) |x|^2},
\end{align*}
so that $\log \rho_s \geq K' - |x|^2$ for some constant $K'$.  In particular, combining our upper and lower bounds, there is a constant $\alpha$ such that for sufficiently large $x$, we have $|1 + \log \rho_s| \leq \alpha |x|^2$.  Recall that $\nabla \psi_R(x)$ is supported when $R \leq |x| \leq 2R$ and thus is bounded by $1 / R \sim 1 / |x|$.   Altogether we have $|\nabla \psi_R (1 + \log \rho_s)| \leq \beta |x|$ for some constant $\beta$ when $|x|$ is large enough.  Thus, our error term is bounded by
\begin{align*}
\int_\delta^t \int |(\nabla \psi_R \cdot \Xi_s)(1 + \log \rho_s)|\,d\mu_s\,ds &\leq \beta \int_\delta^t \int_{|x| \geq R} |x| |\nabla \rho_s(x) / \rho_s(x)|\,d\mu_s(x)\,ds \\
&\leq \frac{1}{2} \beta \int_\delta^t \int_{|x| \geq R} (|x|^2 + |\nabla \rho_s(x) / \rho_s(x)|^2)\,d\mu_s(x)\,ds.
\end{align*}
The right hand is the tail of the convergent integral
\[
\int_\delta^t \int  (|x|^2 + |\nabla \rho_s(x) / \rho_s(x)|^2)\,d\mu_s(x)\,ds = \int_\delta^t [(a + ms) + \mathcal{I}(\mu_s)]\,ds < +\infty,
\]
and therefore it goes to zero as $R \to +\infty$ by the dominated convergence theorem.  We can also apply the dominated convergence theorem to $-\int \psi_R \rho_t \log \rho_t$ and $-\int \psi_R \rho_\delta \log \rho_\delta$ given our earlier estimate that $\rho_s$ is subquadratic for each $s$.  The result is that
\[
h(\mu_t) - h(\mu_\delta) = \frac{1}{2N} \int_\delta^t \int |\nabla \rho_s / \rho_s|^2\,d\mu_s \,ds = \frac{1}{2N} \int_\delta^t \mathcal{I}(\mu_s)\,ds.
\]
To complete the proof of \eqref{eq:integrateFisher}, we must take $\delta \searrow 0$.  We can take the limit of the right hand side by the monotone convergence theorem.  As for the left hand side, Lemma \ref{lem:classicalentropy} (3) implies that $\limsup_{\delta \searrow 0} h(\mu_\delta) \leq h(\mu)$ because $\rho_\delta \to \rho$ almost everywhere by Lebesgue differentiation theory.  On the other hand, $h(\mu_\delta) \geq h(\mu)$ by Lemma \ref{lem:classicalentropy} (4), hence $h(\mu_\delta) \to h(\mu)$, so \eqref{eq:integrateFisher} is proved.

To prove \eqref{eq:integrateFisher2}, we follow \cite[Proposition 7.6]{VoiculescuFE5}.  First, suppose that $h(\mu) > -\infty$.  Note that
\[
h(\mu) = \frac{1}{2} \int_0^t \left( \frac{mN^2}{1 + s} - \frac{1}{N} \mathcal{I}(\mu_s) \right)\,ds - \frac{mN^2}{2} \log(1 + t) + h(\mu_t).
\]
If $h(\mu) > -\infty$, then $\int_0^1 \left( \frac{mN^2}{1+s} - \frac{1}{N} \mathcal{I}(\mu_s) \right)\,ds$ is finite.  In light of \eqref{eq:Fisherinfoestimate}, the integral from $1$ to $+\infty$ is also finite and by the dominated convergence theorem
\[
\lim_{t \to + \infty} \frac{1}{2} \int_0^t \left( \frac{mN^2}{1 + s} - \frac{1}{N} \mathcal{I}(\mu_s) \right)\,ds = \frac{1}{2} \int_0^\infty \left( \frac{mN^2}{1 + s} - \frac{1}{N} \mathcal{I}(\mu_s) \right)\,ds.
\]
It remains to understand the behavior of $h(\mu_t) - (mN^2/2) \log(1+t)$.  By Lemma \ref{lem:classicalentropy} (4) and (2),
\[
h(\mu_t) \geq h(\sigma_{t,N}) = \frac{mN^2}{2} \log \frac{2\pi et}{N} = \frac{mN^2}{2} \log \frac{2\pi e}{N} + \frac{mN^2}{2} \log t.
\]
On the other hand, by Lemma \ref{lem:classicalentropy} (2), since $\int |x|^2\,d\mu_t(x) = N(a + tm)$, we have
\[
h(\mu_t) \leq \frac{mN^2}{2} \log \frac{2\pi e(a+t)}{N} =  \frac{mN^2}{2} \log \frac{2\pi e}{N} + \frac{mN^2}{2} \log (a + t).
\]
As $t \to \infty$, we have $\log (1 + t) - \log (a + t) \to 0$ and $\log (1 + t) - \log t \to 0$ and therefore
\[
h(\mu_t) - \frac{mN^2}{2} \log(1 + t) \to \frac{mN^2}{2} \log \frac{2\pi e}{N} = \frac{mN^2}{2} \log 2\pi e - \frac{mN^2}{2} \log N.
\]
Hence,
\[
h(\mu) = \frac{1}{2} \int_0^\infty \left( \frac{mN^2}{1 + s} - \frac{1}{N} \mathcal{I}(\mu_s) \right)\,ds + \frac{mN^2}{2} \log 2\pi e - \frac{mN^2}{2} \log N,
\]
which is equivalent to the asserted formula \eqref{eq:integrateFisher2}.  In the case where $h(\mu) = -\infty$, we also have $\int_0^1 \left( \frac{mN^2}{1 + s} - \frac{1}{N} \mathcal{I}(\mu_s) \right)\,ds = -\infty$ by \eqref{eq:integrateFisher}, but the integral from $1$ to $\infty$ is finite as shown above.  So both sides of \eqref{eq:integrateFisher2} are $-\infty$.
\end{proof}

\subsection{Free Fisher Information}

The starting point for the definition of free Fisher information is the integration-by-parts formula \eqref{eq:classicalintegrationbyparts}.  Indeed, if we formally apply this to a non-commutative polynomial $p$ and renormalize, we obtain
\begin{equation} \label{eq:noncommutativeintegrationbyparts}
\int \tau_N\left(\frac{1}{N} \Xi_j(x) p(x)\right)\,d\mu(x) = \int \tau_N \otimes \tau_N(\mathcal{D}_jp(x))\,d\mu(x),
\end{equation}
(and this integration by parts is justified under sufficient assumptions of finite moments).  Voiculescu therefore made the following definitions:

\begin{definition}[{\cite[\S 3]{VoiculescuFE5}}] \label{def:freeFisher}
Let $X = (X_1,\dots,X_m)$ be a tuple of self-adjoint random variables in a tracial von Neumann algebra $(\mathcal{M},\tau)$ and assume that $\mathcal{M}$ is generated by $X$ as a von Neumann algebra.  We say that $\xi = (\xi_1,\dots,\xi_m) \in L^2(\mathcal{M},\tau)^m$ is the \emph{(free) conjugate variable} of $X$ if
\begin{equation} \label{eq:freeintegrationbyparts}
\tau(\xi_j p(X)) = \tau \otimes \tau(\mathcal{D}_j p(X))
\end{equation}
for every non-commutative polynomial $p$.  In this case, we say that $X$ (or equivalently the law of $X$) has finite free Fisher information and define $\Phi^*(X) := \Phi^*(\lambda_X) := \sum_j \tau(\xi_j^2)$.  We also denote the conjugate variable $\xi$ by $J(X)$.
\end{definition}

\begin{definition}[{\cite[Definition 7.1]{VoiculescuFE5}}]
The \emph{non-microstates free entropy} of a non-commutative law $\lambda$ is
\[
\chi^*(\lambda) := \frac{1}{2} \int_0^\infty \left( \frac{m}{1+t} - \Phi^*(\lambda \boxplus \sigma_t) \right) + \frac{1}{2} \log 2\pi e.
\]
\end{definition}

Now we are ready to state conditions under which the classical Fisher information of a sequence of measures $\mu_N$ converges to the free Fisher information of the law $\lambda$.  First, to clarify the normalization, note that if $d\mu_N(x) = (1/Z_N) \exp(-N^2 V_N(x))\,dx$, then the classical conjugate variable is given by $\Xi_N = N^2 \nabla V_N$.  The normalized conjugate variable used in \eqref{eq:noncommutativeintegrationbyparts} is $(1/N) \Xi_N = N \nabla V_N = DV_N$.  The corresponding normalized Fisher information is then
\[
\int \norm{DV_N}_2^2\,d\mu_N = \int \frac{1}{N} \left| \frac{1}{N} \Xi_N \right|^2\,d\mu = \frac{1}{N^3} \mathcal{I}(\mu_N),
\]
which is the same normalization as in Lemma \ref{lem:integrateFisher}.

\begin{proposition} \label{prop:convergenceofFisherinfo}
Let $V_N: M_N(\C)_{sa}^m \to \R$ be a potential with $\int \exp(-N^2 V_N(x))\,dx < +\infty$ and let $\mu_N$ be the associated measure.  Assume:
\begin{enumerate}[(A)]
	\item The sequence $\mu_N$ concentrates around a non-commutative law $\lambda$.
	\item The sequence $\{DV_N\}$ is asymptotically approximable by trace polynomials.
	\item For some $n \geq 0$ and $a, b > 0$ we have $\norm{DV_N}_2^2 \leq a + b \sum_{j=1}^m \tau_N(x_j^{2n})$.
	\item There exists $R_0 > 0$ such that
	\[
	\lim_{N \to \infty} \int_{\norm{x}_\infty \geq R_0} \left(1 + \sum_{j=1}^m \tau_N(x_j^{2n}) \right) \,d\mu_N(x) = 0.
	\]
\end{enumerate}
Then
\begin{enumerate}
	\item The law $\lambda$ can be realized by self-adjoint random variables $X = (X_1,\dots,X_m)$ in a tracial von Neumann algebra $(\mathcal{M},\tau)$ with $\norm{X_j} \leq R_0$.
	\item There exists a sequence of trace polynomials $f^{(k)} \in (\TrP_m^1)^m$ such that
	\[
	\lim_{k \to \infty} \limsup_{N \to \infty} \sup_{\norm{x}_\infty \leq R_0} \norm{DV_N(x) - f^{(k)}(x)}_2 = 0.
	\]
	\item If $\{f^{(k)}\}$ is any sequence as in (2), then $\{f_k(X)\}$ converges in $L^2(\mathcal{M},\tau)$ and the limit is the conjugate variable $J(X)$.
	\item The law $\lambda$ has finite free Fisher information and $N^{-3} \mathcal{I}(\mu_N) \to \Phi^*(\lambda)$ as $N \to \infty$.
\end{enumerate}
\end{proposition}

\begin{proof}
(1) This follows from the same argument as Proposition \ref{prop:convergenceofentropy}.

(2) This follows from the definition of asymptotic approximability by trace polynomials.

(3) Let $\{f^{(k)}\}$ be a sequence as in (2).  Because $\mu_N$ concentrates around $\lambda$ and because $\mu_N(\{x: \norm{x}_\infty \leq R_0\}) \to 1$ as $N \to +\infty$ by (4), we have
\[
\lambda[(f^{(j)} - f^{(k)})^*(f^{(j)} - f^{(k)})] = \lim_{N \to \infty} \int_{\norm{x}_\infty \leq R_0} \tau_N[(f^{(j)} - f^{(k)})^*(f^{(j)} - f^{(k)})(x)]\,d\mu_N(x).
\]
For every $\epsilon > 0$, if $j$ and $N$ are large enough, then $\sup_{\norm{x}_\infty \leq R_0} \norm{DV_N(x) - f^{(j)}(x)}_2 < \epsilon$ by our assumption on $f^{(j)}$.  In particular, if $j$ and $k$ are sufficiently large, then $\lambda[(f^{(j)} - f^{(k)})^*(f^{(j)} - f^{(k)})] < (2\epsilon)^2$.  This shows that $\{f^{(k)}(X)\}$ is Cauchy in $L^2(M,\lambda)$ since $X$ has the law $\lambda$.

Let $\xi = \lim_{k \to \infty} f^{(k)}(X)$.  We must show that $\xi$ is the conjugate variable for $X$.  Let $\psi \in C_c^\infty(\R)$ such that $\psi(y) = y$ when $|y| \leq R_0$.  For $x \in M_N(\C)_{sa}^m$, let $\Psi(x) = (\psi(x_1), \dots, \psi(x_m))$.  By \eqref{eq:noncommutativeintegrationbyparts}, because $DV_N(x)$ is the classical conjugate variable for $X$, we have for every non-commutative polynomial $p$ that
\[
\int \tau_N[D_jV_N(x) \cdot p(\Psi(x))]\,d\mu_N(x) = \int D_j[\tau_N(p(\Psi(x)))]\,d\mu_N(x).
\]
It follows from our assumptions (C) and (D) that
\[
\lim_{N \to \infty} \int_{\norm{x}_\infty \geq R_0} \norm{DV_N(x)}_2^2\,d\mu_N(x) = 0.
\]
Because $p(\Psi(x))$ and $D_j[\tau_N(p(\Psi(x)))]$ are globally bounded in operator norm, the integral of these quantities over $\norm{x}_\infty \geq R_0$ will vanish as $N \to \infty$ and therefore
\[
\int_{\norm{x}_\infty \leq R_0} \tau_N[D_jV_N(x) p(\Psi(x))]\,d\mu_N(x) - \int_{\norm{x}_\infty \leq R_0} D_j[\tau(p(\Psi(x)))]\,d\mu_N(x) \to 0
\]
But since $p(\Psi(x)) = p(x)$ on this region, we have
\[
\int_{\norm{x}_\infty \leq R_0} \tau_N[D_jV_N(x) p(x)]\,d\mu_N(x) - \int_{\norm{x}_\infty < R_0} \tau_N \otimes \tau_N[\mathcal{D}_jp(x)] \,d\mu_N(x) \to 0.
\]
Now the second term converges to $\lambda \otimes \lambda[\mathcal{D}_j p] = \tau \otimes \tau[\mathcal{D}_jp(X)]$ by our concentration assumption (A).  For the first term, we can replace $D_j V_N(x)$ by $f_j^{(k)}(x)$ with an error bounded by $\sup_{\norm{x}_\infty \leq R_0} \norm{f^{(k)}(x) - DV_N(x)}_2$.  Then we apply concentration to conclude that $\int \tau_N[f_j^{(k)}(x)^* p(x)]\,d\mu_N(x) \to \lambda[(f_j^{(k)})^*p]$.  Overall,
\[
\left| \lambda[(f_j^{(k)})^* p] - \lambda \otimes \lambda[\mathcal{D}_j p] \right| \leq \limsup_{N \to \infty} \sup_{\norm{x}_\infty \leq R_0} \norm{f^{(k)}(x) - DV_N(x)}_2.
\]
Taking $k \to \infty$, we obtain $\tau[\xi_j p(X)] - \tau \otimes \tau[\mathcal{D}_j p(X)] = 0$ as desired.

(4) We know from (3) that $\lambda$ has finite Fisher information.  Assumptions (C) and (D) imply that
\[
\frac{1}{N^3} \mathcal{I}(\mu_N) - \int_{\norm{x}_\infty \leq R_0} \norm{DV_N(x)}_2^2 \,d\mu_N(x) \to 0.
\]
By similar arguments as before, we can approximate $DV_N$ by $f^{(k)}$ on $\norm{x}_\infty \leq R_0$, approximate $\int_{\norm{x}_\infty \leq R_0} \norm{f^{(k)}}_2^2\,d\mu_N$ by $\lambda((f^{(k)})^* f^{(k)})$, and then approximate $\lambda((f^{(k)})^* f^{(k)})$ by $\tau(\xi^*\xi) = \Phi^*(\lambda)$, where the error terms vanish as $N \to \infty$ and then $k \to \infty$.  This implies that $N^{-3} \mathcal{I}(\mu_N) \to \Phi^*(\lambda)$.
\end{proof}

\section{Evolution of the Conjugate Variables} \label{sec:evolution}

\subsection{Motivation and Statement of the Equation}

In the last section, we stated conditions under which the classical entropy and Fisher information of $\mu_N$ converge to their free counterparts for the limiting non-commutative law $\lambda$.  In order to prove that $\chi(\lambda) = \chi^*(\lambda)$, we want to take the limit in the integral formula \eqref{eq:integrateFisher2}, and therefore, we want $N^{-3} \mathcal{I}(\mu_N * \sigma_{t,N}) \to \Phi^*(\lambda \boxplus \sigma_t)$ for all $t > 0$.  In order to apply Proposition \ref{prop:convergenceofFisherinfo} to $\mu_N * \sigma_{t,N}$, we need to show that $\{DV_{N,t}\}_N$ is asymptotically approximable by trace polynomials, where $V_{N,t}$ is the potential corresponding to $\mu_N * \sigma_{t,N}$.

By adding a constant to each $V_N$ if necessary, we may assume without loss of generality that $Z_N = 1$.  We call that $V_{N,t}(x)$ is given by
\begin{equation} \label{eq:Gaussianconvolutionequation}
\exp(-N^2 V_{N,t}(x)) = \int \exp(-N^2 V_N(x+y)) \,d\sigma_{t,N}(y). 
\end{equation}
Then $\exp(-N^2 V_{N,t}(x))$ solves the normalized heat equation
\begin{equation} \label{eq:heatequation}
\partial_t[\exp(-N^2 V_{N,t}(x))] = \frac{1}{2N} \Delta [\exp(-N^2 V_{N,t}(x))],
\end{equation}
where $(1/N) \Delta = L_N$ is the normalized Laplacian.  However, we do not know how to show that $DV_N(\cdot,t)$ is asymptotically approximable by trace polynomials from a direct analysis of the heat equation because of the dimension-dependent factor of $N^2$ in the exponent.  What we want is a dimension-independent and ``hands-on'' way of producing $V_{N,t}$ from $V_N$.

As in \S \ref{sec:convergenceandconcentration}, we will analyze the PDE which describes the evolution of the function $V_{N,t}$.  We first derive the equation by rewriting the \eqref{eq:heatequation} in terms of $V_{N,t}$ rather than $e^{-N^2 V_{N,t}}$.  By the chain rule,
\[
\partial_t[\exp(-N^2V_{N,t})] = -N^2 \partial_t V_{N,t} \cdot \exp(-N^2V_{N,t})
\]
and
\begin{align*}
\Delta[\exp(-N^2V_{N,t})] &= [\Delta(-N^2V_{N,t}) + |\nabla(-N^2V_{N,t})|^2] \exp(-N^2V_{N,t}) \\
&= (-N^2 \Delta V_{N,t} + N^4 |\nabla V_{N,t}|^2) \exp(-N^2 V),
\end{align*}
where $\Delta$ and $\nabla$ denote the classical (non-normalized) Laplacian and gradient, where $M_N(\C)_{sa}^m$ has been identified with $\R^{mN^2}$ using the coordinates in \ref{subsec:matrixnotation}.  Thus, our equation becomes
\begin{align*}
-N^2 \partial_t V_{N,t} &= \frac{1}{2N}(-N^2 \Delta V_{N,t} + N^4 |\nabla V_{N,t}|^2) \\
\partial_t V_{N,t} &= \frac{1}{2N} \Delta V_{N,t} - N |\nabla V_{N,t}|^2.
\end{align*}
Recall that $(1/N) \Delta$ is the normalized Laplacian discussed in \S \ref{subsec:TPdiff}.  The normalized gradient is $DV_{N,t} = N \nabla V_{N,t}$, and the normalized Euclidean norm is $\norm{x}_2^2 = \sum_{j=1}^m \tau_N(x_j^2) = \frac{1}{N} \sum_{j=1}^m \Tr(x_j^2) = \frac{1}{N} |x|^2$.  Then
\[
N|\nabla V_{N,t}|^2 = \frac{1}{N} |N \nabla V_{N,t}|^2 = \frac{1}{N} |DV_{N,t}|^2 = \norm{DV_{N,t}}_2^2.
\]
and therefore we obtain the following equation that is normalized in a dimension-independent way
\begin{equation} \label{eq:normalizedevolution}
\partial_t V_{N,t} = \frac{1}{2} L_N V_{N,t} - \frac{1}{2} \norm{DV_{N,t}}_2^2.
\end{equation}

In the remainder of this section, we study a semigroup $R_t$ acting on convex and semi-concave functions on $M_N(\C)_{sa}^m$ such that $V_{N,t} = R_t V_N$ (here $R_t$ depends implicitly on $N$).  In \S \ref{subsec:iteration} - \S \ref{subsec:continuity1}, we construct $R_t$ from scratch by iterating the heat semigroup and Hopf-Lax semigroup.  Next, in \S \ref{subsec:viscosity}, we verify that $R_t V_N$ solves \eqref{eq:normalizedevolution} in the \emph{viscosity sense} (for background, see \cite{CIL1992}), and deduce that $R_t V_N$ must agree with the smooth solution $V_{N,t}$ defined by \eqref{eq:Gaussianconvolutionequation}.  Finally, in (\S \ref{subsec:approximation1}), we show that if $\{DV_N\}$ is asymptotically approximable by trace polynomials, then so is $\{D(R_t V_N)\}$.

\subsection{Strategy to Approximate Solutions} \label{subsec:iteration}

To construct the semigroup $R_t$ that solves \eqref{eq:normalizedevolution}, we view the equation as a hybrid between the heat equation $\partial_t u = (1/2N) \Delta u$ and the Hamilton-Jacobi equation with quadratic potential $\partial_t u = -(1/2) \norm{D u}_2^2$.  The heat equation can be solved by the heat semigroup
\begin{equation}
P_t u(x) := \int u(x + y)\,d\sigma_{t,N}(y),
\end{equation}
while the Hamilton-Jacobi equation can be solved using the inf-convolution semigroup
\begin{equation}
Q_t u(x) := \inf \left[u(x + y) + \frac{1}{2t} \norm{y}_2^2 \right]
\end{equation}
as a special case of the Hopf-Lax formula (see \cite[Chapter 3.3]{Evans}).

In Dabrowski's approach, the solution to \eqref{eq:normalizedevolution} was expressed through a formula of Bou\'e, Dupuis and \"Ustunel as the infimum of $E [u(x + B_t + \int_0^t Y_s\,ds) + (1/2) \int_0^t \norm{Y_s}_2^2\,ds]$ over a certain class of stochastic processes $Y_t$ adapted to a standard Brownian motion $B_t$ (see \cite[Theorem 3.1]{Dabrowski2017}).  This formula, roughly speaking, combines the Gaussian convolution and inf-convolution operations by replacing the $y$ in the definition of $Q_t$ by a stochastic process and allowing it to evolve with $B_t$.  Dabrowski then identifies the minimizing process $Y_t$ as a Brownian bridge \cite[Section 5]{Dabrowski2017} and analyzes it using a forward-backward SDE.  Through the Picard iteration solving the SDE, he shows that the solution is well-approximated by non-commutative functions.

We instead give a deterministic proof following the same strategy as in \S \ref{sec:convergenceandconcentration} that is motivated by Trotter's formula, we define a semigroup $R_t u$ at dyadic times $t$ by alternating between $P_{2^{-\ell}}$ and $Q_{2^{-\ell}}$ and then letting $\ell \to \infty$.  We establish convergence through a telescoping series argument after showing that $P_t Q_t - Q_t P_t = o(t)$.  Then we show that $R_t u$ depends continuously on $t$ in order to extend its definition to all positive real $t$.

In contrast to \S \ref{sec:convergenceandconcentration}, we must understand how the semigroups $P_t$, $Q_t$, and $R_t$ affect $Du$ as well as $u$, and we want $D(R_t u)$ to be Lipschitz for all $t$.  We therefore view these operators as acting on spaces of the form
\[
\mathcal{E}(c,C) = \left\{u: M_N(\C)_{sa}^m \to \R,  u(x) - \frac{1}{2}c \norm{x}_2^2 \text{ is convex and } u(x) - \frac{1}{2} C \norm{x}_2^2 \text{ is concave} \right\},
\]
where $0 \leq c \leq C < +\infty$, where we suppress the dependence on $m$ and $N$ in the notation. These spaces have the virtue that if $u \in \mathcal{E}(c,C)$, then $\norm{Du}_{\Lip} \leq C$ automatically (see Proposition \ref{prop:convex} (3)).

At every step of the proof, we include estimates both for $u$ and for $Du$.  In addition, controlling the error propagation requires more work because $Q_t$ and $R_t$ are not contractions with respect to $\norm{Du}_{L^\infty}$.

The following theorem summarizes the results of the construction.

Here, for a measurable function $u: M_N(\C)_{sa}^m \to \R$, the notation $\norm{u}_{L^\infty}$ is the standard $L^\infty$ norm.  If $F: M_N(\C)_{sa}^m \to M_N(\C)_{sa}^m$ (for instance $F = Du$ for some $u: M_N(\C)_{sa}^m \to \R$, then $\norm{F}_{L^\infty} = \sup_{x \in M_N(\C)_{sa}^m} \norm{F(x)}_2$; similarly, $\norm{F}_{\Lip}$ is the Lipschitz norm of $F$ when using $\norm{\cdot}_2$ in both the domain and the target space.

Note that $\norm{F}_2$ does \emph{not} denote the $L^2$ norm of $F$ with respect to any measure, but rather $(\sum_{j=1}^m \tau(F_j^2))^{1/2}$, which is a function of $x$.  We denote $\N = \{1,2,3,\dots\}$ and $\N_0 = \{0,1,2,\dots\}$.  We also denote by $\Q_2^+ = \bigcup_{n \geq 0} 2^{-n} \N_0$ the nonnegative dyadic rationals.  Moreover, we assume throughout the section that $0 \leq c \leq C < +\infty$.

\begin{theorem} \label{thm:semigroup}
There exists a semigroup of nonlinear operators $R_t: \bigcup_{C > 0} \mathcal{E}(0,C) \to \bigcup_{C > 0} \mathcal{E}(0,C)$ with the following properties:
\begin{enumerate}
	\item {\bf Change in Convexity:} If $u \in \mathcal{E}(c,C)$ where $0 \leq c \leq C$, then $R_t u \in \mathcal{E}(c(1+ct)^{-1},C(t+Ct)^{-1})$.
	\item {\bf Approximation by Iteration:} For $\ell \in \Z$ and $t \in 2^{-\ell} \N_0$, denote $R_{t,\ell}u = (P_{2^{-\ell}} Q_{2^{-\ell}})^{2^\ell t} u$.  Suppose $t \in \Q_2^+$ and $u \in \mathcal{E}(0,C)$.
		\begin{enumerate}[(a)]
			\item If $2^{-\ell-1} C \leq 1$, then
			\[
			|R_t u - R_{t,\ell} u| \leq \left( \frac{3}{2} \frac{C^2mt}{1+Ct} + \log(1 + Ct) (m + Cm + \norm{Du}_2^2) \right) 2^{-\ell}.
			\]
			\item $\displaystyle \norm{D(R_{t,\ell} u) - D(R_t u)}_{L^\infty} \leq [t/2 + C(t/2)^2] C^2 m^{1/2}(2 \cdot 2^{-\ell/2} + 2^{-3\ell/2}C)$.
		\end{enumerate}
	\item {\bf Continuity in Time:} Suppose $s \leq t \in \R_+$ and $u \in \mathcal{E}(0,C)$.
		\begin{enumerate}[(a)]
			\item $R_t u \leq R_s u + \frac{m}{2} [\log(1 + Ct) - \log(1 + Cs)]$.
			\item $R_t u \geq R_s u - \frac{1}{2} (t - s)(Cm + \norm{Du}_2^2)$.
			\item If $C(t - s) \leq 1$, then $\norm{D(R_t u) - D(R_s u)}_2 \leq 5Cm^{1/2} 2^{1/2}(t - s)^{1/2} + C(t - s) \norm{Du}_2$.
		\end{enumerate}
	\item {\bf Error Estimates:} Let $t \in \R_+$ and $u, v \in \mathcal{E}(0,C)$.  Then
		\begin{enumerate}[(a)]
			\item $\norm{D(R_t u) - D(R_t v)}_{L^\infty} \leq (1 + Ct) \norm{Du - Dv}_{L^\infty}$.
			\item If $u \leq v + a + b \norm{Dv}_2^2$ where $a \in \R$ and $b \geq 0$, then
			\[
			R_t u \leq R_t v + a + b\frac{C^2mt}{1+Ct} + b \norm{D(R_tv)}_2^2.
			\]
			\item We have
			\[
			\norm{D(R_tu)}_2^2 \leq \frac{C^2mt}{1+Ct} + \norm{Du}_2^2.
			\]
		\end{enumerate}
\end{enumerate}
\end{theorem}

\begin{remark}
Knowing that $\exp(-N^2(R_tu)) = P_t \exp(-N^2u)$, one can deduce (1) from the Braskamp-Lieb and H\"older inequalities, as in \cite[Theorem 4.3]{BL1976}.  But the proof of (1) given here is independent of \cite{BL1976}.

We also point out that the ideas of semigroups and discrete-time approximation schemes have been employed to study Hamilton-Jacobi equations in Hilbert space (e.g.\ by \cite{Barbu1986}).
\end{remark}

\subsection{The Hopf-Lax Semigroup, the Heat Semigroup, and Convexity}

We remind the reader of our standing assumption that $0 \leq c \leq C$.

\begin{lemma} \label{lem:Gaussianconvexity}
Suppose $u \in \mathcal{E}(c,C)$.  Then
\begin{enumerate}
	\item $P_t u \in \mathcal{E}(c,C)$.
	\item $\norm{D(P_tu) - Du}_{L^\infty} \leq Cm^{1/2} t^{1/2}$.
\end{enumerate}
\end{lemma}

\begin{proof}
(1) follows because $\mathcal{E}(c,C)$ is closed under translation and averaging, hence convolution by a probability measure.

(2) We know that $Du$ is $C$-Lipschitz and thus
\begin{align*}
\norm{D(P_tu)(x) - Du(x)}_2 &\leq \int \norm{Du(x+y) - Du(x)}_2 \,d\sigma_{t,N}(y) \\
&\leq \int C\norm{y}_2\,d\sigma_{t,N}(y) \\
&\leq Cm^{1/2} t^{1/2}. \qedhere
\end{align*}
\end{proof}

The following lemma gives basic properties of $Q_t$ from the PDE literature; see for instance \cite[p.\ 309-311]{EL1980}, \cite{LL1986}, \cite[Lemma A.5]{CIL1992}, \cite[Section 3.3.2]{Evans}.  For completeness and convenience, we include a proof of all the facts we will use.

\begin{lemma} \label{lem:monotonicityandquadratic} ~
\begin{enumerate}
	\item If $u, v: M_N(\C)_{sa}^m \to \R$ and $u \leq v$, then $P_t u \leq P_t v$ and $Q_t u \leq Q_t v$.
	\item Suppose that $v(x) = a + \ip{p,x}_2 \frac{1}{2} \ip{Ax,x}_2$ where $a \in \R$, $p \in M_N(\C)_{sa}^m$, and $A$ is a positive semi-definite linear map $M_N(\C)_{sa}^m \to M_N(\C)_{sa}^m$.  Then
	\begin{align*}
	P_t v(x) &= a + \frac{t}{2N^2} \Tr(A) + \ip{p,x} + \frac{1}{2} \ip{Ax,x}_2, \\
	Q_t v(x) &= a - \frac{t}{2} \norm{p}^2 + \ip{p,x} + \frac{1}{2} \ip{A(1 + tA)^{-1}(x - tp),x - tp}.
	\end{align*}
\end{enumerate}
\end{lemma}

\begin{remark} \label{rem:trace}
The meaning of $\Tr(A)$ In the above formula is as follows.  Using the identification of $M_N(\C)_{sa}^m$ to $\R^{mN^2}$ given by \S \ref{subsec:matrixnotation}, we can express $A$ as an $mN^2 \times mN^2$ and compute its trace in this way.  Alternatively, since $A$ is a linear transformation of the real inner product space $M_N(\C)_{sa}^m$, we may compute $\Tr(A)$ using an orthonormal basis of $M_N(\C)_{sa}^m$.  Because the trace is similarity-invariant, this answer is independent of the choice of basis (and also independent of the choice of normalization for the inner product).  Note that the trace of the identity is $mN^2$, which makes the normalization in the above formula dimension-independent.
\end{remark}

\begin{proof}
(1) is immediate to check from the definition.  We leave the first formula of (2) as an exercise.  To prove the last formula, fix $t > 0$ and $x \in M_N(\C)_{sa}^m$ and note that $u(y) + \frac{1}{2t} \norm{y - x}_2^2$ is a uniformly convex function of $y$ and therefore it has a unique minimizer.  The minimizer $y$ must be a critical point and hence
\[
0 = Du(y) + \frac{1}{t}(y - x) = p + Ay + \frac{1}{t}(y - x).
\]
Thus, $(1 + tA)y = x - tp$ and $y - x = -t(p + Ay)$.  Thus,
\begin{align*}
Q_tu(x) &= u(y) + \frac{1}{2t} \norm{y - x}_2^2 \\
&= a + \ip{p,y} + \frac{1}{2}\ip{Ay,y} - \frac{1}{2} \ip{p+Ay,y-x} \\
&= a + \frac{1}{2} \ip{p,y} + \frac{1}{2} \ip{p+Ay,x} \\
&= a + \ip{p,x} + \frac{1}{2} \ip{p,y-x} + \frac{1}{2} \ip{Ay,x} \\
&= a + \ip{p,x} - \frac{t}{2} \ip{p,p+Ay} + \frac{1}{2} \ip{Ay,x} \\
&= a - \frac{t}{2} \norm{p}^2 + \ip{p,x} + \frac{1}{2} \ip{Ay,x - tp} \\
&= a - \frac{t}{2} \norm{p}^2 + \ip{p,x} + \frac{1}{2} \ip{A(1 + tA)^{-1}(x - tp),x - tp} \qedhere
\end{align*}
\end{proof}

\begin{lemma} \label{lem:infconvolution}
Let $u \in \mathcal{E}(c,C)$ and $t \in \R^+$,
\begin{enumerate}
	\item The operators $\{Q_t\}_{t \geq 0}$ form a semigroup, that is, $Q_s Q_t u = Q_{s+t} u$.
	\item For each $x_0 \in M_N(\C)_{sa}^m$, the infimum $Q_t u(x_0) = \inf_y [u(y) + \frac{t}{2} \norm{y - x_0}_2^2$ is achieved at a unique point $y_0$ satisfying $y_0 = x_0 - t Du(y_0)$.
	\item If $x_0 \in M_N(\C)_{sa}^m$ and $y_0$ is the minimizer from (1), then $D(Q_tu)(x_0) = Du(y_0)$.
	\item We have $Q_t u \in \mathcal{E}(c(1+ct)^{-1},C(1+Ct)^{-1})$.
	\item $\norm{D(Q_tu)(x_0)}_2 = \norm{Du(y_0)}_2 \leq (1 + ct)^{-1} \norm{D u(x_0)}_2$.
\end{enumerate}
\end{lemma}

\begin{proof}
(1) By definition
\begin{align*}
Q_s Q_t u(x) &= \inf_y [Q_t u(y) + \frac{1}{2s} \norm{x - y}_2^2] \\
&= \inf_y \inf_z \left[u(z) + \frac{1}{2t} \norm{y - z}_2^2 + \frac{1}{2s} \norm{x - y}_2^2 \right] \\
&= \inf_z \left[u(z) + \inf_y \left[ \frac{1}{2t} \norm{y - z}_2^2 + \frac{1}{2s} \norm{x - y}_2^2 \right] \right].
\end{align*}
But note that $\inf_y [(1/2t) \norm{y - z}_2^2 + (1/2s) \norm{x - y}_2^2]$ is by definition $Q_s f(z)$, where $f(x) = (1/2t) \norm{x - z}_2^2$.  If $g(x) = (1/2t) \norm{x}_2^2$, then by the previously lemma, we have
\[
Q_s g(x) = \frac{1}{2}\frac{t^{-1}}{1 + t^{-1}s} \norm{x}_2^2 = \frac{1}{2(s+t)} \norm{x}_2^2.
\]
Since $Q_s$ is clearly translation-invariant, $Q_s f(x) = 1/2(s + t) \cdot \norm{x - z}_2^2$.  Therefore,
\[
Q_s Q_t u(x) = \inf_z \left[ u(z) + \frac{1}{2(s+t)} \norm{x - z}_2^2 \right] = Q_{s+t} u(x).
\]

(2) Fix $x_0$.  Note that the function $y \mapsto \left[u(y) + \frac{1}{2t} \norm{y - x_0}_2^2 \right]$ is in $\mathcal{E}(c+1/2t, C + 1/2t)$ and hence it achieves a global minimum at the unique critical point.  Thus, the infimum is achieved at the point $y_0$ satisfying $Du(y_0) = (1/t) (y_0 - x_0)$, or in other words $y_0 = x_0 - t Du(y_0)$.

(3) and (4)  Let $x_0$ and $y_0$ be as above.  Let $p = Du(y_0)$.  Because $u \in \mathcal{E}(c,C)$, we have for all $x$ that
\[
u(y_0) + \ip{p, x - y_0}_2 + \frac{c}{2} \ip{x - y_0} \leq u(x) \leq u(y_0) + \ip{p, x - y_0}_2 + \frac{C}{2} \norm{x - y_0}_2^2
\]
Let $\underline{v}(y)$ and $\overline{v}(y)$ be the functions on the left and right hand sides.  Then by Lemma \ref{lem:monotonicityandquadratic} (1), we have $Q_t \underline{v} \leq Q_t u \leq Q_t \overline{v}$.  To compute $Q_t \underline{v}$, we apply Lemma \ref{lem:monotonicityandquadratic} (2) with $A = cI$ and with a change of coordinates to translate $y_0$ to the origin, and we obtain
\[
Q_t \underline{v}(x) = u(y_0) - \frac{t}{2} \norm{p}_2^2 + \ip{p,x-y_0} + \frac{1}{2} c(1 + ct)^{-1} \norm{x - y_0 - tp}^2.
\]
Since $y_0 = x_0 + tp$ and $p = (y_0 - x_0)/t$, this becomes
\begin{align*}
Q_t \underline{v}(x) &= u(y_0) - \frac{t}{2} \norm{p}_2^2 + t \norm{p}^2 + \ip{p,x-x_0} + \frac{1}{2} c(1 + ct)^{-1} \norm{x - x_0}_2^2 \\
&= u(y_0) + \frac{1}{2t} \norm{y_0 - x_0}_2^2 + \ip{p,x-x_0} + \frac{1}{2} c(1 + ct)^{-1} \norm{x - x_0}_2^2 \\
&= Q_t u(x_0) + \ip{p,x-x_0} + \frac{1}{2} c(1 + ct)^{-1} \norm{x - x_0}_2^2.
\end{align*}
The analogous computation holds for $Q_t \overline{v}$ as well.  Thus, we have
\[
Q_t u(x_0) + \ip{p,x-x_0} + \frac{1}{2} c(1 + ct)^{-1} \norm{x - x_0}_2^2 \leq Q_t u(x) \leq Q_t u(x_0) + \ip{p,x-x_0} + \frac{1}{2} C(1 + Ct)^{-1} \norm{x - x_0}_2^2.
\]
This inequality implies that $D(Q_t u)(x_0) = p = Du(y_0)$.  Since the above inequality holds for every $x_0$, we see that $Q_t u \in \mathcal{E}(c(1+ct)^{-1}, C(1 + Ct)^{-1})$.

(5) Let $x_0$, $y_0$, and $p$ be as above.  Then we have
\[
\ip{Du(y_0) - Du(x_0), y_0 - x_0}_2 \geq c \norm{y_0 - x_0}_2^2.
\]
But recall that $y_0 - x_0 = - t Du(y_0)$ and hence
\[
- t \ip{Du(y_0) - Du(x_0), Du(y_0)} \geq ct^2 \norm{Du(y_0)}_2^2.
\]
Rearranging produces
\[
(1 + ct) \norm{Du(y_0)}_2^2 \leq \ip{Du(x_0), Du(y_0)}_2 \leq \norm{Du(x_0)}_2 \norm{Du(y_0)}_2,
\]
and hence $(1 + ct) \norm{Du(y_0)}_2 \leq \norm{Du(x_0)}_2$ as desired.
\end{proof}

\begin{corollary} \label{cor:infconvolution} Let $u \in \mathcal{E}(c,C)$ and $s, t \geq 0$.
\begin{enumerate}
	\item For each $x$, the gradient $D(Q_t u)(x)$ is the unique vector $p$ satisfying $p = D u(x - tp)$.
	\item We have $Q_t u(x) = u(x - t D(Q_tu)(x)) + \frac{t}{2} \norm{D(Q_tu)(x)}_2^2$.
	\item $u(x) - \frac{t}{2}(1 + Ct) \norm{D(Q_tu)(x)}^2 \leq Q_t u(x) \leq u(x) - \frac{t}{2}(1 + ct) \norm{D(Q_tu)(x)}_2^2$.
\end{enumerate}
\end{corollary}

\begin{proof}
(1) and (2) follow from Lemma \ref{lem:infconvolution} (2) and (3).

To prove (3), fix $x$ and let $y = x - t D(Q_t u)(x)$.  By Proposition \ref{prop:convex} (2),
\[
u(y) + \ip{Du(y), x - y}_2 + \frac{c}{2} \norm{x - y}_2^2 \leq u(x) \leq u(y) + \ip{Du(y), x - y} + \frac{C}{2} \norm{x - y}_2^2.
\]
Hence,
\[
u(x) - \ip{Du(y), x - y}_2 - \frac{C}{2} \norm{x - y}_2^2 \leq u(y) \leq u(x) - \ip{Du(y), x - y}_2 - \frac{c}{2} \norm{x - y}_2^2.
\]
But from the previous lemma, we know that $Du(y) = D(Q_t u)(x)$ and $x - y = t D(Q_t u)(x)$, so that
\[
u(x) - t \norm{D(Q_t u)(x)}_2^2 - \frac{C}{2} t^2 \norm{D(Q_t u)(x)}_2^2 \leq u(y) \leq u(x) - t \norm{D(Q_t u)(x)}_2^2 - \frac{c}{2} \norm{D(Q_t u)(x)}_2^2.
\]
Finally, we substitute $Q_t u(x) = u(y) + (t/2) \norm{D(Q_t u)(x)}_2^2$ and obtain (3).
\end{proof}

\subsection{Estimates for Error Propagation}

To prepare for our iteration, we first prove some estimates that will help control the propagation of errors.

\begin{lemma} \label{lem:initialerrorestimate}
If $u, v \in \mathcal{E}(c,C)$, then we have
\begin{enumerate}
	\item $\norm{D(P_tu) - D(P_tv)}_{L^\infty} \leq \norm{Du - Dv}_{L^\infty}$.
	\item $\norm{D(Q_tu) - D(Q_tv)}_{L^\infty} \leq (1 + Ct) \norm{Du - Dv}_{L^\infty}$.
\end{enumerate}
\end{lemma}

\begin{proof}
The first inequality follows because $D(P_tu) - D(P_tv)$ is the convolution of $Du - Dv$ with the Gaussian density.  To prove the second inequality, note that
\begin{align*}
\norm{D(Q_tu)(x) - D(Q_tv)(x)}_2
&= \norm{D u(x - t D (Q_tu)(x)) - D v(x - t D (Q_tv)(x)}_2 \\
&\leq \norm{D u(x - t D (Q_tv)(x)) - D v(x - t D (Q_tv)(x)}_2 \\
&\quad + \norm{D u(x - t D (Q_tu)(x)) - D u(x - t D (Q_tv)(x)}_2 \\
&\leq \norm{D u - Dv}_{L^\infty} + Ct \norm{D(Q_tu)(x) - D(Q_tv)(x)}_2,
\end{align*}
where the last inequality follows because $Du$ is $C$-Lipschitz.  This implies that for $t < 1/C$,
\[
\norm{D(Q_tu) - D(Q_tv)}_{L^\infty} \leq (1 - Ct)^{-1} \norm{D u - D v}_{L^\infty}.
\]
Now we improve the estimate using the semigroup property of $Q_t$.  Fix a positive integer $k$ and for $j = 1, \dots, k$, let $t_j = tj/k$, and let $C_j = C(1 - Ct_j)^{-1}$.  Then $Q_{t_j}u$ and $Q_{t_j}$ are in $\mathcal{E}(0,C_j)$.  Thus, we have
\[
\norm{D(Q_{t_{j+1}}u) - D(Q_{t_{j+1}}v)}_{L^\infty} \leq (1 - C_jt/k)^{-1} \norm{D(Q_{t_j}u) - D(Q_{t_j}v)}_{L^\infty},
\]
and hence
\[
\norm{D(Q_tu) - D(Q_tv)}_{L^\infty} \leq \norm{Du - Dv}_{L^\infty} \prod_{j=0}^{k-1} \frac{1}{1 - C_jt/k}.
\]
Now
\begin{align*}
\log \prod_{j=0}^{k-1} \frac{1}{1 - C_j t/k} &= \sum_{j=0}^{k-1} -\log(1 - C_jt/k) \\
&= \sum_{j=0}^{k-1} \left( C_jt/k + O(1/k^2) \right) \\
&= \sum_{j=0}^{k-1} \frac{C}{1 + Ct_j}(t_{j+1} - t_j) + O(1/k) \\
&= \int_0^t \frac{C}{1 + Cs}\,ds + O(1/k) \\
&= \log(1 + Ct) + O(1/k).
\end{align*}
Hence,
\[
\norm{D(Q_tu) - D(Q_tv)}_{L^\infty} \leq (1 + Ct + O(1/k))\norm{Du - Dv}_{L^\infty},
\]
and the proof is completed by taking $k \to \infty$.
\end{proof}

\begin{lemma} \label{lem:initialerrorinequality}
Suppose that $u: M_N(\C)_{sa}^m \to \R$ be convex and let $v \in \mathcal{E}(c,C)$ and $u \leq v + a + b \norm{Dv}_2^2$ for some $a \in \R$ and $b \geq 0$.
\begin{enumerate}
	\item $P_t u \leq P_t v + a + b C^2 mt + b \norm{D(P_tv)}_2^2$.
	\item $Q_t u \leq Q_t v + a + b \norm{D(Q_tv)}_2^2$.
\end{enumerate}
\end{lemma}

\begin{proof}
(1) Using monotonicity and linearity of $P_t$, we have
\[
P_t u \leq P_t v + a + b \int \norm{Dv(x+y)}_2^2\,d\sigma(y).
\]
So it suffices to show that
\[
\int \norm{Dv(x+y)}_2^2\,d\sigma_{t,N}(y) - \norm{D(P_tv)(x)}_2^2 \leq b C^2mt.
\]
In probabilistic terms, the left hand side is the variance of the random variable $Dv(x+Y)$ where $Y \sim \sigma_{t,N}$.  Since the variance is translation-invariant, this is the same as the variance of $Dv(x+Y) - Dv(x)$, and this is bounded above by the second moment
\[
E\norm{Dv(x+Y) - Dv(x)}_2^2 \leq C^2 \cdot E\norm{Y}_2^2 = C^2 mt.
\]

(2) Note that
\begin{align*}
Q_t u(x) &= \inf_y [u(y) + \frac{1}{2t} \norm{y - x}_2^2] \\
&\leq u(x - t D(Q_tv)(x)) + \frac{t}{2} \norm{D(Q_tv)(x)}_2^2 \\
&\leq v(x - t D(Q_tv)(x)) + \frac{t}{2} \norm{D(Q_tv)(x)}_2^2 + a + b \norm{Dv(x - tD(Q_tv)(x))}_2^2 \\
&= Q_t v(x) + a + b \norm{D(Q_tv)(x)}_2^2,
\end{align*}
where the last equality follows from Corollary \ref{cor:infconvolution} (1) and (2).
\end{proof}

\begin{lemma} \label{lem:gradgrowth}
Let $u \in \mathcal{E}(0,C)$.  Then
\begin{enumerate}
	\item $\norm{D(Q_tu)}_2^2 \leq \norm{Du}_2^2$.
	\item $\norm{D(P_tu)}_2^2 \leq C^2 m t + \norm{Du}_2^2$.
\end{enumerate}
\end{lemma}

\begin{proof}
The first claim follows from Lemma \ref{lem:infconvolution} (5).  To prove the second claim, note that by Minkowski's inequality,
\[
\norm{D(P_tu)(x)}_2^2 = \norm*{ \int Du(x+y)\,d\sigma_{t,N}(y) }_2^2 \leq \int \norm{Du(x+y)}_2^2\,d\sigma_{t,N}(y) \leq C^2mt + \norm{Du(x)}_2^2,
\]
where the last inequality was shown in the proof of Lemma \ref{lem:initialerrorinequality} (1).
\end{proof}

Next, we iterate the previous inequalities to obtain our main lemma on error propagation.

\begin{lemma} \label{lem:errorpropagation}
Let $t_1$,\dots,$t_n > 0$ and write
\begin{align*}
t^* &= t_1 + \dots + t_n \\
R &= P_{t_n} Q_{t_n} \dots P_{t_1} Q_{t_1}
\end{align*}
Let $u, v \in \mathcal{E}(c,C)$.
\begin{enumerate}
	\item $Ru, Rv \in \mathcal{E}(c(1+ct^*)^{-1},C(1+Ct^*)^{-1})$.
	\item $\norm{D(Ru) - D(Rv)}_{L^\infty} \leq (1 + Ct^*) \norm{Du -Dv}_{L^\infty}$.
	\item If $u \leq v + a + b \norm{Dv}_2^2$ with $a \in \R$ and $b \geq 0$, then we have
	\[
	Ru \leq Rv + a + b \frac{C^2mt^*}{1 + Ct^*} + b \norm{D(Rv)}_2^2.
	\]
	In particular, $u \leq v$ implies $Ru \leq Rv$.
	\item We have
	\begin{align*}
	\norm{D(Ru)}_2^2 &\leq \frac{C^2mt^*}{1 + Ct^*} + \norm{Du}_2^2 \\
	&\leq Cm + \norm{Du}_2^2.
	\end{align*}
\end{enumerate}
\end{lemma}

\begin{proof}
(1) Let $u \in \mathcal{E}(c,C)$.  Let $t_j^* = t_1 + \dots + t_j$ and $u_j = P_{s_j} Q_{t_j} \dots P_{s_1} Q_{t_1} u$.  We show by induction that $u_j \in \mathcal{E}(c(1+ct_j^*)^{-1}, C(1 + Ct_j^*)^{-1})$.  The base case $j = 0$ is trivial.  For the induction step, note that
\[
\frac{c(1 + ct_j^*)^{-1}}{1 + [c(1+ct_j^*)^{-1}]t_{j+1}} = \frac{c}{(1 + ct_j^*) + ct_{j+1}} = c(1 + ct_{j+1}^*)^{-1}
\]
and the same holds for $C$.  Hence, by Lemma \ref{lem:infconvolution} (4), if $u_j \in \mathcal{E}(c(1+ct_j^*)^{-1},C(1+Ct_j^*)^{-1})$, then $Q_{t_{j+1}} u_j \in \mathcal{E}(c(1+ct_{j+1}^*)^{-1}, C(1+Ct_{j+1}^*)^{-1})$.  By Lemma \ref{lem:Gaussianconvexity}, this implies that $u_{j+1} = P_{t_{j+1}} Q_{t_{j+1}} u_j \in \mathcal{E}(c(1+ct_{j+1}^*)^{-1}, C(1+Ct_{j+1}^*)^{-1})$.  The same argument of course applies to $v$.

(2)  Let $t_j^*$ and $u_j$ be as in the proof of (1) and define $v_j$ similarly to $u_j$.  We show by induction that $\norm{Du_j - Dv_j}_{L^\infty} \leq (1 + Ct_j^*) \norm{Du - Dv}_{L^\infty}$.  The base case $j = 0$ is trivial.  For the induction step, recall that $u_j, v_j \in \mathcal{E}(c(1+ct_j^*)^{-1},C(1+Ct_j^*)^{-1})$ and hence by Lemma \ref{lem:initialerrorestimate} and the induction hypothesis,
\begin{align*}
\norm{D(Q_{t_{j+1}}u_j) - D(Q_{t_{j+1}} v_j)}_{L^\infty} &\leq (1 + C(1+Ct_j^*)^{-1} t_{j+1}) \norm{Du_j - Dv_j}_{L^\infty} \\
&\leq (1 + C(1+Ct_j^*)^{-1} t_{j+1})(1 + Ct_j^*) \norm{Du - Dv}_{L^\infty} \\
&= (1 + Ct_{j+1}^*) \norm{Du - Dv}_{L^\infty}.
\end{align*}
Then by Lemma \ref{lem:initialerrorestimate} again, since $u_{j+1} = P_{t_{j+1}} Q_{t_{j+1}} u_j$ and $v_{j+1} = P_{t_{j+1}} Q_{t_{j+1}} v_j$, we have
\[
\norm{Du_{j+1} - Dv_{j+1}}_{L^\infty} \leq (1 + Ct_{j+1}^*) \norm{Du - Dv}_{L^\infty}.
\]

(3) First, we show by induction on $j$ that
\[
u_j \leq v_j + a + b \sum_{i=1}^j \frac{C^2mt_i}{(1 + Ct_i^*)^2} + b \norm{Dv_j}_2^2.
\]
The base case $j = 0$ is trivial.  If the claim holds for $u_j$ and $v_j$, then it also holds for $Q_{t_{j+1}} u_j$ and $Q_{t_{j+1}} v_j$ by Lemma \ref{lem:initialerrorinequality} (2).  Then we apply Lemma \ref{lem:initialerrorinequality} (1) together with the fact that $Q_{t_{j+1}}u_j$ and $Q_{t_{j+1}}v_j$ are in $\mathcal{E}(c(1+ct_{j+1}^*)^{-1},C(t+Ct_{j+1}^*)^{-1})$ to conclude that
\[
u_{j+1} \leq v_{j+1} + a + b \sum_{i=1}^{j+1} \frac{C^2mt_i}{(1 + Ct_i^*)^2} + b \norm{Dv_{j+1}}_2^2.
\]
This completes the induction.  Finally, we observe that $\sum_{i=1}^n C^2mt_i / (1 + Ct_i^*)^2$ is the lower Riemann sum for the function $C^2 m / (1 + Ct)^2$ on the interval $[0,t^*]$ with respect to the partition $\{0,t_1^*,\dots,t_n^*\}$.  Thus,
\[
\sum_{i=1}^n \frac{C^2mt_i}{(1 + Ct_i^*)^2} \leq \int_0^{t^*} \frac{C^2m}{(1 + Ct)^2}\,dt = Cm \left( 1 - \frac{1}{1 + Ct^*} \right) = \frac{C^2mt^*}{1 + Ct^*}.
\]
This shows the main claim of (3), and the claim that $u \leq v$ implies $Ru \leq Rv$ is the special case when $a = 0$ and $b = 0$.

(4) By Lemma \ref{lem:gradgrowth}, we have $\norm{D(Q_{t_{j+1}}u_j)}_2^2 \leq \norm{Du_j}_2^2$ and
\[
\norm{Du_{j+1}}_2^2 \leq \frac{C^2mt_{j+1}}{1 + Ct_{j+1}^*} + \norm{D(Q_{t_{j+1}} u_j)}_2^2 \leq \frac{C^2mt_{j+1}}{1 + Ct_{j+1}^*} + \norm{Du_j}_2^2.
\]
We sum from $j = 0$, \dots, $n-1$ and obtain the same lower Riemann sum as in the proof of (3).  The final estimate $Cm + \norm{Du}_2^2$ follows because $C^2 m t / (1 + Ct) \leq Cm$.
\end{proof}

\subsection{Iterative Construction of $R_t$ for Dyadic $t$}

We are now ready to carry out the Trotter's formula strategy and construct the semigroup for dyadic values of $t$.  The next step is to show that the operators $P_t$ and $Q_t$ almost commute when $t$ is small.

\begin{lemma} \label{lem:commutation}
Let $u \in \mathcal{E}(c,C)$ and $t > 0$.
\begin{enumerate}
	\item $\norm{D(Q_t P_t u) - D(P_t Q_t u)}_{L^\infty} \leq C^2m^{1/2}(2 + Ct)t^{3/2}$.
	\item $P_t Q_t u \leq Q_t P_t u$.
	\item If $Ct \leq 1$, then $Q_t P_t u \leq P_t Q_t u + 2C^2mt^2 + 2Ct^2 \norm{D(P_t Q_tu)}_2^2$.
\end{enumerate}
\end{lemma}

\begin{proof}
(1) Note that
\[
D(Q_t P_t u)(x) = D(P_t u)(x - t D (Q_t P_t u)(x)) = \int D u(x + y - t D(Q_tP_t u)(x))\,d\sigma_{t,n}(y).
\]
On the other hand,
\[
D(P_t Q_t u)(x) = \int D(Q_t u)(x + y)\,d\sigma_{t,n}(y) = \int D u(x + y - t D(Q_t u)(x + y))\,d\sigma_{t,n}(y).
\]
Because $Du$ is $C$ Lipschitz, we have
\[
\norm{D(Q_t P_t u)(x) - D(P_t Q_t u)(x)}_2 \leq Ct \int \norm{D(Q_t u)(x+y) - D(Q_tP_t u)(x)} _2\,d\sigma_{t,n}(y).
\]
We can estimate the integrand by
\[
\norm{D (Q_t u)(x + y) - D (Q_t u)(x)}_2 + \norm{D(Q_t u)(x) - D(Q_t P_t u)(x)}_2.
\]
Integrating the first term and using the fact that $D(Q_tu)$ is $C$-Lipschitz (since $u \in \mathcal{E}(0,C)$ by Lemma \ref{lem:infconvolution} (4)), we have
\[
\int \norm{D (Q_t u)(x + y) - D (Q_t u)(x)}_2  \,d\sigma_{t,n}(y) \leq C \int \norm{y}_2 \,d\sigma_{t,n} \leq C m^{1/2} t^{1/2}
\]
Meanwhile, the second term is independent of $y$ and thus it is unchanged when we integrate it against the probability measure $\sigma_{t,N}$, and this quantity can be estimated using Lemma \ref{lem:initialerrorestimate} (2) and Lemma \ref{lem:Gaussianconvexity} (2) as
\[
\norm{D (Q_t u)(x) - D(Q_t P_t u)(x)}_2 \leq (1 + Ct) \norm{Du - D(P_t u)}_{L^\infty} \leq (1 + Ct)Cm^{1/2} t^{1/2},
\]
Altogether, we obtain
\[
\norm{D(Q_t P_t u)(x) - D(P_t Q_t u)(x)}_2 \leq C^2m^{1/2}(2 + Ct)t^{3/2}.  \qedhere
\]

(2) The idea is that the average of the infimum is less than or equal to the infimum of the average.  More precisely,
\begin{align*}
P_t Q_t u(x) &= \int \inf_y [u(y) + \frac{1}{2t} \norm{(x+z) - y}_2^2] \,d\sigma_{t,N}(z) \\
&= \int \inf_y [u(y-z) + \frac{1}{2t} \norm{x - y}_2^2] \,d\sigma_{t,N}(z) \\
&\leq \inf_y \int [u(y-z) + \frac{1}{2t} \norm{x - y}_2^2] \,d\sigma_{t,N}(z) \\
&= \inf_y [P_tu(y) + \frac{1}{2t} \norm{x - y}_2^2] \\
&= Q_t P_t u(x).
\end{align*}

(3) To prove the other inequality, note that by Corollary \ref{cor:infconvolution} (3),
\begin{equation} \label{eq:commutationineq1}
Q_t P_t u \leq P_t u - \frac{t}{2} \norm{D(Q_tP_tu)}_2^2.
\end{equation}
Also by Corollary \ref{cor:infconvolution} (3),
\[
u \leq Q_t u + \frac{t}{2}(1 + Ct) \norm{D(Q_tu)}_2^2.
\]
Hence, by Lemma \ref{lem:initialerrorinequality}, since $Q_t u \in \mathcal{E}(c(1+ct)^{-1},C(1+Ct)^{-1}) \subseteq \mathcal{E}(0,C)$, we have
\begin{equation} \label{eq:commutationineq2}
P_t u \leq P_t Q_t u + \frac{C^2mt^2}{2}(1 + Ct) + \frac{t}{2}(1 + Ct) \norm{D(P_tQ_t)}_2^2.
\end{equation}
Plugging \eqref{eq:commutationineq2} into \eqref{eq:commutationineq1}, we obtain
\begin{equation} \label{eq:commutationineq3}
Q_t P_t u \leq P_tQ_t u + \frac{C^2mt^2}{2}(1 + Ct) - \frac{t}{2} \norm{D(Q_tP_tu)}_2^2 + \frac{t}{2}(1 + Ct) \norm{D(P_tQ_t)}_2^2.
\end{equation}
By using part (1), we have
\begin{align*}
\norm{D(Q_tP_tu)}_2^2 &\geq [\norm{D(P_tQ_t)u}_2 - C^2 m^{1/2} t^{3/2}(2 + Ct)]^2 \\
&\geq \norm{D(P_tQ_tu)}_2^2 - 2C^2 m^{1/2} t^{3/2}(2 + Ct) \norm{D(P_tQ_tu)}_2 \\
&\geq \norm{D(P_tQ_tu)}_2^2 - (2 + Ct)[ C^3mt^2 + Ct \norm{D(P_tQ_tu)}_2^2]
\end{align*}
where the last step follows from the arithmetic-geometric mean inequality
\[
2Cm^{1/2}t^{1/2} \norm{D(P_tQ_tu)}_2 \leq C^2mt + \norm{D(P_tQ_tu)}_2^2.
\]
So substituting our estimate for $\norm{D(Q_tP_tu)}_2^2$ into \eqref{eq:commutationineq3}, we see that $P_t Q_t u - Q_t P_t u$ is bounded by
\[
\frac{C^2mt^2}{2} + \frac{t}{2}(2 + Ct)[C^3mt^2 + Ct \norm{D(P_tQ_tu)}_2^2]  - \frac{t}{2} \norm{D(P_tQ_tu)}_2^2  + \frac{t}{2}(1 + Ct) \norm{D(P_tQ_t)}_2^2
\]
Now we cancel the first-order terms $(t/2) \norm{D(P_tQ_tu)}_2^2$ and we estimate $2 + Ct$ by $3$ using our assumption that $Ct \leq 1$.  Thus, this is bounded by
\begin{align*}
& \frac{C^2mt^2}{2} + \frac{3}{2} t[C^3mt^2 + Ct \norm{D(P_tQ_tu)}_2^2] + \frac{1}{2} Ct^2 \norm{D(P_tQ_tu)}_2^2 \\
\leq& 2C^2mt^2 + 2Ct^2 \norm{D(P_t Q_tu)}_2^2,
\end{align*}
where we have again used our assumption $Ct \leq 1$ to cancel $Ct$ from the term $t \cdot C^3mt^2$.
\end{proof}

Finally, we can construct the semigroup $R_t$ for dyadic values of $t$.  As in the statement of Theorem \ref{thm:semigroup}, we define $R_{t,\ell}u = (P_{2^{-\ell}} Q_{2^{-\ell}})^{2^\ell t} u$ whenever $\ell \in \Z$ and $t \in 2^{-\ell} \N_0$.

\begin{lemma} \label{lem:convergence}
Let $C \geq 0$.  For $t \in \Q_2^+$ and $u \in \mathcal{E}(0,C)$, the limit $R_t u = \lim_{\ell \to \infty} R_{t,\ell} u$ exists.  Moreover, we have for $t \in 2^{-\ell} \N_0$ that
\begin{enumerate}
	\item $R_{t,\ell}u \leq R_t u$.
	\item If $C / 2^{\ell+1} \leq 1$, then
	\[
	R_t u \leq R_{t,\ell} u + \left( \frac{3}{2} \frac{C^2mt}{1+Ct} + \log(1 + Ct) (m + Cm + \norm{Du}_2^2) \right) 2^{-\ell}.
	\]
	\item $\displaystyle \norm{D(R_{t,\ell} u) - D(R_t u)}_{L^\infty} \leq [t/2 + C(t/2)^2] C^2 m^{1/2}(2 \cdot 2^{-\ell/2} + 2^{-3\ell/2}C)$.
\end{enumerate}
\end{lemma}

\begin{proof}
First, we prove some intermediate claims relating $R_{t,\ell}u$ and $R_{t,\ell+1}u $.  To this end, we fix $\ell \in \Z$ and $t = 2^{-\ell} n$ for some $n \in \N_0$.  Let $\delta = 2^{-\ell-1}$.  For $j = 0$, \dots, $n$, define
\[
u_j = (P_\delta Q_\delta)^{2(n-j)}(P_{2\delta} Q_{2\delta})^j u.
\]
and note that
\[
u_0 = R_{t,\ell+1} u \qquad u_n = R_{t,\ell} u.
\]
Let
\[
v_j = Q_\delta (P_{2\delta} Q_{2\delta})^j u.
\]
Then for $j = 1, \dots, n$, we have
\begin{align*}
u_{j-1} &= [(P_\delta Q_\delta)^{2(n-j)} P_\delta](Q_\delta P_\delta v_{j-1}) \\
u_j &= [(P_\delta Q_\delta)^{2(n-j)} P_\delta](P_\delta Q_\delta v_{j-1}).
\end{align*}
We also define for $k = 1,\dots, 2n$,
\[
C_k = C(1 + Ck\delta)^{-1}, \qquad c_k = c(1 + ck \delta)^{-1}.
\]
Thus, by Lemma \ref{lem:errorpropagation} (1) and Lemma \ref{lem:infconvolution} (4), we have $v_{j-1} \in \mathcal{E}(c_{2j-1},C_{2j-1})$.

First, we claim that
\begin{equation} \label{eq:convergencemonotone}
R_{t,\ell}u \leq R_{t,\ell+1} u.
\end{equation}
Now by Lemma \ref{lem:commutation} (2), we have $P_\delta Q_\delta v_{j-1} \leq Q_\delta P_\delta v_{j-1}$.  Hence, by monotonicity of $P_t$ and $Q_t$ (Lemma \ref{lem:errorpropagation} (3)), we have $u_j \leq u_{j-1}$.  Hence, $R_{t,\ell} u = u_n \leq u_0 = R_{t,\ell+1} u$, proving \eqref{eq:convergencemonotone}.

For an inequality in the other direction, we claim that
\begin{equation} \label{eq:convergence2}  
R_{t,\ell+1} u \leq R_{t,\ell} u + \left( \frac{3}{2} C m + \log(1 + Ct) (m + Cm + \norm{Du}_2^2) \right) 2^{-\ell-1}
\end{equation}
By Lemma \ref{lem:commutation} (3), since $v_{j-1} \in \mathcal{E}(c_{2j-1},C_{2j-1})$, we obtain
\[
Q_\delta P_\delta v_{j-1} \leq P_\delta Q_\delta v_{j-1} + 2 C_{2j-1}^2 m \delta^2 + 2 C_{2j-1} \delta^2 \norm{D(P_\delta Q_\delta v_{j-1})}_2^2
\]
Thus, by Lemma \ref{lem:initialerrorinequality} (1), since $Q_\delta P_\delta v_{j-1}$ and $P_\delta Q_\delta v_{j-1}$ are in $\mathcal{E}(c_{2j},C_{2j})$, we have
\begin{align*}
P_\delta Q_\delta P_\delta v_{j-1} &\leq P_{2\delta} Q_\delta v_{j-1} + 2 C_{2j-1} m \delta^2 + 2 C_{2j-1} \delta^2 \left( C_{2j}^2m\delta + \norm{D(P_{2\delta} Q_\delta v_{j-1})}_2^2 \right)
\end{align*}
Recalling that $u_{j-1}$ and $u_j$ are obtained by applying $(P_\delta Q_\delta)^{2(n-j)}$ to $P_\delta Q_\delta P_\delta v_{j-1}$ and $P_{2\delta} Q_\delta v_{j-1}$, and that $P_\delta Q_\delta P_\delta v_{j-1}$ and $P_{2\delta} Q_\delta v_{j-1}$ are in $\mathcal{E}(c_{2j},C_{2j})$, we may apply Lemma \ref{lem:errorpropagation} (3) to conclude that
\begin{align*}
u_{j-1} &\leq u_j + 2 C_{2j-1} m \delta^2 + 2 C_{2j-1} \delta^2 \left( C_{2j}^2 m\delta + \frac{C_{2j}^2m(n-j)\delta}{1+2C_{2j}(n-j) \delta} + \norm{Du_j}_2^2 \right)
\end{align*}
By our assumption, $C_{2j} \delta \leq C\delta \leq 1$, and thus
\[
C_{2j}^2 m\delta + \frac{C_{2j}^2m(n-j)\delta}{1+2C_{2j}(n-j) \delta} \leq C_{2j} m \delta + \frac{1}{2} C_{2j} m \delta = \frac{3}{2} C_{2j} m \delta \leq \frac{3}{2} C_{2j-1} m \delta.
\]
Therefore,
\[
u_{j-1} - u_j \leq 2 C_{2j-1} m \delta^2 + 3 C_{2j-1}^2 m \delta^2 + 2 C_{2j-1} \delta^2 \norm{Du_j}_2^2.
\]
By Lemma \ref{lem:errorpropagation} (4), we have $\norm{Du_j}_2 \leq Cm + \norm{Du}_2^2$, and hence
\[
u_{j-1} - u_j \leq 3 C_{2j-1}^2 m \delta^2 + 2 C_{2j-1} \delta^2(m + Cm + \norm{Du}_2^2).
\]
Therefore, summing from $j = 1$, \dots, $n$, we have
\begin{align*}
R_{t,\ell+1}u - R_{t,\ell} u &\leq 3m \delta^2 \sum_{j=1}^n C_{2j-1}^2 + 2 \delta^2 (m + Cm + \norm{Du}_2^2) \left( \sum_{j=1}^n C_{2j-1} \right) \\
&= \frac{3}{2} m \delta \left( \sum_{j=1}^n C_{2j-1}^2(2\delta) \right) + \delta(m + Cm + \norm{Du}_2^2) \left( \sum_{j=1}^n C_{2j-1} (2\delta) \right)
\end{align*}
Recalling the definition of $C_{2j-1}$, two times the first sum is $\sum_{j=1}^n C^2(2\delta) / (1 + C(2j-1)\delta)^2$, which is the Riemann sum for the function $\phi(s) =  C^2 / (1 + Cs)^2$ on the interval $[0,t] = [0,2n\delta]$, where we use a partition into subintervals of length $2\delta$ and evaluate $\phi$ at the midpoint of each interval.  Because $\phi$ is convex, the value of $\phi$ at the midpoint is less than or equal to the average value over the subinterval and therefore
\[
\sum_{j=1}^n \frac{C^2(2\delta)}{(1 + C(2j-1)\delta)^2} \leq \int_0^t \frac{C^2}{(1 + Cs)^2}\,ds = \frac{C^2 t}{1 + Ct}.
\]
By similar reasoning,
\[
\sum_{j=1}^n C_{2j-1} (2\delta) = \sum_{j=1}^n \frac{C\delta}{(1 + C(2j-1)\delta)} \leq \int_0^t \frac{C}{1 + Cs}\,ds = \log(1 + Ct).
\]
Therefore,
\begin{align*}
R_{t,\ell+1}u - R_{t,\ell} u &\leq \left( \frac{3}{2} \frac{C^2 mt}{1 + Ct} + \log(1 + Ct) (m + Cm + \norm{Du}_2^2) \right) \delta,
\end{align*}
which proves \eqref{eq:convergence2}.

Together, \eqref{eq:convergencemonotone} and \eqref{eq:convergence2} show that
\[
|R_{t,\ell+1}u - R_{t,\ell}u| \leq \left( \frac{3}{2} \frac{C^2 mt}{1 + Ct} + \log(1 + Ct) (m + Cm + \norm{Du}_2^2) \right) 2^{-\ell-1}.
\]
Because the right hand side is summable in $\ell$, we see that the sequence $\{R_{t,\ell}u(x)\}_{\ell \in \N}$ is Cauchy and hence converges.  Thus, $\lim_{\ell \to \infty} R_{t,\ell} u$ exists.  Also, by \eqref{eq:convergencemonotone} the convergence is monotone and thus $R_{t,\ell} u \leq R_t u$, establishing (1).  On the other hand, we obtain (2) by summing up the estimate \eqref{eq:convergence2} from $\ell$ to $\infty$ using the geometric series formula.

It remains to prove (3).  We first claim that
\begin{equation} \label{eq:convergenceofgrad}
\norm{D(R_{t,\ell+1}u) - D(R_{t,\ell}u)}_{L^\infty} \leq [t/2 + C(t/2)^2] C^2 m^{1/2}(2 + 2^{-(\ell+1)}C) 2^{-(\ell+1)/2}
\end{equation}
By Lemma \ref{lem:errorpropagation} (3), we have $Q_\delta P_\delta v_{j-1}$ and $P_\delta Q_\delta v_{j-1}$ are in $\mathcal{E}(c(1+2cj\delta)^{-1}, C(1 + 2Cj\delta)^{-1})$, hence in $\mathcal{E}(0,C)$.  Therefore, by Lemma \ref{lem:errorpropagation} (??) and Lemma \ref{lem:commutation} (1), we have
\begin{align*}
\norm{Du_j - Du_{j-1}}_{L^\infty} &\leq [1 + 2C(n-j) \delta] \norm{D(Q_\delta P_\delta v_j) - D(P_\delta Q_\delta v_j)}_{L^\infty} \\
&\leq [1 + 2C(n-j) \delta] C^2m^{1/2}(2 + C \delta) \delta^{3/2}.
\end{align*}
Therefore,
\begin{align*}
\norm{D(R_{t,\ell+1}u) - D(R_{t,\ell}u)}_{L^\infty} &\leq \sum_{j=1}^n \norm{Du_j - Du_{j-1}}_{L^\infty} \\
&\leq \sum_{j=1}^n [1 + 2C(n-j) \delta] C^2 m^{1/2}(2 + C \delta) \delta^{3/2} \\
&= [n + Cn(n-1) \delta] C^2 m^{1/2}(2 + C \delta) \delta^{3/2} \\
&\leq [t/2 + C(t/2)^2] C^2 m^{1/2}(2 + C \delta) \delta^{1/2} \\
&= [t/2 + C(t/2)^2] C^2 m^{1/2}(2 + 2^{-(\ell+1)}C) 2^{-(\ell+1)/2}
\end{align*}
since $2 n \delta = t$.  This proves \eqref{eq:convergenceofgrad}.

Because $[t/2 + C(t/2)^2] C^2 m^{1/2}(2 + 2^{-(\ell+1)}C) 2^{-(\ell+1)/2}$ is summable with respect to $\ell$, we see that $\{D(R_{t,\ell}u)\}_{\ell \in \N}$ is Cauchy with respect to the $L^\infty$ norm (even though the individual functions may not be in $L^\infty$) and hence converges uniformly to some function.   We already know that $R_{t,\ell} u$ converges to $R_t u$, so the limit of $D(R_{t,\ell} u)$ must be $D(R_t u)$.  We obtain the estimate (3) by summing the \eqref{eq:convergenceofgrad} from $\ell$ to $\infty$ using the geometric series formula.
\end{proof}

\begin{corollary} \label{cor:Rerrorpropagation}
Let $0 \leq c \leq C$.  Let $u, v \in \mathcal{E}(c,C)$ and let $t \geq 0$ be a dyadic rational.
\begin{enumerate}
	\item $R_t u, R_t v \in \mathcal{E}(c(1+ct)^{-1},C(1+Ct)^{-1})$.
	\item $\norm{D(R_tu) - D(R_tv)}_{L^\infty} \leq (1 + Ct) \norm{Du - Dv}_{L^\infty}$.
	\item If $u \leq v + a + b \norm{Dv}_2^2$ for some $a \in \R$ and $b \geq 0$, then
	\[
	R_t u \leq R_t v + a + b \frac{C^2mt}{1 + Ct} + b \norm{D(R_tv)}_2^2.
	\]
	\item $\norm{D(R_tu)}_2^2 \leq \frac{C^2mt}{1 + Ct} + \norm{Du}_2^2$.
\end{enumerate}
\end{corollary}

\begin{proof}
We know that these properties hold for $R_{t,\ell}$ by Lemma \ref{lem:errorpropagation}.  By Lemma \ref{lem:convergence}, they also hold in the limit taking $\ell \to \infty$.  (For (1), we use the fact that $\mathcal{E}(c',C')$ is closed under pointwise limits for each $c'$ and $C'$.)
\end{proof}

\subsection{Continuity and Semigroup Property} \label{subsec:continuity1}

In order to extend $R_t$ to all real $t \geq 0$, we prove estimates that show that $R_t$ depends continuously on $t$.  We begin with some simple estimates for $P_t$ and $Q_t$.

\begin{lemma} \label{lem:initialcontinuityestimates}
Let $\ell \in \Z$ and suppose that $t \in 2^{-\ell} \N_0$ and $u \in \mathcal{E}(0,C)$.  Then
\begin{enumerate}
	\item $u \leq P_t u \leq u + Cmt/2$.
	\item $u - (t/2) \norm{Du}_2^2 \leq Q_t u \leq u$.
	\item $\norm{D(Q_t u) - Du}_2 \leq Ct \norm{Du}_2$.
\end{enumerate}
\end{lemma}

\begin{proof}
(1) Because $u$ is convex and $u(x) - (C/2) \norm{x}_2^2$ is concave, we have
\[
u(x) + \ip{Du(x),y} \leq u(x+y) \leq u(x) + \ip{Du(x),y} + \frac{C}{2} \norm{y}_2^2.
\]
Integrating with respect to $d\sigma_{t,N}(y)$ yields
\[
u(x) \leq P_t u(x) \leq u(x) + \frac{Cmt}{2} \text{ for } u \in \mathcal{E}(0,C).
\]

(2) As for the operator $Q_t$, it is immediate from the definition that $Q_t u \leq u$.  On the other hand,
\begin{align*}
Q_t u(x) &= u(x - t D(Q_tu)(x)) + \frac{t}{2} \norm{D(Q_tu)(x)}_2^2 \\
&\geq u(x) -t \ip{D(Q_tu)(x), Du(x)}_2 + \frac{t}{2} \norm{D(Q_tu)(x)}_2^2 \\
&\geq u(x) - \frac{t}{2} \norm{Du(x)}_2^2,
\end{align*}
where the last inequality follows because $\ip{D(Q_tu)(x), Du(x)}_2 \leq \frac{1}{2} \norm{D(Q_tu)(x)}_2^2 + \frac{1}{2} \norm{Du(x)}_2^2$.

(3) Using the fact that $Du$ is $C$-Lipschitz, together with Corollary \ref{cor:infconvolution} (1) and Lemma \ref{lem:infconvolution} (5)
\begin{align*}
\norm{D(Q_tu)(x) - Du(x)}_2 &= \norm{Du(x - t D(Q_t u)(x)) - Du(x)}_2 \\
&\leq Ct \norm{D(Q_tu)(x)}_2 \\
&\leq Ct \norm{Du(x)}_2. \qedhere
\end{align*}
\end{proof}

\begin{lemma} \label{lem:continuityestimates}
Let $s \leq t$ be two numbers in $\Q_2^+$, and let $u \in \mathcal{E}(0,C)$.
\begin{enumerate}
	\item $R_t u \leq R_s u + \frac{m}{2} [\log(1 + Ct) - \log(1 + Cs)]$.
	\item $R_t u \geq R_s u - \frac{1}{2} (t - s)(Cm + \norm{Du}_2^2)$.
	\item If $C(t - s) \leq 1$, then $\norm{D(R_t u) - D(R_s u)}_2 \leq 5Cm^{1/2} 2^{1/2}(t - s)^{1/2} + C(t - s) \norm{Du}_2$.
\end{enumerate}
Moreover, if $\ell \in \Z$ and if $s, t \in 2^{-\ell} \N_0$, then the same estimates hold with $R_t$ replaced by $R_{t,\ell}$.
\end{lemma}

\begin{proof}
(1)  Fix $\ell \in \Z$ and let $\delta = 2^{-\ell}$.  Suppose that $s = n\delta$ and $t = n' \delta$ where $n, n' \in \N_0$.  By the previous lemma,
\begin{align*}
R_{(j+1)\delta,\ell} u &= P_\delta Q_\delta R_{j\delta,\ell} u \\
&\leq Q_\delta R_{j\delta,\ell} u + \frac{Cm \delta}{2(1 + C(j+1)\delta)} \\
&\leq R_{j\delta,\ell} u + \frac{Cm\delta}{2(1 + C(j+1)\delta)},
\end{align*}
where we have used the fact that $Q_\delta R_{j\delta,\ell}u \in \mathcal{E}(0,C(1+C(j+1)\delta)^{-1})$.  Therefore,
\[
R_{n'\delta,\ell} u \leq R_{n\delta,u} + \sum_{j=n}^{n'-1} \frac{Cm\delta}{2(1+C(j+1)\delta)}.
\]
Since the sum on the right hand side is a lower Riemann sum for the function $Cm\delta / 2(1+C \tau)$ for $\tau \in [s,t]$, we obtain
\[
R_{t,\ell} u \leq R_{s,\ell} u + \frac{m}{2} [\log(1 + Ct) - \log(1 + Cs)].
\]
We obtain (1) by letting $\ell \to +\infty$ and using Lemma \ref{lem:convergence}.

(2) Let $\ell, \delta, s,t, n, n'$ be as above.  By the previous lemma,
\begin{align*}
R_{(j+1)\delta,\ell} u &= P_\delta Q_\delta R_{j\delta,\ell} u \\
&\geq Q_\delta R_{j\delta,\ell} u \\
&\geq R_{j\delta,\ell} u - \frac{\delta}{2} \norm{D(R_{j\delta,\ell}u)}_2^2 \\
&\geq R_{j\delta,\ell} u - \frac{\delta}{2}(Cm + \norm{Du}_2^2),
\end{align*}
where the last inequality follows from Lemma \ref{lem:errorpropagation} (4).  So when we sum from $j = n$ to $n' - 1$, we obtain
\[
R_t u \geq R_s u - \frac{t - s}{2}(Cm + \norm{Du}_2^2).
\]
Then (2) follows by taking $\ell \to +\infty$.

(3) Assume that $s, t \in 2^{-\ell} \N_0$.  Choose $k \in \Z$ such that $2^{-k-1} \leq t - s \leq 2^{-k}$.  Then we may write $t - s$ in a binary expansion
\[
t - s = \sum_{j=k+1}^\ell a_j 2^{-j},
\]
where $a_j \in \{0,1\}$ for each $j$ and $a_{k+1} = 1$.  Let $t_j = s + a_{k+1} 2^{-k-1} + \dots + a_j 2^{-j}$.  Let $u_j = R_{t_j,\ell}u$.  We will estimate $\norm{Du_j(x) - Du_{j-1}(x)}_2$ for each $j$.  Of course, if $a_j = 0$, then $u_j = u_{j-1}$, so there is nothing to prove.  On the other hand, suppose that $a_j = 1$.  Now we estimate (at our given point $x$, suppressed in the notation)
\begin{align} \label{eq:threeterm}
\norm{D(R_{2^{-j},\ell} u_{j-1}) - Du_{j-1}}_2
&\leq \norm{D(R_{2^{-j},\ell} u_{j-1}) - D(P_{2^{-j}} Q_{2^{-j}}u_{j-1})}_2 \\
& \quad + \norm{D(P_{2^{-j}} Q_{2^{-j}} u_{j-1}) - D(Q_{2^{-j}} u_{j-1})}_2 \nonumber \\
& \quad + \norm{D(Q_{2^{-j}}u_{j-1}) - Du_{j-1}}_2. \nonumber
\end{align}
The first term on the right hand side may be estimated as follows.  Recall that we proved Lemma \ref{lem:convergence} (3) from the estimate from \eqref{eq:convergenceofgrad} by summing the geometric series.  The same reasoning shows that if $\ell \geq j$ and $\delta \in 2^{-\ell} \N_0$, then
\[
\norm{D(R_{\delta,\ell} u_{j-1}) - D(R_{\delta,j} u_{j-1})}_{L^\infty} \leq [\delta/2 + C(\delta/2)^2] C^2 m^{1/2}(2 \cdot 2^{-j/2} + 2^{-3j/2}C)
\]
since $u_{j-1} \in \mathcal{E}(0,C)$.  When we substitute $\delta = 2^{-j}$, $R_{2^{-j},j}$ is simply equal to $P_{2^{-j}} Q_{2^{-j}}$.  Thus, at the point $x$,
\[
\norm{D(R_{2^{-j},\ell} u_{j-1}) - D(P_{2^{-j}} Q_{2^{-j}} u_{j-1})}_2 \leq C^2 m^{1/2} [2^{-j}/2 + C2^{-2j}/4][2 \cdot 2^{-j/2} + 2^{-3j/2} C].
\]
By our assumption $C2^{-j} \leq C(t - s) \leq 1$ and hence we may replace $C2^{-2j}/4$ by $2^{-j}/2$ and repalce $2^{-3j/2} C$ by $2^{-j/2}$ and hence
\[
\norm{D(R_{2^{-j},\ell} u_{j-1}) - D(P_{2^{-j}} Q_{2^{-j}} u_{j-1})}_2 \leq 3C^2m^{1/2} 2^{-3j/2} \leq 3Cm^{1/2} 2^{-j/2}.
\]
The second term on the right hand side of \eqref{eq:threeterm} can be estimated by Lemma \ref{lem:Gaussianconvexity} (2) by
\[
\norm{D(P_{2^{-j}} Q_{2^{-j}} u_{j-1}) - D(Q_{2^{-j}} u_{j-1})}_2 \leq Cm^{1/2} 2^{-j/2}
\]
since $Q_{2^{-j}} u_{j-1} \in \mathcal{E}(0,C)$.  The third term on the right hand side of \eqref{eq:threeterm} can be estimated using Corollary \ref{lem:initialcontinuityestimates} (3) by
\[
\norm{D(Q_{2^{-j}} u_{j-1}) - Du_{j-1}}_2 \leq C 2^{-j} \norm{Du_{j-1}}_2.
\]
Meanwhile, by Lemma \ref{lem:errorpropagation} (4) and the triangle inequality
\[
\norm{Du_{j-1}}_2 \leq \sqrt{Cm + \norm{Du}_2^2} \leq C^{1/2} m^{1/2} + \norm{Du}_2.
\]
So using the fact $C 2^{-j} \leq 1$, we have
\[
\norm{D(Q_{2^{-j}} u_{j-1}) - Du_{j-1}}_2 \leq C^{3/2} m^{1/2} 2^{-j} + C 2^{-j} \norm{Du}_2 \leq Cm^{1/2} 2^{-j/2} + C (t_j - t_{j-1})\norm{Du}_2.
\]
Therefore, plugging all our estimates into \eqref{eq:threeterm}, we get
\[
\norm{Du_j - Du_{j+1}}_2 \leq 5C m^{1/2} 2^{-j/2} + C(t_j - t_{j-1})\norm{Du}_2.
\]
Then summing from $j = k+1$ to $\ell$ we obtain
\begin{align*}
\norm{Du_\ell - Du_k}_2 &\leq 5C m^{1/2} 2^{-k/2} + C(t - s) \norm{Du}_2 \\
&\leq 5Cm^{1/2} 2^{1/2}(t - s)^{1/2} + C(t - s) \norm{Du}_2.
\end{align*}
Because $u_\ell = R_{t,\ell} u$ and $u_k = R_{s,\ell} u$, we have shown that (3) holds for $R_{s,\ell}$ and $R_{t,\ell}$ instead of $R_s$ and $R_t$.  Thus, (3) follows by taking $\ell \to +\infty$.
\end{proof}

\begin{proof}[Proof of Theorem \ref{thm:semigroup}]
Lemma \ref{lem:continuityestimates} shows that if $t \geq 0$ and if $t_\ell$ is a sequence of dyadic rationals converging to $t$ as $\ell \to \infty$, then $R_{t_\ell} u$ converges to some function $v$ and this function is independent of the approximating sequence, so we define $R_t u = v$.   Claim (1), (3), and (4) of the Theorem were proved for dyadic $t$ in Corollary \ref{cor:Rerrorpropagation} (1), Lemma \ref{lem:continuityestimates}, and Corollary \ref{cor:Rerrorpropagation} (2) - (4) respectively, and each of these claims can be extended to real $t \geq 0$ in light of the continuity estimates Lemma \ref{lem:continuityestimates}.  Claim (3) of the theorem is Lemma \ref{lem:convergence}.

Thus, it remains to show that $R_t$ is a semigroup.  That is, we must show that $R_s R_tu = R_{s+t}u$ for $u \in \mathcal{E}(0,C)$ (and we have not even checked this for dyadic $s, t$ yet).  First, we check this property for real $s, t \geq 0$ under the additional restriction that $Ct \leq 1/2$.  For each $\ell \in \Z$, there exist $s_\ell$ and $t_\ell \in 2^{-\ell} \N_0$ such that $s - 2^{-\ell} < s_\ell \leq s$ and $t - 2^{-\ell} < t_\ell \leq t$.  By Lemma \ref{lem:continuityestimates} (1) and (2) we have
\[
|R_{t_\ell} u - R_t u| \leq \frac{|t_\ell - t|}{2} (Cm + \norm{Du}_2^2) \leq 2^{-\ell} \frac{1}{2}(Cm + \norm{Du}_2^2),
\]
since $|\log(1+Ct_\ell) - \log(1 + Ct)| \leq C|t_\ell - t|$ (from computation of the derivative of $\log(1 + Ct)$).  By Lemma \ref{lem:convergence} (1) and (2), if $C 2^{-\ell-1} \leq 1$, then
\[
|R_{t_\ell,\ell} u - R_{t_\ell} u| \leq 2^{-\ell} \left( \frac{3}{2} \frac{C^2mt}{1+Ct} + \log(1 + Ct_\ell)(m + Cm + \norm{Du}_2^2) \right).
\]
Since $t_\ell \leq t$, we can replace $t_\ell$ by $t$ on the right hand side.  By the triangle inequality, we obtain
\begin{equation} \label{eq:semigroupproof1}
|R_{t_\ell,\ell} u - R_t u| \leq 2^{-\ell} K_t(1 + \norm{Du}_2^2)
\end{equation}
for some constant $K_t$ depending on $t$ (and $C$).  Using Lemma \ref{lem:continuityestimates} (3), or rather its extension to real values of $t$,
\begin{align*}
\norm{D(R_tu) - Du}_2 &\leq 5Cm^{1/2} 2^{1/2} t^{1/2} + Ct \norm{Du}_2 \\
&\leq 5Cm^{1/2} 2^{1/2} t^{1/2} + Ct \norm{D(R_tu) - Du}_2 + Ct \norm{D(R_t u)}_2.
\end{align*}
Hence,
\[
\norm{D(R_tu) - Du}_2 \leq (1 - Ct)^{-1}[5Cm^{1/2} 2^{1/2} t^{1/2} + Ct \norm{D(R_t u)}_2],
\]
so by the triangle inequality,
\[
\norm{Du}_2 \leq \norm{D(R_tu)}_2 + (1 - Ct)^{-1}[5Cm^{1/2} 2^{1/2} t^{1/2} + Ct \norm{D(R_t u)}_2.
\]
By squaring and applying the arithmetic-geometric mean inequality, we get
\[
\norm{Du}_2^2 \leq A_t + B_t \norm{D(R_tu)}_2^2
\]
for some constants $A$ and $B$ depending on $t$.  The same reasoning applies to $R_{t_\ell,\ell}$ since Lemma \ref{lem:continuityestimates} (3) holds for $R_{t_\ell,\ell}$ also.  We thus obtain
\begin{align*}
\norm{Du}_2 &\leq \norm{D(R_{t_\ell,\ell} u)}_2 + (1 - Ct_\ell)^{-1}[5Cm^{1/2} 2^{1/2} t_\ell^{1/2} + Ct_\ell \norm{D(R_{t_\ell,\ell} u)}_2 \\
&\leq \norm{D(R_{t_\ell,\ell} u)}_2 + (1 - Ct)^{-1}[5Cm^{1/2} 2^{1/2} t^{1/2} + Ct \norm{D(R_{t_\ell,\ell} u)}_2
\end{align*}
and so
\[
\norm{Du}_2^2 \leq A_t + B_t \norm{D(R_{t_\ell,\ell}u)}_2^2.
\]
Overall,
\begin{align*}
R_t u &\leq R_{t_\ell,\ell} u + 2^{-\ell} K_t(1 + A_t + B_t \norm{D(R_{t_\ell,\ell} u)}_2^2 \\
R_{t_\ell,\ell} u &\leq R_t u + 2^{-\ell} K_t(1 + A_t + B_t \norm{D(R_tu)}_2^2
\end{align*}
So by Lemma \ref{lem:errorpropagation} (3) and (4)
\begin{align*}
R_{s_\ell,\ell} R_t u &\leq R_{s_\ell,\ell} R_{t_\ell,\ell} u + 2^{-\ell} K_1(1 + A_t + B_t \norm{D(R_{s_\ell,\ell} R_{t_\ell,\ell}u)}_2^2) \\
&\leq R_{s_\ell,\ell} R_{t_\ell,\ell} u + 2^{-\ell} K_1(1 + A_t + CmB_t + B_t \norm{Du}_2^2)
\end{align*}
and the same holds with $R_t$ and $R_{t_\ell,\ell}$ switched, so that
\begin{equation} \label{eq:semigroupproof1B}
|R_{s_\ell,\ell} R_t u - R_{s_\ell+t_\ell,\ell} u| \leq 2^{-\ell} K_1(1 + A_t + CmB_t + B_t \norm{Du}_2^2),
\end{equation}
where we have noted that $R_{s_\ell+t_\ell,\ell}u = R_{s_\ell,\ell} R_{t_\ell,\ell}u$.

But then by Lemma \ref{lem:errorpropagation}.  By the same token as \eqref{eq:semigroupproof1}, since $R_t u \in \mathcal{E}(0,C)$, we have
\begin{equation}
|R_{s_\ell,\ell} R_t u - R_s R_t u| \leq 2^{-\ell} K_s(1 + \norm{D(R_tu)}_2^2).
\end{equation}
Similarly, since $(s+t) - (s_\ell + t_\ell) \leq 2 \cdot 2^{-\ell}$, we have
\begin{equation}
|R_{s_\ell+t_\ell,\ell}u - R_{s+t} u| \leq 2^{-\ell} \cdot 2K_{s+t}(1 + \norm{Du}_2^2).
\end{equation}
Combining these with \eqref{eq:semigroupproof1B} using the triangle inequality, we get
\begin{multline*}
|R_s R_t u - R_{s+t}u| \leq 2^{-\ell} K_1(1 + A_t + CmB_t + B_t \norm{Du}_2^2) \\ + 2^{-\ell} K_s(1 + \norm{D(R_tu)}_2^2) + 2^{-\ell} \cdot 2K_{s+t}(1 + \norm{Du}_2^2).
\end{multline*}
Taking $\ell \to \infty$, we get $R_s R_t u = R_{s+t} u$ as desired.  This completes the case when $Ct \leq 1/2$.

In the general case, suppose $s, t \geq 0$ and $u \in \mathcal{E}(0,C)$.  Choose $n$ large enough that $Ct/n \leq 1/2$.  Then for $j = 1$, \dots, $n-1$, we have $R_{t/k}^{n-j} u \in \mathcal{E}(0,C)$.  Therefore, by the previous argument
\[
R_{s+jt/n} R_{t/n}^{n-j} u = (R_{s+jt/n} R_{t/n})(R_{t/n}^{n-j-1} u) = R_{s+(j+1)t/n} R_{t/n}^{n-j-1} u,
\]
so by induction $R_{s+t} u = R_s R_{t/n}^n u$.  Since this also holds with $s$ replaced by $0$, we have $R_{t/n}^n u = R_tu$.  Thus, $R_{s+t} u = R_s R_t u$.
\end{proof}

\subsection{Solution to the Differential Equation} \label{subsec:viscosity}

It remains to show that the semigroup $R_t$ produces solutions to the differential equation $\partial_t u = (1/2N) \Delta u - (1/2) \norm{Du}_2^2$, and that the result agrees with the solution produced by solving the heat equation for $\exp(-N^2u)$.  More precisely, we will prove the following.

\begin{theorem} \label{thm:solution}
Let $u_0: M_N(\C)_{sa}^m \to \R$ be a given function in $\mathcal{E}(c,C)$ for some $c \geq 0$.  Let $u(x,t) = R_t u(x)$.  Then $u$ is a smooth function on $M_N(\C)_{sa}^m \times (0,+\infty)$ and it solves the equation $\partial_t u = (1/2N) \Delta u - (1/2) \norm{Du}_2^2$.  Moreover, $\exp(-N^2 \cdot R_t u_0) = P_t[\exp(-N^2 u_0)]$.
\end{theorem}

At this point, we have not proved enough smoothness for $R_t u$ to show that it solves the equation in the classical sense.  Therefore, as an intermediate step, we show that $u$ solves the equation in the \emph{viscosity sense} defined by \cite{CL1983}; see \cite{CIL1992} and the references cited therein for further background.  We will then deduce that $\exp(-N^2u)$ is a viscosity solution of the heat equation and hence show it agrees with the smooth solution of the heat equation.

The definition of viscosity solution for parabolic equations is as follows.  Here we continue to use the vector space $M_N(\C)_{sa}^m$ with the normalized inner product (rather than $\R^n$ for some $n$).  For smooth $u: M_N(\C)_{sa}^m \to \R$, we denote by $Du$ and $Hu$ the gradient and Hessian with respect to the inner product $\ip{\cdot,\cdot}_2$; in other words, if $x_0 \in M_N(\C)_{sa}^m$, then $Du(x_0)$ is the vector in $M_N(\C)_{sa}^m$ and $Hu(x_0)$ is the linear transformation $M_N(\C)_{sa}^m \to M_N(\C)_{sa}^m$ such that
\[
u(x) = u(x_0) + \ip{Du(x_0), x - x_0}_2 + \frac{1}{2} \ip{Hu(x_0)[x - x_0], x - x_0}_2 + o(\norm{x - x_0}_2^2).
\]
We denote the space of linear transformations $M_N(\C)_{sa}^m \to M_N(\C)_{sa}^m$ by $B(M_N(\C)_{sa}^m)$, and we denote the self-adjoint elements by $B(M_N(\C)_{sa}^m)_{sa}$.

\begin{definition}
Let $F: B(M_N(\C)_{sa}^m)_{sa} \times M_N(\C)_{sa} \times \R \times M_N(\C)_{sa} \to \R$ be continuous, and consider the partial differential equation
\begin{equation}
\partial_t u = F(Hu,Du,u,x).
\end{equation}
We say that a function $u: M_N(\C)_{sa}^m \times [0,+\infty) \to \R$ is a \emph{viscosity subsolution} if it is upper semi-continuous and if the following condition holds:  Suppose that
\[
x_0 \in M_N(\C)_{sa}^m, \quad t_0 > 0, \quad A \in B(M_N(\C)_{sa}^m)_{sa}, \quad p \in M_N(\C)_{sa}^m, \quad \alpha \in \R,
\]
and suppose that $u$ satisfies
\begin{equation} \label{eq:upperjet}
u(x,t) \leq u(x_0,t_0) + \alpha(t - t_0) + \ip{p,x-x_0}_2 + \frac{1}{2} \ip{A(x - x_0), x - x_0}_2 + o(|t - t_0| + \norm{x - x_0}_2^2).
\end{equation}
Then we also have
\begin{equation} \label{eq:subsolution}
\alpha \leq F(A,p,u(x_0),x_0).
\end{equation}
\end{definition}

\begin{definition}
With the same setup as above, we say that $u: M_N(\C)_{sa}^m \times [0,+\infty) \to \R$ is a \emph{viscosity supersolution} if it is lower semi-continuous and the following condition holds: If $x_0, t_0, A, p, \alpha$ are as above and if
\begin{equation} \label{eq:lowerjet}
u(x,t) \geq u(x_0,t_0) + \alpha(t - t_0) + \ip{p, x-x_0}_2 + \frac{1}{2} \ip{A(x - x_0), x - x_0}_2 + o(|t - t_0| + \norm{x - x_0}_2^2),
\end{equation}
then
\begin{equation} \label{eq:supersolution}
\alpha \geq F(A,p,u(x_0),x_0).
\end{equation}
\end{definition}

\begin{definition}
We say that $u$ is a \emph{viscosity solution} if it is both a subsolution and a supersolution.
\end{definition}

\begin{remark}
Roughly speaking, being a viscosity solution means that whenever there exist upper or lower second-order Taylor approximations to $u$, then we can evaluate the differential operator $F$ on the Taylor approximation and get an inequality in one direction.
\end{remark}

\begin{example}
The heat equation $\partial_t u = (1/2N) \Delta u$ is obtained by taking
\[
F(A,p,u,x) = \frac{1}{2N^2} \Tr(A).
\]
To understand why $1/N^2$ is the correct normalization on the right hand side, suppose that $u$ is smooth and $A = Hu(x_0)$ and $p = Du(x_0)$, so that
\[
u(x) = u(x_0) + \ip{p, x - x_0}_2 + \frac{1}{2} \ip{A(x - x_0), x - x_0}_2 + o(\norm{x - x_0}_2^2).
\]
In terms of the non-normalized inner product (which we denote by the dot product), this means that
\[
u(x) = u(x_0) + \frac{1}{N} p \cdot (x - x_0) + \frac{1}{2N} (A(x - x_0)) \cdot (x - x_0).
\]
Thus, the Hessian with respect to the non-normalized inner product is $(1/N)A$.  Hence, $(1/N) \Delta u(x_0) = (1/N^2) \Tr(A)$.  Similarly, the equation $\partial_t u = (1/2N) \Delta u - (1/2) \norm{Du}_2^2$ is obtained by taking
\[
F(A,p,u,x) = \frac{1}{2N^2} \Tr(A) - \frac{1}{2} \norm{p}_2^2.
\]
\end{example}

\begin{proposition}
Let $u_0 \in \mathcal{E}(0,C)$ and define $u(x,t) = R_t u_0(x)$.  Then $u$ is a viscosity solution of the equation $\partial_t u = (1/2N) \Delta u - (1/2) \norm{Du}_2^2$.
\end{proposition}

\begin{proof}
First, note that $u$ is continuous.  Indeed, by Theorem \ref{thm:semigroup} (3), $u$ is continuous in $t$ with a rate of continuity that is uniform for $x$ in a bounded region (this follows because the term $\norm{Du_0}_2^2$ on the right hand side of Lemma \ref{lem:continuityestimates} (2) is bounded on bounded regions since $Du_0$ is $C$-Lipschitz).  Also, $u(\cdot,t)$ is continuous for each $t$ since it is in $\mathcal{E}(0,C)$.  Together, this implies $u$ is jointly continuous in $(x,t)$.

To show that $u$ is a viscosity subsolution, suppose that we have a lower second-order approximation at the point $(x_0,t_0)$, where $x_0 \in M_N(\C)_{sa}^m$ and $t_0 > 0$, given by
\[
u(x,t) \geq u(x_0,t_0) + \alpha (t - t_0) + \ip{p,x-x_0} + \frac{1}{2} \ip{A(x-x_0),x-x_0}_2 + o(|t - t_0| + \norm{x - x_0}_2^2).
\]
Then we must show that $\alpha \leq (1/2N^2) \Tr(A) - (1/2) \norm{p}_2^2$.

Our first goal is to replace the soft bound $o(|t - t_0| + \norm{x - x_0}_2^2)$ by a more explicit error bound, at the cost of modifying $\alpha$ and $A$ by some positive $\epsilon$.  Pick $\epsilon > 0$.  Then there exists $r > 0$ such that if $|t - t_0| + \norm{x - x_0}_2^2 < 2r$, then we have
\begin{equation} \label{eq:revisedTaylorexpansion}
u(x,t) \geq u(x_0,t_0) + \alpha(t - t_0) - \epsilon |t - t_0| + \ip{p,x-x_0} + \frac{1}{2} \ip{(A - \epsilon I)(x-x_0),x-x_0}_2.
\end{equation}
Let us assume that $t_0 - r < t \leq t_0$, so that the above inequality holds for $\norm{x - x_0}^2 < r$ and we have $\alpha(t - t_0) - \epsilon|t - t_0| = (\alpha + \epsilon)(t - t_0)$.  For $x$ such that $\norm{x - x_0}_2^2 \geq r$, we may use Theorem \ref{thm:semigroup} (3b), the fact that $Du$ is $C$-Lipschitz, and the convexity of $u$ to conclude that
\begin{align*}
u(x,t) &\geq u_0(x) - \frac{t}{2}(Cm + \norm{Du}_2^2) \\
&\geq u_0(x_0) + \ip{Du(x_0),x-x_0}_2 - \frac{t}{2}(Cm + (\norm{Du(x_0)}_2 + C\norm{x-x_0}_2)^2)
\end{align*}
In other words, $u$ is bounded below by a quadratic in $x-x_0$, and the estimate holds uniformly for $t$ in a bounded interval.  Moreover, the right hand side of \eqref{eq:revisedTaylorexpansion} is also bounded by a quadratic in $x - x_0$ uniformly for $t \in [t_0-r,t_0+r]$.  It follows that for a large enough constant $K_\epsilon$, we have
\[
u(x_0,t_0) + (\alpha + \epsilon)(t - t_0) + \ip{p,x-x_0} + \frac{1}{2} \ip{(A - \epsilon I)(x-x_0),x-x_0}_2 - u(x,t) \leq K_\epsilon \norm{x - x_0}_2^4
\]
whenever $t \in (t_0-t,t_0]$ and $\norm{x - x_0}_2 \geq r$.  Therefore, overall, assuming that $t \in (t_0 - r,t_0]$, we have
\begin{equation} \label{eq:revisedTaylorexpansion2}
u(x,t) \geq u(x_0,t_0) + (\alpha + \epsilon)(t - t_0) + \ip{p,x-x_0} + \frac{1}{2} \ip{(A - \epsilon I)(x-x_0),x-x_0}_2 - K_\epsilon \norm{x - x_0}_2^4
\end{equation}

For $t \in \R$, let us denote $u_t(x) = u(x,t) = R_t u_0(x)$.  Now the strategy for proving $\alpha + \epsilon \leq (1/2N^2) \Tr(A - \epsilon I) - (1/2) \norm{p}_2^2$ is roughly to use the fact that $u_{t_0}(x_0) = R_\delta u_{t_0-\delta}(x_0)$ and estimate $u_{t_0 - \delta}(x_0)$ from above using the upper Taylor approximation for small $\delta > 0$.  However, for the sake of computation, it is easier to estimate $Q_\delta P_\delta u_{t_0-\delta}$ rather than $R_\delta$ (and then we will control the error between $R_\delta$ and $Q_\delta P_\delta$ using Lemmas \ref{lem:commutation} and \ref{lem:convergence}).

Let $\delta \in (0,r)$.  Then using the above inequality and monotonicity of $P_\delta$, we have
\begin{align*}
P_\delta u_{t_0-\delta}(x) &\geq u_{t_0}(x_0) + (\alpha - \epsilon) \delta + \ip{p,x-x_0} \\
& \quad + \frac{1}{2N^2} \Tr(A - \epsilon I) \delta + \frac{1}{2} \ip{(A-\epsilon I)(x - x_0), x - x_0} \\
& \quad - K_\epsilon (\norm{x - x_0}_2^4 + 2(1 + 2/N^2) m \delta \norm{x - x_0}_2^2 + 2m^2(1 + 2/N^2) \delta^2).
\end{align*}
Here we have evaluated $P_\delta$ applied to $\norm{x - x_0}_2^4$ using Example \ref{ex:fourthpower} and translation-invariance of $P_\delta$.  Now recall that $Q_\delta P_\delta u_{t_0 - \delta}(x_0)$ is obtained by evaluating $P_\delta u_{t_0 - \delta}$ at $x_0 - \delta D(Q_\delta P_\delta u_{t_0-\delta})(x_0)$.  Also, in light of Lemma \ref{lem:errorpropagation} (4) and Corollary \ref{cor:Rerrorpropagation} (4), $\norm{D(Q_\delta P_\delta u_{t_0-\delta})(x_0)}_2^2$ is bounded by $\norm{Du_0(x_0)}_2^2$ plus a constant.  In particular, $\norm{D(Q_\delta P_\delta u_{t_0-\delta})(x_0)}_2$ is bounded as $\delta \to 0$.  Therefore,
\begin{align} \label{eq:Taylorexpansion3}
Q_\delta P_\delta u_{t_0 - \delta}(x_0) &= P_\delta u_{t_0-\delta}(x_0 - \delta D(Q_\delta P_\delta u_{t_0 - \delta})(x_0)) + \frac{1}{2} \delta \norm{D(Q_\delta P_\delta u_{t_0-\delta})(x_0)}_2^2 \\
&\geq u_{t_0}(x_0) + \frac{1}{2} \norm{D(Q_\delta P_\delta u_{t_0-\delta})(x_0)}_2^2 \delta \nonumber \\
& \quad + (\alpha + \epsilon)(-\delta) - \ip{p, D(Q_\delta P_\delta u_{t_0-\delta})(x_0)} \delta + \frac{1}{2N^2} \Tr(A - \epsilon I) + O(\delta^2). \nonumber
\end{align}
(Here the the implicit constant in $O(\delta^2)$ depends on $\epsilon$.)

Because $u_{t_0-\delta} \in \mathcal{E}(0,C)$, Lemma \ref{lem:commutation} (2) and (3) imply that if $C \delta \leq 1$, then
\begin{align*}
|Q_\delta P_\delta u_{t_0-\delta}(x_0) - P_\delta Q_\delta u_{t_0-\delta}(x_0)| \leq 2C^2 m \delta^2 + 2C \delta^2 \norm{D(P_\delta Q_\delta u_{t_0-\delta})(x_0)}_2.
\end{align*}
Again by Lemma \ref{lem:errorpropagation} (4) and Theorem \ref{thm:semigroup} (4c), $\norm{D(Q_\delta P_\delta u_{t_0-\delta})(x_0)}_2^2$ is bounded by $\norm{Du_0(x_0)}_2^2$ plus a constant, so that
\[
Q_\delta P_\delta u_{t_0-\delta} = P_\delta Q_\delta u_{t_0-\delta} + O(\delta^2).
\]
Also, if we let $\delta_\ell = 2^{-\ell}$ for $\ell \in \Z$, then Lemma \ref{lem:convergence} implies that when $2C \delta_\ell \leq 1$ and $\delta_\ell < r$, then
\begin{align*}
|P_{\delta_\ell} Q_{\delta_\ell} u_{t_0-\delta}(x_0) - R_{\delta_\ell} u_{t_0 - \delta_\ell}(x_0)| &= |R_{\delta_\ell,\ell} u_{t_0-\delta_\ell}(x_0) - R_{\delta_\ell} u_{t_0-\delta_\ell}(x_0)| \\
&\leq \left( \frac{3}{2} \frac{C^2m\delta_\ell}{1 - C\delta_\ell} + \log(1 + C\delta_\ell)(m + Cm + \norm{Du(x_0)}_2^2 \right) 2^{-\ell} \\
&= O(\delta_\ell^2).
\end{align*}
So overall
\begin{equation} \label{eq:Taylorerror1}
Q_{\delta_\ell} P_{\delta_\ell} u_{t_0-\delta_\ell}(x_0) = R_{\delta_\ell} u_{t_0 - \delta_\ell}(x_0) + O(\delta_\ell^2) = u_{t_0}(x_0) + O(\delta_\ell^2).
\end{equation}

Using similar reasoning, Lemma \ref{lem:commutation} (1) shows that
\[
D(Q_{\delta_\ell} P_{\delta_\ell} u_{t_0-\delta_\ell})(x_0) = D(P_{\delta_\ell} Q_{\delta_\ell} u_{t_0-\delta_\ell})(x_0) + O(\delta_\ell^{3/2}).
\]
Then using Lemma \ref{lem:convergence} (3), we obtain
\[
D(P_{\delta_\ell} Q_{\delta_\ell} u_{t_0-\delta_\ell})(x_0) = D(R_{\delta_\ell} u_{t_0-\delta_\ell})(x_0) + O(\delta^{3/2}).
\]
Finally, because $u_{t_0-\delta} \in \mathcal{E}(0,C)$, it is differentiable everywhere; the upper Taylor approximation \eqref{eq:revisedTaylorexpansion} implies that $u_{t_0}(x) \leq u_{t_0}(x_0) + \ip{p,x - x_0}_2 + o(\norm{x - x_0}_2)$ and therefore $p$ must equal $Du_{t_0}(x_0)$.  Thus, overall
\begin{equation} \label{eq:Taylorerror2}
D(Q_{\delta_\ell} P_{\delta_\ell} u_{t_0 - \delta_\ell})(x_0) = p + O(\delta_\ell^{3/2}).
\end{equation}

Substituting \eqref{eq:Taylorerror1} and \eqref{eq:Taylorerror2} into \eqref{eq:Taylorexpansion3}, we obtain
\[
u_{t_0}(x_0) \geq u_{t_0}(x_0) + \frac{1}{2} \norm{p}_2^2 \delta_\ell + (\alpha + \epsilon)(-\delta_\ell) - \norm{p}_2^2 \delta_\ell + \frac{1}{2N^2} \Tr(A - \epsilon I) + O(\delta_\ell^2).
\]
We cancel $u_{t_0}(x_0)$ from both sides, divide by $\delta_\ell$, and move $\alpha - \epsilon$ to the left hand side to conclude that
\[
\alpha + \epsilon \geq \frac{1}{2N^2} \Tr(A - \epsilon I) - \frac{1}{2} \norm{p}_2^2 + O(\delta_\ell).
\]
Then taking $\ell \to \infty$, we get $\alpha + \epsilon \geq (1/2N^2) \Tr(A - \epsilon I) - (1/2) \norm{p}_2^2$.  Since $\epsilon$ was arbitrary, we have $\alpha \geq (1/2N^2) \Tr(A) - (1/2) \norm{p}_2^2$.  This shows that $u$ is a viscosity supersolution.

To show that the $u$ is a viscosity subsolution, the argument is symmetrical for the most part.  However, to obtain the constant $K_\epsilon$ in \eqref{eq:revisedTaylorexpansion2}, we used the one-sided estimate Theorem \ref{thm:semigroup} (3b) to show that $u$ is bounded below by a quadratic in $x-x_0$ that is independent of $t$, so long as $t \in (t_0-r,t_0]$.  To show that $u$ is a viscosity supersolution, we want to prove an analogous quadratic upper bound.  But by Theorem \ref{thm:semigroup} (3a) and semi-concavity of $u_0$, we have for $t \leq t_0$ that
\begin{align*}
u_t(x) &\leq u_0(x) + \frac{m}{2} \log(1 + Ct_0) \\
&\leq u_0(x_0) + \ip{Du_0(x_0), x - x_0} + \frac{C}{2} \norm{x - x_0}_2^2 + \frac{m}{2} \log(1 + Ct_0),
\end{align*}
which is the desired upper bound.  The rest of the argument is symmetrical except that $\alpha + \epsilon$ is replaced by $\alpha - \epsilon$ and $A - \epsilon I$ is replaced by $A + \epsilon I$.
\end{proof}

\begin{lemma}
Let $u: M_N(\C)_{sa}^m \times [0,+\infty) \to \R$.  Then $u$ is a viscosity solution to $\partial_t u = (1/2N) \Delta u - (1/2) \norm{Du}_2^2$ if and only if $\exp(-N^2u)$ is a viscosity solution to $\partial_t u = (1/2N) \Delta u$.
\end{lemma}

\begin{proof}
More precisely, we claim that $u$ is a viscosity subsolution if and only if $\exp(-N^2 u)$ is viscosity supersolution and vice versa.  Suppose that $u$ is a subsolution, and let us show that $v = \exp(-N^2 u)$ is a supersolution.  If $u$ is upper semi-continuous, then $v$ is lower semi-continuous.  Now suppose that we have a lower Taylor approximation at $(x_0,t_0)$
\[
v(x,t) \geq v(x_0,t_0) + \alpha(t - t_0) + \ip{p, x - x_0}_2 + \frac{1}{2} \ip{A(x - x_0), x - x_0}_2 + o(|t - t_0| + \norm{x - x_0}_2^2).
\]
Note that $v > 0$ and $u = -1/N^2 \log v$.  The function $h \mapsto \log h$ is increasing and analytic for $h > 0$ and we have
\[
\log(h + \delta) = \log(h) + \log(1 + \delta / h) = \log(h) + \frac{\delta}{h} - \frac{1}{2} \left( \frac{\delta}{h} \right)^2 + O(\delta^3).
\]
Substituting $h = v(x_0,t_0) = \exp(-N^2 u(x_0,t_0))$ and $\delta = v(x,t) - v(x_0,t_0) = \alpha(t - t_0) + \ip{p, x - x_0}_2 + \frac{1}{2} \ip{A(x - x_0), x - x_0}_2 + o(|t - t_0| + |x - x_0|^2)$, we get
\begin{align*}
-N^2 u(x,t) &\geq -N^2 u(x_0,t_0) + \frac{\alpha}{v(x_0,t_0)} (t - t_0) + \frac{1}{v(x_0,t_0)} \ip{p, x - x_0}_2 \\
& \quad + \frac{1}{2 v(x_0,t_0)} \ip{A(x - x_0), x - x_0}_2 - \frac{1}{2 v(x_0,t_0)^2} \ip{p,x-x_0}_2^2 \\
& \quad + o(|t - t_0| + \norm{x - x_0}_2^2),
\end{align*}
since $\ip{p,x-x_0}_2 / v(x_0,t_0)^2$ is the only term from $(-1/2) (\delta / h)^2 + O(\delta^3)$ that is not $o(|t - t_0| + \norm{x - x_0}_2^2)$ (here we use the fact that $|t - t_0| \norm{x - x_0}_2 \leq (2/3) |t - t_0|^{3/2} + (1/3) \norm{x - x_0}_2^3$).  Let us denote by $P$ the linear map $P(x - x_0) = p \ip{p, x - x_0}_2$.  Then the above inequality becomes
\begin{align*}
u(x,t) &\leq u(x_0,t_0) - \frac{\alpha}{N^2 v(x_0,t_0)} (t - t_0) - \frac{1}{N^2 v(x_0,t_0)} \ip{p, x - x_0}_2  \\
& \quad - \frac{1}{2N^2 v(x_0,t_0)} \ip{A(x - x_0), x - x_0}_2 + \frac{1}{2N^2 v(x_0,t_0)^2} \ip{P(x - x_0), (x - x_0)} \\
& \quad + o(|t - t_0| + \norm{x - x_0}_2^2).
\end{align*}
Because $u$ is a subsolution, we have
\[
-\frac{\alpha}{N^2 v(x_0,t_0)} \leq -\frac{1}{2N^4} \Tr(A) + \frac{1}{2N^4 v(x_0,t_0)^2} \Tr(P) - \frac{1}{2N^4 v(x_0,t_0)^2} \norm{p}_2^2.
\]
But $\Tr(P) = \norm{p}_2^2$, so the last two terms cancel.  Thus, $\alpha \geq \frac{1}{2N^2} \Tr(A)$ as desired.  So $v$ is a supersolution.

A symmetrical argument shows that if $v$ is a supersolution, then $u$ is a subsolution.  The other two claims are proved in the same way except using the Taylor expansion of the exponential function instead of the logarithm.
\end{proof}

Now we are ready to prove Theorem \ref{thm:solution} in the special case where $u_0$ is bounded below.

\begin{lemma}
Suppose that $u_0 \in \mathcal{E}(0,C)$ is bounded below.  Then $\exp(-N^2 R_t u_0) = P_t[\exp(-N^2u_0)]$.
\end{lemma}

\begin{proof}
Let $v(x,t) = \exp(-N^2 R_t u_0(x))$ and let $w(x,t) = P_t[\exp(-N^2 u_0)](x)$.  Since $u_0$ is bounded below by some constant $K$, we have $R_t u_0 \geq K$ by monotonicity of $R_t$ (see Corollary \ref{cor:Rerrorpropagation} (3)) and the fact that it does not affect constant functions (since the same is true of $P_t$ and $Q_t$).  Hence, $v = \exp(-N^2 R_t u_0) \leq \exp(-N^2K)$.  We also have $\exp(-N^2 u_0) \leq \exp(-N^2 K)$ and hence $w \leq \exp(-N^2K)$.

Thus, $v$ and $w$ are both bounded, $w$ is a smooth solution to the heat equation, and $v$ is a viscosity solution by the previous lemma.  We will conclude from this that $v = w$ (and this is nothing but a standard argument for the maximum principle together with the basic philosophy of viscosity solutions).

To show that $v \leq w$, choose $\epsilon > 0$, and consider the function
\[
\phi(x,t) = w(x,t) - v(x,t) - \epsilon \norm{x}_2^2 - 2m \epsilon t.
\]
Suppose for contradiction that $\phi > 0$ at some point.  Since $\phi$ is continuous on $M_N(\C)_{sa}^m \times [0,+\infty)$ and since $w$ and $v$ are bounded, $\phi$ achieves a maximum at some $(x_0,t_0)$.  Since the maximum is strictly positive, we have $t_0 > 0$.  Let $\psi(x,t) = w(x,t) - (1/2) \epsilon \norm{x}_2^2 - 2 \epsilon t$, so that $\phi(x,t) = v(x,t) - \psi(x,t)$.  Then $\phi(x,t) \leq \phi(x_0,t_0)$ implies that
\begin{align*}
v(x,t) &\geq v(x_0,t_0) + \psi(x,t) - \psi(x_0,t_0) \\
&= v(x_0,t_0) + \partial_t \psi(x_0,t_0)(t - t_0) + \ip{D\psi(x_0,t_0), x - x_0}_2 \\
& \quad + \frac{1}{2} \ip{H\psi(x_0,t_0)(x - x_0), x - x_0}_2 + o(|t - t_0| + \norm{x - x_0}_2^2),
\end{align*}
where the last step follows because $\psi$ is smooth.  Because $v$ is a viscosity supersolution,
\[
\partial_t \psi(x_0,t_0) \geq \frac{1}{2N} \Delta \psi(x_0,t_0).
\]
However, this is a contradiction because at every point $(x,t)$, we have
\[
\partial_t \psi = \partial_t w - 2m \epsilon < \frac{1}{2N} \Delta w - m\epsilon = \frac{1}{2N} \Delta \psi,
\]
by computation and the fact that $w$ solves the heat equation.  It follows that $\phi \leq 0$ and hence $v(x,t) \geq w(x,t) - \frac{1}{2} \epsilon \norm{x}_2^2 - m \epsilon t$.  Since $\epsilon$ was arbitrary $v \geq w$.  Then a symmetrical argument shows that $v \leq w$.
\end{proof}

Thus, to prove Theorem \ref{thm:solution}, it only remains to remove the boundedness assumption on $u_0$.  We achieve this by replacing $u_0$ with the function
\begin{equation} \label{eq:tilting}
\tilde{u}_0(x) = u_0(x) - \ip{Du_0(0),x}_2, 
\end{equation}
which is nonnegative by convexity of $u_0$ and hence bounded below.

\begin{lemma} \label{lem:heattilting}
Let $u_0 \in \mathcal{E}(0,C)$ and let $\tilde{u}_0$ be given by \eqref{eq:tilting}.  Let $v_0 = \exp(-N^2 u_0)$ and $\tilde{v}_0 = \exp(-N^2 \tilde{u}_0)$.  Then the integral defining $P_t \exp(-N^2 u_0)$ is well-defined and also
\[
P_t v_0(x) = \exp(-N^2 \ip{Du_0(0),x} + \frac{N^2 t}{2} \norm{Du_0(0)}_2^2) P_t \tilde{v}_0(x - t Du_0(0))
\]
\end{lemma}

\begin{proof}
We can express $d\sigma_{t,N}(y) = (1/Z_N) \exp(-(N^2/2t) \norm{y}_2^2)\,dy$.  Also, denote $p = Du_0(0)$.  Then
\begin{align*}
P_t v_0(x) &= \frac{1}{Z_N} \int \exp(-N^2 u_0(x+y)) \exp(-(N^2/2t) \norm{y}_2^2)\,dy \\
&= \frac{1}{Z_N} \int \exp(-N^2 \tilde{u}_0(x + y) -N^2 \ip{p,x+y} - \frac{N^2}{2t} \norm{y}_2^2)\,dy \\
&= \frac{1}{Z_N} \int \exp(-N^2 \tilde{u}_0(x + y) - N^2 \ip{p,x} + \frac{N^2t}{2} \norm{p}_2^2 - \frac{N^2}{2t} \norm{y + tp}_2^2)\,dy \\
&= \frac{1}{Z_N} \int \exp(-N^2 \tilde{u}_0(x - tp + z) - N^2 \ip{p,x} + \frac{N^2t}{2} \norm{p}_2^2 - \frac{N^2}{2t} \norm{z}_2^2)\,dy \\
&= \exp(-N^2 \ip{p,x} + \frac{N^2 t}{2} \norm{p}_2^2) P_t \tilde{v}_0(x - tp).  \qedhere
\end{align*}
\end{proof}

\begin{lemma} \label{lem:HJtilting}
Let $u_0 \in \mathcal{E}(0,C)$, let $p \in M_N(\C)_{sa}^m$, and let $\tilde{u}_0(x) = u_0(x) - \ip{p,x}_2$.  Then
\begin{enumerate}
	\item $P_t u_0(x) = P_t \tilde{u}_0(x) + \ip{p,x}_2$.
	\item $Q_t u_0(x) = Q_t \tilde{u}_0(x-tp) + \ip{p,x}_2 - \frac{t}{2} \norm{p}_2^2$.
	\item $R_t u_0(x) = R_t \tilde{u}_0(x-tp) + \ip{p,x}_2 - \frac{t}{2} \norm{p}_2^2$.
\end{enumerate}
\end{lemma}

\begin{proof}
(1) holds because $P_t$ is linear and it does not affect linear functions.  To prove (2), fix $x$ and let $y$ be the point where the infimum defining $Q_t u_0(x)$ is achieved and let $\tilde{y}$ be the point where the infimum defining $Q_t \tilde{u}_0(x - tp)$ is achieved.  By Corollary \ref{cor:infconvolution} (1), the points $y$ and $\tilde{y}$ are characterized respectively by the relations
\[
y = x - t Du_0(y), \qquad \tilde{y} = x - tp - tD\tilde{u}_0(\tilde{y}).
\]
But $D\tilde{u}_0(\tilde{y}) = Du_0(\tilde{y}) - p$.  Thus, $x - tDu_0(\tilde{y}) = \tilde{y}$, so that $y = \tilde{y}$.
\begin{align*}
Q_t u_0(x) &= u_0(y) + \frac{1}{2t} \norm{y - x}_2^2 \\
&= \tilde{u}_0(y) + \ip{p,y}_2 + \frac{1}{2t} \norm{y - x}_2^2 \\
&= \tilde{u}_0(y) + \ip{p,x}_2 - \frac{t}{2} \norm{p}_2^2 + \frac{1}{2t} \norm{y - (x - tp)}_2^2 \\
&= Q_t \tilde{u}_0(x-tp) + \ip{p,x}_2 - \frac{t}{2} \norm{p}_2^2.
\end{align*}
(3) It follows by iteration (after some computation) that for $t \in 2^{-\ell} \N_0$, we have $R_{t,\ell}u_0(x) = R_t \tilde{u}_0(x-tp) + \ip{p,x}_2 - \frac{t}{2} \norm{p}_2^2$.  Then by Lemma \ref{lem:convergence}, we may take $\ell \to \infty$, and by Theorem \ref{thm:semigroup} (3), we may extend the inequality to all real $t$.
\end{proof}

\begin{proof}[Proof of Theorem \ref{thm:solution}]
We have already proved the case where $u_0$ is bounded.  For the general case, let $u_0 \in \mathcal{E}(0,C)$.  Define $p = Du_0(0)$ and $\tilde{u}_0(x) = u_0(x) - \ip{p,x}_2$.  As remarked above, $\tilde{u}_0$ is bounded below by zero.  By Lemmas \ref{lem:heattilting}, the bounded case, and \ref{lem:HJtilting},
\begin{align*}
P_t \exp(-N^2 u_0)(x) &= \exp\left(-N^2 \ip{p,x} + \frac{N^2 t}{2} \norm{p}_2^2\right) [P_t \exp(-N^2\tilde{u}_0)](x - tp) \\
&= \exp\left(-N^2 \ip{p,x} + \frac{N^2 t}{2} \norm{p}_2^2\right) \exp(-N^2 R_t \tilde{u}_0(x - tp)) \\
&= \exp\left(-N^2\left(R_t \tilde{u}_0(x-tp) + \ip{p,x} - \frac{t}{2} \norm{p}_2^2 \right) \right) \\
&= \exp(-N^2 R_t u_0(x)).
\end{align*}
In particular, since $P_t \exp(-N^2 \tilde{u}_0)$ is smooth for $t > 0$, we see that all the functions in the above equation are smooth for $t > 0$, and hence $R_t u_0(x)$ is smooth function of $(x,t)$.  Also, $P_t[\exp(-N^2 u_0)] = \exp(-N^2 R_t u_0)$ as desired.
\end{proof}

\subsection{Approximation by Trace Polynomials} \label{subsec:approximation1}

Now we are ready to prove that $R_t$ preserves asymptotic approximability by trace polynomials.

\begin{proposition} \label{prop:Rtpreservesapproximation}
Let $\{V_N\}$ be a sequence of functions $M_N(\C)_{sa}^m \to \R$ such that $V_N$ is convex and $V_N(x) - (C/2) \norm{x}_2^2$ is concave, and $\{DV_N\}$ is asymptotically approximable by trace polynomials.  Then for every $t > 0$, the sequences $\{D(P_tV_N)\}$, $\{D(Q_tV_N)\}$, and $\{D(R_tV_N)\}$ are asymptotically approximable by trace polynomials.
\end{proposition}

\begin{proof}
The fact that $\{D(P_tV_N)\}$ is asymptotically approximable by trace polynomials follows from Lemma \ref{lem:convolutionapproximation}.

Now consider $D(Q_tV_N)$.  Note that by Corollary \ref{cor:infconvolution} (1), $D(Q_tV_N)(x)$ is the solution of the fixed point equation
\[
y = DV_N(x - ty).
\]
But if $t < 1/C$, then $y \mapsto DV_N(x - ty)$ is a contraction and thus iterates of this function will converge to the fixed point.  Let us define $\phi_{N,0}(x) = 0$ and $\phi_{N,\ell+1}(x) = DV_N(x - t \phi_{N,\ell}(x))$.  By Lemma \ref{lem:infconvolution} (5), the distance from the fixed point $D(Q_tV_N)(x)$ from $0$ is bounded by $\norm{DV_N(x)}_2$, hence,
\[
\norm{\phi_{N,\ell}(x) - D(Q_tV_N)(x)}_2 \leq C^\ell t^\ell \norm{DV_N(x)}_2.
\]
Because $DV_{N,t}$ is $C$-Lipschitz, Lemma \ref{lem:compositionapproximation} implies that $\{\phi_{N,\ell}\}_N$ is asymptotically approximable by trace polynomials.

Now $\norm{DV_N(0)}_2$ is bounded by some constant $A$ as $N \to \infty$ because $DV_N$ is asymptotically approximable by trace polynomials.  Since $DV_N$ is also $C$-Lipschitz, $\norm{DV_N(x)}_2 \leq A + C\norm{x}_2$.  In particular, $\norm{\phi_{N,\ell}(x) - D(Q_tV_N)(x)}_2 \leq C^\ell t^\ell (A +  C\norm{x}_2)$.  Thus, by Lemma \ref{lem:convergenceapproximation}, $\{D(Q_tV_N)\}$ is asymptotically approximable by trace polynomials.

This holds whenever $t < 1/C$.  But for general $t$, we can write $Q_t = Q_{t/n}^n$ where $n$ is large enough that $t / n < 1/C$, and then iterating the previous statement shows that $\{Q_t V_N\}$ is asymptotically approximable by trace polynomials.

For the sequence $\{D(R_tV_N)\}$, first note that when $t \in \Q_2^+$, we know $\{D(R_{t,\ell}V_N)\}$ is asymptotically approximable by trace polynomials.  By Theorem \ref{thm:semigroup} (1c) and Lemma \ref{lem:convergenceapproximation}, the sequence $\{D(R_tV_N)\}$ is asymptotically approximable by trace polynomials for $t \in \Q_2^+$.  Finally, by Theorem \ref{thm:semigroup} (2d) and Lemma \ref{lem:convergenceapproximation}, the sequence $\{D(R_t V_N)\}$ is asymptotically approximable by trace polynomials for all $t \in \R^+$.
\end{proof}

\section{Main Theorem on Free Entropy} \label{sec:mainthm}

We are now ready to prove the following theorem which shows that $\chi = \chi^*$ for a law which is the limit of log-concave random matrix models.

\begin{theorem} \label{thm:mainthm}
Let $\mu_N$ be a sequence of probability measures on $M_N(\C)_{sa}^m$ given by the potential $V_N$.  Assume
\begin{enumerate}[(A)]
	\item The potential $V_N(x)$ is convex and $V_N(x) - (C/2) \norm{x}_2^2$ is concave for some $C > 0$ independent of $N$.
	\item The sequence $\mu_N$ concentrates around some non-commutative law $\lambda$ with $\lambda(X_j^2) > 0$.
	\item For some $R_0 > 0$, we have $\lim_{N \to \infty} \int_{\norm{x} \geq R_0} (1 + \norm{x}_2^2) \,d\mu_N(x) = 0$.
	\item The sequence $\{DV_N\}$ is asymptotically approximable by trace polynomials.
\end{enumerate}
Then $\lambda \in \Sigma_{m,R_0}$ and moreover
\begin{enumerate}
	\item The law $\lambda$ has finite Fisher information $\Phi^*(\lambda)$, and for all $t \geq 0$, we have
	\[
	\lim_{N \to \infty} \frac{1}{N^3} \mathcal{I}(\mu_N * \sigma_{t,N}) \to \Phi^*(\lambda \boxplus \sigma_t).
	\]
	\item We have for all $t \geq 0$,
	\[
	\chi(\lambda \boxplus \sigma_t) = \underline{\chi}(\lambda \boxplus \sigma_t) = \lim_{N \to \infty} \frac{1}{N^2} \left( h(\mu_N * \sigma_{t,N}) + \frac{m}{2} \log N \right) =  \chi^*(\lambda \boxplus \sigma_t).
	\]
	\item The functions $t \mapsto \frac{1}{N^3} \mathcal{I}(\mu_N * \sigma_{t,N})$ and $t \mapsto \Phi^*(\lambda \boxplus \sigma_t)$ are decreasing and Lipschitz and and the absolute value of the derivative (where defined) is bounded by $C^2m(1 + Ct)^{-2}$.
\end{enumerate}
\end{theorem}

\begin{remark}
If $V_N(x) - (c/2) \norm{x}_2^2$ is convex and $V_N(x) - (C/2) \norm{x}_2^2$ is concave and if $\{DV_N\}$ is asymptotically approximable by trace polynomials, then Theorem \ref{thm:convergenceandconcentration} implies that $\mu_N$ satisfies the hypotheses of Theorem \ref{thm:mainthm} for some law $\lambda$.

However, Theorem \ref{thm:mainthm} holds in a slightly more general situation than Theorem \ref{thm:convergenceandconcentration} in that we do not have to assume uniform convexity, finite moments, or exponential concentration.
\end{remark}

In preparation for the proof of the Theorem \ref{thm:mainthm}, we have already verified that the hypotheses (A), (C), and (D) are preserved under Gaussian convolution.  Now we show that (B) is preserved in Lemma \ref{lem:concentrationpreserved}.  This is straightforward apart from one subtlety | although we have assumed that for every non-commutative polynomial $p$, the non-commutative moment $\tau_N(p(x))$ concentrates around $\lambda(p)$ under $\mu_N$, we have \emph{not} assumed that $|\tau_N(p(x))|$ has finite expectation.  To deal with this issue, we first prove an auxiliary lemma.

\begin{lemma}
Let $\lambda$ be a non-commutative law in $\Sigma_m$, let $p(X,Y) = p(X_1,\dots,X_m,Y_1,\dots,Y_m)$ be a non-commutative polynomial of $2m$ variables, and let $R > 0$.  Then there exists a neighborhood $\mathcal{V}$ of $\lambda$ in $\Sigma_m$ and a constant $L$ such that, for all $N \in \N$, for all $x \in \Gamma_N(\mathcal{V})$, the function $y \mapsto \tau_N(p(x,y))$ is $L$-Lipschitz with respect to $\norm{\cdot}_2$ for self-adjoint tuples $y$ in the operator-norm ball $\{y: \norm{y_j} \leq R\}$.
\end{lemma}

\begin{proof}
To prove the lemma, it suffices to consider the case of a non-commutative monomial.  Indeed, if $p = \sum_{j=1}^n p_j$ where $p_j$ is a monomial, and if we find neighborhoods $\mathcal{V}_j$ and Lipschitz constants $L_j$ for each $p_j$, then the result will also hold for $p$ with $\mathcal{V} = \bigcap_{j=1}^n \mathcal{V}_j$ and $L = \sum_{j=1}^n L_j$.

Thus, assume without loss of generality that $p(X,Y)$ is a non-commutative monomial.  Then it can be written in the form
\[
p(X,Y) = q_0(X) Y_{i_1} q_1(X) Y_{i_2} \dots q_{\ell-1}(X) Y_{i_\ell} q_\ell(X).
\]
where $i_j \in \{1,\dots,m\}$ and $q_j(X)$ is a non-commutative monomial in $X$ (which of course is allowed to be $1$).  Consider $x, y, y' \in M_N(\C)_{sa}^m$, and suppose that $\norm{y_i}_\infty \leq R$ and $\norm{y_i'}_\infty \leq R$ for each $i$.  Then
\[
p(x,y) - p(x,y') = \sum_{j=1}^\ell q_0(x) y_{i_1} \dots y_{i_{j-1}} q_{j-1}(x)(y_{i_j} - y_{i_j}') q_i(x) y_{i_{j+1}} \dots y_{i_\ell} q_\ell(x).
\]
Recalling the non-commutative $L^\alpha$ norms and H\"older's inequality (see \ref{subsec:NCLP}), we have
\[
\norm{p(x,y) - p(x,y')}_1 \leq \left( \sum_{j=1}^\ell \prod_{k \neq j} \norm{q_j(x)}_{2(\ell+1)} \prod_{k < j} \norm{y_{i_k}}_\infty \prod_{k > j} \norm{y_{i_k}'}_\infty \right) \norm{y_j - y_j'}_2.
\]
This implies that
\[
|\tau_N(p(x,y)) - \tau_N(p(x,y'))| \leq \left( \sum_{j=1}^\ell \prod_{k \neq j} \norm{q_j(x)}_{2(\ell+1)} \right) R^{\ell-1} \norm{y - y'}_2.
\]
Now
\[
\norm{q_j(x)}_{2(\ell+1)} = \left( \tau_N[ (q_j(x)^* q_j(x))^{\ell+1} ] \right)^{1/2(\ell+1)}.
\]
We can define
\[
\mathcal{V} = \{\lambda': \lambda'[ (q_j^*q_j)^{\ell+1} ] < \lambda[ (q_j^*q_j)^{\ell+1} ] + 1 \text{ for } j = 0, \dots, \ell\}.
\]
Then $\norm{q_j(x)}_{2(\ell+1)}$ is uniformly bounded for $x \in \Gamma_N(\mathcal{V})$ for each $j = 0, \dots, \ell$.  Suppose that each of these quantities is bounded by $K$.  Then the above estimate shows that
\[
|\tau_N(p(x,y)) - \tau_N(p(x,y))| \leq \ell K^{\ell+1} R^{\ell-1} \norm{y - y'}_2
\]
whenever $x \in \Gamma_N(\mathcal{V})$ and $y, y'$ are in the operator-norm ball of radius $R$.
\end{proof}

\begin{lemma} \label{lem:concentrationpreserved}
Suppose that $\{\mu_N\}$ concentrates around a non-commutative law $\lambda$.  Then $\{\mu_N * \sigma_{t,N}\}$ concentrates around $\lambda \boxplus \sigma_t$ for every $t > 0$.
\end{lemma}

\begin{proof}
Fix $t$.  Let $X_N = (X_{N,1},\dots,X_{N,m})$ and $Y_N = (Y_{N,1},\dots,Y_{N,m})$ be independent random variables with the laws $\mu_N$ and $\sigma_{t,N}$ respectively.  Because the topology on the space $\Sigma_m$ of non-commutative laws is generated by non-commutative moments, it suffices to show that for each non-commutative polynomial $p$ and $\delta > 0$,
\[
\lim_{N \to \infty} P(|\tau_N(p(X_N+Y_N)) - \lambda \boxplus \sigma_t(p)| \geq \delta) = 0.
\]
Fix $p$ and let $d$ be its degree.  By the previous lemma, there is a neighborhood $\mathcal{V}$ of $\lambda$ and a constant $K$ such that for every $x \in \Gamma_N(\mathcal{V})$, the function $y \mapsto \tau_N(p(x+y))$ is $K$-Lipschitz with respect to $\norm{\cdot}_2$ on the operator-norm ball $\{y = (y_1,\dots,y_m): \norm{y_j} \leq 4t^{1/2}\}$.  By shrinking $\mathcal{V}$ if necessary, we may also assume that on $\tau_N(q(x))$ is uniformly bounded for every non-commutative monomial $q(x)$ of degree less than or equal to $d$.

Choose a $C_c^\infty$ function $\psi: \R \to \R$ such that $\psi(z) = z$ for $|z| \leq 3t^{1/2}$ and $|\psi(z)| \leq 4t^{1/2}$.  Then $\Psi: (y_1,\dots,y_m) \mapsto (\psi(y_1),\dots,\psi(y_m))$ is globally Lipschitz in $\norm{\cdot}_2$ and it also maps $M_N(\C)_{sa}^m$ into the operator norm ball of radius $4t^{1/2}$ (which is the region where $z \mapsto \tau_N(p(x,z))$ was assumed to $K$-Lipschitz with respect to $\norm{\cdot}_2$ whenever $x \in \Gamma_N(\mathcal{V})$).  This implies that there is some constant $K'$ such that $y \mapsto \tau_N(p(x,\Psi(y)))$ is $K'$-Lipschitz for all $x \in \Gamma_N(\mathcal{V})$.  Let
\begin{align*}
\alpha_N(x) &= E[\tau_N(p(x+\Psi(Y_N))] \\
\beta_N(x) &= E[\tau_N(p(x+Y+N))] = \exp(tL_N/2) [\tau(p)](x) \\
\beta(x) &= \exp(tL/2)[\tau(p)](x)
\end{align*}
Therefore, by Theorem \ref{thm:concentration} applied to $Y_N$,
\[
x \in \Gamma_N(\mathcal{V}) \implies P(|\tau_N(p(x+\Psi(Y_N))) - \alpha_N(x)| \geq \delta/3) \leq 2e^{-\delta^2 N^2 / 18 t (K')^2}.
\]
On the other hand, we know by standard tail estimates on the GUE (see Corollary \ref{cor:operatornormtail}) that
\[
\lim_{N \to \infty} E[\tau_N(q(Y_N) \mathbf{1}_{\norm{Y_N} \geq 3t^{1/2}}] = 0
\]
for every non-commutative polynomial $q$.  This implies that $|\alpha_N(x) - \beta_N(x)| \to 0$ uniformly for $x \in \Gamma_N(\mathcal{V})$.  On the other hand, by Lemma \ref{lem:Laplacianlimit},
\[
\beta_N(x) = \exp(tL_N/2)[\tau(p)](x) \to \exp(tL/2)[\tau(p)](x) = \beta(x)
\]
where the convergence occurs coefficient-wise.  Now $\exp(tL_N/2)[\tau(p)]$ is a sum of products of traces of non-commutative monomials $q$ of degree $\leq d$ and for every such $q$, we know $\tau_N(q(x))$ is uniformly bounded on $\Gamma_N(\mathcal{V})$ by our choice of $\mathcal{V}$.  Thus, coefficient-wise convergence of $\beta_N \to \beta$ implies uniform convergence for $x \in \Gamma_N(\mathcal{V})$.  Therefore, for sufficiently large $N$ we have $|\beta_N(x) - \beta_N(x)| \leq \delta / 3$ for $x \in \Gamma_N(\mathcal{V})$, and hence
\[
P\Bigl( \bigl|\tau_N(p(X_N+Y_N)) - \tau(\beta(X_N))\bigr| \geq 2\delta/3, \, X_N \in \Gamma_N(\mathcal{V}), \, \norm{Y_N} \leq 3t^{1/2}\Bigr) \leq 2e^{-\delta^2 N^2 / 18 t (K')^2},
\]
where we have applied the Fubini-Tonelli theorem for the product measure $\mu_N \otimes \sigma_{t,N}$.  By our concentration assumption,
\[
P\Bigl( \bigl|\tau_N(\beta(X_N)) - \lambda[\beta] \bigr| \geq \delta / 3 \Bigr) \to 0, \qquad P(X_N \in \Gamma_N(\mathcal{V})) \to 1,
\]
and by Corollary \ref{cor:operatornormtail} also $P(\norm{Y_k} \geq 3t^{1/2}) \to 0$.  Altogether, we have
\[
P\Bigl(\bigl|\tau_N(p(X_N+Y_N)) - \lambda[\beta] \bigr| \geq \delta \Bigr) \to 0.
\]
But note that $\lambda[\beta] = \lambda[\exp(tL/2)[\tau(p)]] = \lambda \boxplus \sigma_t[p]$ by Lemma \ref{lem:freeconvolutionevaluation}.  Thus, the proof is complete.
\end{proof}

\begin{proof}[Proof of Theorem \ref{thm:mainthm}]
Let $V_{N,t} = R_t V_N$ be the potential associated to $\mu_N * \sigma_{t,N}$.  Let us verify that $V_{N,t}$ satisfies the assumptions (A) - (D) for every $t > 0$.
\begin{enumerate}[(A)]
	\item This follows from Theorem \ref{thm:semigroup} (1) because $V_{N,t} = R_t V_N$, hence $V_{N,t} \in \mathcal{E}(0,C)$.
	\item This follows from Lemma \ref{lem:concentrationpreserved}.
	\item This follows from tail bounds on the GUE (Corollary \ref{cor:operatornormtail}).
	\item This follows from Proposition \ref{prop:Rtpreservesapproximation}.
\end{enumerate}

Next, the fact that $\lambda \in \Sigma_{m,R_0}$ follows from Proposition \ref{prop:convergenceofentropy} (1) with $n = 1$.

Claim (1) of the theorem follows by applying Proposition \ref{prop:convergenceofFisherinfo} to $\mu_N * \sigma_{t,N}$ with $n = 1$.

For claim (2), recall that by Lemma \ref{lem:integrateFisher}, equation \eqref{eq:integrateFisher2},
\begin{equation} \label{eq:integrationofFisher}
\frac{1}{N^2} h(\mu_N) + \frac{m}{2} \log N = \frac{1}{2} \int_0^\infty \left( \frac{m}{1+t} - \frac{1}{N^3} \mathcal{I}(\mu_N * \sigma_{t,N}) \right) ds + \frac{m}{2} \log 2\pi e.
\end{equation}
Because $N^{-3} \mathcal{I}(\mu_N)$ converges as $N \to \infty$, there is some constant $K$ with $N^{-3} \mathcal{I}(\mu_N) \leq K$ for all $N$.  Also, because of assumptions (B) and (C), we have $\int \norm{x}_2^2\,d\mu_N(x) \to \sum_{j=1}^m \lambda_j(X_j^2) > 0$.  Therefore, there is a constant $a$ such that $\int \norm{x}_2^2\,d\mu_N(x) \geq ma$ for large enough $N$.  Thus, \eqref{eq:Fisherinfoestimate}, we have for sufficiently large $N$ that
\[
\frac{m}{a+t} \leq \frac{1}{N^3} \mathcal{I}(\mu_N * \sigma_{t,N}) \leq \min \left(M, \frac{m}{t} \right).
\]
Thus, we can apply the dominated convergence theorem to take the limit as $N \to \infty$ inside the integral on the right hand side of \eqref{eq:integrationofFisher} and apply claim (1) to conclude that
\[
\lim_{N \to \infty} \left( \frac{1}{N^2} h(\mu_N) + \frac{m}{2} \log N \right) \to \chi^*(\lambda).
\]
On the left hand side of \eqref{eq:integrationofFisher}, we will apply Proposition \ref{prop:convergenceofentropy} with $n = 1$.  We may replace $V_N$ by $V_N - V_N(0)$ without changing $\mu_N$ (because the definition of $\mu_N$ includes the normalizing constant $Z_N$ anyway).  Then because $\{DV_N\}$ is asymptotically approximable by trace polynomials, we know that $\{V_N\}$ is asymptotically approximable by trace polynomials (Lemma \ref{lem:gradientapproximation}).  Therefore, the hypotheses of Proposition \ref{prop:convergenceofentropy} are satisfied and so
\[
\chi(\lambda) = \limsup_{N \to \infty} \left( \frac{1}{N^2} h(\mu_N) + \frac{m}{2} \log N \right) = \chi^*(\lambda)
\]
and the same holds for $\underline{\chi}(\lambda)$.  Moreover, this holds for $\mu_N * \sigma_{t,N}$ just as well as $\mu_N$ because $\mu_N * \sigma_{t,N}$ satisfies the same assumptions (A) - (D).

For claim (3), first fix $N$ and let $X$ be a random variable with law $\mu_N$, and let $Y_t$ be an independent Hermitian Brownian motion (here $Y_t \sim \sigma_{t,N}$).  Let $\Xi_t = DV_{N,t}(X + Y_t)$, which is the conjugate variable of $X + Y_t$.  Then
\[
\frac{1}{N^3} \mathcal{I}(\mu_N * \sigma_{t,N}) = E \norm{\Xi_t}_2^2
\]
Suppose $0 \leq s \leq t \leq T$.  Then $X + Y_t$ is the sum of the independent random variables $X + Y_s$ and $Y_t - Y_s$, and thus $\Xi_t = E[\Xi_s | X + Y_t]$ by Lemma \ref{lem:conditionalexpectation}.  In other words, $\Xi_t$ is the orthogonal projection of $DV_{N,s}(X+Y_s)$ onto the space of $L^2$ random variables that are functions of $X + Y_t$, or in other words it is the closest function of $X + Y_t$ to $\Xi_s$ in $L^2$.  This implies that
\begin{align*}
\left[ \norm{\Xi_s - \Xi_t}_2^2 \right] &\leq E \left[ \norm{DV_{N,s}(X+Y_s) - DV_{N,s}(X+Y_t)}_2^2 \right] \\
&\leq E \left[ \frac{C^2}{(1 + Cs)^2} \norm{Y_s - Y_t}_2^2 \right] \\
&= \frac{C^2}{(1 + Cs)^2} m(t - s)
\end{align*}
using the fact that $V_{N,s} \in \mathcal{E}(0,C(1+Ct)^{-1})$ and hence $DV_{N,s}$ is $C(1+Cs)^{-1}$-Lipschitz.  Since $\Xi_t$ is the orthogonal projection of $\Xi_s$ onto this subspace, we know $\Xi_s - \Xi_t$ is orthogonal to $\Xi_t$ and hence
\[
E\left[\norm{\Xi_s}_2^2 \right] - E\left[\norm{\Xi_t}_2^2 \right] = E \left[\norm{\Xi_s - \Xi_t}_2^2 \right].
\]
Overall,
\[
0 \leq \frac{1}{N^3} \mathcal{I}(\mu_N * \sigma_{s,N}) - \frac{1}{N^3} \mathcal{I}(\mu_N * \sigma_{t,N}) \leq \frac{C^2}{(1 + Cs)^2} m(t - s).
\]
This immediately proves that $t \mapsto N^{-3} \mathcal{I}(\mu_N * \sigma_{t,N})$ is decreasing function of $t$, it is Lipschitz, and the absolute value of the derivative is bounded by $C^2 m / (1 + Ct)^2$.  The same holds for $\Phi^*(\lambda \boxplus \sigma_t)$ by taking the limit as $N \to \infty$.
\end{proof}

\section{Free Gibbs Laws}

In the situation of Theorem \ref{thm:convergenceandconcentration}, we want to interpret the law $\lambda$ as the free Gibbs state for a potential which is the limit of the $V_N$'s.  To this end, we will define a non-commutative function space where each point is a limit of functions on $M_N(\C)_{sa}^m$.  We will then give several characterizations of the closure of trace polynomials in this space, as well as the class of potentials to which our previous results apply.

\subsection{Asymptotic Approximation and Function Spaces}

Let $Y_\bullet = \{Y_N\}$ be a sequence of normed vector spaces.  We define a (possibly infinite) semi-norm on sequences $\phi_\bullet = \{\phi_N\}$ of functions $M_N(\C)_{sa}^m \to Y_N$ by
\[
\norm{\phi_\bullet}_{R,Y_\bullet} = \limsup_{N \to \infty} \sup_{\norm{x} \leq R} \norm{\phi_N(x)}_{Y_N}.
\]
Let $\mathcal{F}_m(Y_\bullet)$ be the vector space
\[
\{\phi_\bullet:  \norm{\phi_\bullet}_{R,Y_\bullet} < +\infty \text{ for all } R\} / \{\phi_\bullet: \norm{\phi_\bullet}_{R,Y_\bullet} = 0 \text{ for all } R\}.
\]
For a sequence $\phi_\bullet$, we denote its equivalence class by $[\phi_\bullet]$.

We equip $\mathcal{F}_m(Y_\bullet)$ with the topology generated by the seminorms $\norm{\cdot}_{R,Y_\bullet}$, or equivalently given by the metric
\begin{equation} \label{eq:metric}
d_{\mathcal{F}_m(Y_\bullet)}(\phi_\bullet, \psi_\bullet) = \sum_{n=1}^\infty \frac{1}{2^n} \min(\norm{\phi_\bullet - \psi_\bullet}_{n,Y},1).
\end{equation}
Note that $\mathcal{F}_m(Y_\bullet)$ is a complete metric space in this metric and is a locally convex topological vector space.

The vector space of scalar-valued trace polynomials $\TrP_m^0$ embeds into $\mathcal{F}_m^0 := \mathcal{F}_m(\C)$ by the map that sends a trace polynomial to the corresponding sequence of functions it defines on $M_N(\C)_{sa}^m$.  A sequence $\phi_\bullet$ is asymptotically approximable by trace polynomials if and only if $[\phi_\bullet]$ is in the closure of $\TrP_m^0$ in $\mathcal{F}_m^0$, which we will denote by $\mathcal{T}_m^0$.

Similarly, let $M_\bullet(\C)^m$ be the sequence $\{M_N(\C)^m\}$ equipped with $\norm{\cdot}_2$.  The vector space $\TrP_m^1$ embeds into $\mathcal{F}_m^1 :=\mathcal{F}_m(M_\bullet(\C))$.  A sequence $\phi_\bullet$ of functions $M_N(\C)_{sa}^m \to M_N(\C)_{sa}$ is asymptotically approximable by trace polynomials if and only if $[\phi_\bullet]$ is in the closure of $\TrP_m^1$, which we denote by $\mathcal{T}_m^1$.

The spaces $\mathcal{T}_m^0$ and $\mathcal{T}_m^1$ can be viewed as non-commutative function spaces through the following alternative characterization.  Here $\mathcal{R}$ denotes the hyperfinite $\II_1$ factor and $\mathcal{R}^\omega$ denotes its ultrapower.  For background, see \cite[\S 1.6 and \S 5.4]{AP2017} or \cite[{p.\ 5 - 7}]{Capraro2010}.

\begin{lemma} \label{lem:threeseminorms}
Let $f \in \TrP_m^0$.  Then we have
\begin{equation} \label{eq:threeseminorms}
\limsup_{N \to \infty} \sup_{\substack{x \in M_N(\C)_{sa}^m \\ \norm{x}_\infty \leq R}} |f(x)| = \sup_N \sup_{\substack{x \in M_N(\C)_{sa}^m \\ \norm{x}_\infty \leq R}} |f(x)| = \sup_{\substack{x \in (\mathcal{R}_{sa}^\omega)^m \\ \norm{x}_\infty \leq R}} |f(x)|.
\end{equation}
If we denote the common value by $\norm{f}_{\mathcal{T}_m^0,R}$, then this family of seminorms defines a metrizable topology on $\TrP_m^0$ with the metric given as in \eqref{eq:metric}, and $\mathcal{T}_m^0$ is the completion of $\TrP_m^0$ in this metric.  The same result holds for $\mathcal{T}_m^1$ using the seminorm
\begin{equation} \label{eq:threeseminorms2}
\limsup_{N \to \infty} \sup_{\substack{x \in M_N(\C)_{sa}^m \\ \norm{x}_\infty \leq R}} \norm{f(x)}_2 = \sup_N \sup_{\substack{x \in M_N(\C)_{sa}^m \\ \norm{x} \leq R}} \norm{f(x)}_2 = \sup_{\substack{x \in (\mathcal{R}_{sa}^\omega)^m \\ \norm{x}_\infty \leq R}} \norm{f(x)}_2.
\end{equation}
\end{lemma}

\begin{proof}
Fix $f$ and let $A$, $B$, and $C$ be the three quantities in \eqref{eq:threeseminorms} from left to right.  It is clear that $A \leq B$.  Moreover, $B \leq C$ because there is an isometric trace-preserving embedding of $M_N(\C)$ into $\mathcal{R}^\omega$.  To show that $C \leq A$, pick $x \in (\mathcal{R}_{sa}^\omega)^m$ with $\norm{x} \leq R$.  Then there exists $x_n \in \mathcal{R}_{sa}^m$ with $\norm{x_n} \leq R$ and $x = \lim_{n \to \omega} x_n$.  For each $n$, we can choose an $N_n$, an embedding $M_{N_n}(\C) \to \mathcal{R}$ and a $y_n \in M_{N_n}(\C)$ such that $\norm{y_n} \leq R$ and $\norm{x_n - y_n}_2 \leq 1/2^n$ and $\lim_{n \to \infty} N_n = +\infty$.  Then $x = \lim_{n \to \omega} y_n$ and $|f(x)| = \lim_{n \to \omega} |f(y_n)| \leq A$.  This shows that the three seminorms in \eqref{eq:threeseminorms} are equal, and the other claims follow because these seminorms are the same as the seminorms for $\mathcal{F}_m^0$.
\end{proof}

From this point of view, any $f \in \mathcal{T}_m^0$ has a canonical sequence that represents its equivalence class in $\mathcal{F}_m^0$, constructed as follows.  If we write $f$ as the limit of a sequence of trace polynomials $f^{(k)}$, then $f_{(k)}|_{M_N(\C)_{sa}^m}$ converges locally uniformly on $M_N(\C)_{sa}^m$ as $k \to \infty$ and the limit is independent of the approximating sequence $f^{(k)}$.  We can therefore define $f|_{M_N(\C)_{sa}^m}$ to be this limit.

Similarly, $f$ defines a function on $(\mathcal{R}_{sa}^\omega)^m$.  Moreover, if $(\mathcal{M},\tau)$ is a tracial von Neumann algebra and there is a trace-preserving embedding $\iota: \mathcal{M} \to \mathcal{R}^\omega$, then we may define $f|_{\mathcal{M}} = f \circ \iota$.  It is easy to see that this is independent of the choice of trace-preserving embedding if $f$ is a trace polynomial, and this holds for general $f \in \mathcal{T}_m^0$ or $\mathcal{T}_m^1$ by density of trace polynomials.  In this sense, $\mathcal{T}_m^0$ and $\mathcal{T}_m^1$ represent spaces of universal scalar- or operator-valued functions that can be applied to self-adjoint operators in every $\mathcal{R}_{sa}^\omega$-embeddable tracial von Neumann algebra.

In the scalar-valued case, we have yet another characterization of $\mathcal{T}_m^0$:

\begin{lemma}
Let $\Sigma_{m,\text{bdd}} = \bigcup_{R > 0} \Sigma_{m,R}$.  Let $C(\Sigma_{m,\text{bdd}})$ be the space of functions $g: \Sigma_{m,\text{bdd}} \to \C$ such that $g \in C(\Sigma_{m,R})$ for all $R$, equipped with the family of seminorms $\norm{\cdot}_{C(\Sigma_{m,R})}$.  Then $\mathcal{T}_m^0$ is isomorphic to $C(\Sigma_{m,\text{bdd}})$ as a topological vector space.
\end{lemma}

\begin{proof}
For a scalar-valued trace polynomial $f$, the value $f(x)$ only depends on the law of $x$, so that $f(x) = g(\lambda_x)$ for some function $g: \Sigma_m \to \R$ such that $g \in C(\Sigma_{m,R})$ for all $R$, and we have
\[
\norm{f}_{\mathcal{T}_m^0,R} = \norm{g}_{C(\Sigma_{m,R})}.
\]
Passing to the completion with respect to the metric defined as in \eqref{eq:metric}, we have a map $\iota: \mathcal{T}_m^0 \to C(\Sigma_{m,\text{bdd}})$ which is an isomorphism onto its image.  To show that $\iota$ is surjective, note the algebra of trace polynomials is self-adjoint and separates points in $\Sigma_{m,R}$, and hence by the Stone-Weierstrass theorem, trace polynomials are dense in $C(\Sigma_{m,R})$ for every $R$.  Therefore, if $g \in C(\Sigma_{m,R})$, we can choose a trace polynomial $g^{(k)}(\lambda_x) = f^{(k)}(x)$ such that $\norm{g - g^{(k)}}_{C(\Sigma_{m,k})} \leq 1/2^k$.  Then $f^{(k)}$ converges to some $f$ in $\mathcal{T}_m^0$, and we have $\iota(f) = g$.
\end{proof}

\subsection{Convex Differentiable Functions}

Now we are ready to characterize the type of convex functions which occur in Theorem \ref{thm:mainthm}.  First of all, we let $\mathcal{T}_m^{0,1}$ be the completion of the trace polynomials with respect to the metric
\[
d(f,g) = \sum_{n=1}^\infty \frac{1}{2^n} [\min(1,\norm{f - g}_{\mathcal{T}_m^0,n}) + \min(1, \norm{Df - Dg}_{(\mathcal{T}_m^1)^m,n})].
\]
Observe that if $f \in \mathcal{T}_m^{0,1}$ and $f^{(k)}$ is a sequence of trace polynomials converging to $f$ in $\mathcal{T}_m^{0,1}$ as $k \to \infty$, then $Df^{(k)}$ converges in $(\mathcal{T}_m^1)^m$ and the limit is independent of the choice of approximating sequence.  We denote this limit by $Df$.

\begin{remark}
If $f$ and $f^{(k)}$ are as above, then since $Df^{(k)}$ is a tuple of trace polynomials, it is continuous on the operator norm ball $\{y \in M_N(\C)_{sa}^m: \norm{y}_\infty \leq R\}$ with a modulus of continuity that only depends on $R$ and does not depend on $N$.  Because $Df^{(k)} \to Df$ uniformly on the operator-norm ball (with rate of convergence independent of $N$), then $Df$ is also continuous on this operator-norm ball with modulus of continuity independent of $N$.

It follows that for every $x, y \in M_N(\C)_{sa}^m$ with $\norm{x}, \norm{y} \leq R$, we have
\[
f(y) - f(x) = \ip{Df(x), y - x}_2 + o(\norm{y - x}_2),
\]
where the error estimate only depends on $R$ and not on $N$.  In particular, this shows $Df$ is uniquely determined by $f$.  Also, it shows that $Df|_{M_N(\C)_{sa}^m}$ is equal to the normalized gradient of $f|_{M_N(\C)_{sa}^m}$ in the ordinary sense of functions on $M_N(\C)_{sa}^m \cong \R^{mN^2}$.
\end{remark}

\begin{lemma} \label{lem:convexityequivalence}
Let $f \in \mathcal{T}_m^{0,1}$ be real-valued.  The following are equivalent:
\begin{enumerate}
	\item The function $f|_{M_N(\C)_{sa}^m}$ is convex for every $N$.
	\item The function $f$ is convex as a function on $(\mathcal{R}_{sa}^\omega)^m$.
	\item There exists a sequence of differentiable convex functions $V_N: M_N(\C)_{sa}^m \to \R$ such that $[V_\bullet] = f$ and $[DV_\bullet] = Df$.  (Here $DV_\bullet$ denotes the sequence $(DV_N)_{N \in \N}$, where $D$ is the normalized gradient understood in the standard sense of calculus.)
\end{enumerate}
\end{lemma}

\begin{proof}
The equivalence between (1) and (2) follows from similar argument to the proof of Lemma \ref{lem:threeseminorms}.

(1) $\implies$ (3) because we can take $V_N = f|_{M_N(\C)_{sa}^m}$.

(3) $\implies$ (1).  Fix $N$.  To prove that $f|_{M_N(\C)_{sa}^m}$ is convex, it suffices to show that $\ip{Df(x) - Df(y), x - y}_2 \geq 0$ for every $x, y \in M_N(\C)_{sa}^m$.  For $k \in \N$, consider $x \otimes I_k$ and $y \otimes I_k$ in $M_{Nk}(\C)_{sa}^m$.  Then as $k \to \infty$,
\[
\ip{Df(x) - Df(y), x - y}_2 = \ip{Df(x \otimes I_k) - Df(y \otimes I_k), x \otimes I_k - y \otimes O_k}_2;
\]
meanwhile, if $R = \max(\norm{x}, \norm{y})$, then since $DV_N - Df \to 0$ in $\norm{\cdot}_2$ uniformly on the operator norm ball of radius $R$, we have as $k \to \infty$ that
\[
\ip{Df(x \otimes I_k) - Df(y \otimes I_k), x \otimes I_k - y \otimes I_k}_2 - \ip{DV_{Nk}(x \otimes I_k) - DV_{Nk}(y \otimes I_k), x \otimes I_k - y \otimes I_k}_2 \to 0.
\]
Because $V_{Nk}$ is convex, the second inner product is $\geq 0$ and therefore $\ip{Df(x) - Df(y), x - y}_2 \geq 0$.
\end{proof}

Let $\mathcal{E}_m(c,C)^{0,1}$ denote the class of $V \in \mathcal{T}_m^{0,1}$ such that $V(x) - (c/2) \norm{x}_2^2$ is convex and $V(x) - (C/2) \norm{x}_2^2$ is concave.  If $0 < c < C$ and if $V \in \mathcal{E}_m(c,C)^{0,1}$, if $V_N = V|_{M_N(\C)_{sa}^m}$, then the sequence of normalized gradients $DV_N$ is asymptotically approximable by trace polynomials.  If we let $\mu_N$ be the corresponding measure, then Theorem \ref{thm:convergenceandconcentration} (the hypothesis \eqref{eq:operatornormhypothesis} being trivially satisfied by unitary invariance) implies that $\mu_N$ concentrates around a non-commutative law $\lambda_V$, which we will call the \emph{free Gibbs state for the potential $V$}.

Furthermore, the free Gibbs state $\lambda_V$ is independent of the choice of representative sequence in the following sense.  Let $\mu_N$ be the measure on $M_N(\C)_{sa}^m$ given by the potential $V_N = V|_{M_N(\C)_{sa}^m}$.  Let $W_N$ be another sequence of potentials satisfying the hypotheses of Theorem \ref{thm:convergenceandconcentration} such that $[W_\bullet] = V$ in $\mathcal{T}_m^{0,1}$, and let $\nu_N$ be the sequence of random matrix measures given by $W_N$.  By Theorem \ref{thm:convergenceandconcentration} $\nu_N$ concentrates around some non-commutative law $\lambda$.  We claim that $\lambda = \lambda_V$.  To prove this, consider the sequence $\tilde{V}_N$ which equals $V_N$ for odd $N$ and $W_N$ for even $N$.  Then $[\tilde{V}_\bullet] = V$ in $\mathcal{T}_m^{0,1}$, which means that $\{D\tilde{V}_N\}_{N \in \N}$ is asymptotically approximable by trace polynomials.  Therefore,
\[
\lambda_V(p) = \lim_{\substack{N \text{ even} \\ N \to \infty}} \int \tau_N(p)\,d\mu_N = \lim_{\substack{N \text{ odd} \\ N \to \infty}} \int \tau_N(p)\,d\nu_N = \lambda(p).
\]

In fact, Lemma \ref{lem:convexityequivalence} implies that the non-commutative laws $\lambda$ which occur as limits in Theorem \ref{thm:convergenceandconcentration} are precisely the free Gibbs laws for potentials $V \in \mathcal{E}_m(c,C)^{0,1}$.  In particular, Theorem \ref{thm:mainthm} implies that $\chi = \underline{\chi} = \chi^*$ for every such law.

\begin{remark}
We have \emph{not} proved that the law $\lambda_V$ is uniquely characterized by the Schwinger-Dyson equation $\lambda[DV(X) f(X)] = \lambda \otimes \lambda[\mathcal{D}f(X)]$, although something like this is implied by \cite{Dabrowski2017}.  One could hope to prove this by letting the semigroup $T_t^V$ act on an abstract space of Lipschitz functions which is the completion of trace polynomials (where the metric now allows $x$ to come from any tracial von Neumann algebra rather than only the $\mathcal{R}^\omega$-embeddable algebras).  We would want to show that if $\lambda$ satisfies the Schwinger-Dyson equation, then $\lambda(T_t^Vu) = \lambda(u)$, but to justify the computation, we need to show more regularity of $T_t^Vu$ than we have done in this paper.  In the SDE approach as well, the proof that $\lambda_V$ is characterized by Schwinger-Dyson is subtle when we do not assume more regularity for $V$ (see \cite{Dabrowski2010}, \cite{Dabrowski2017}).
\end{remark}

\subsection{Examples of Convex Potentials}

A natural class of examples of functions in $\mathcal{E}_m(c,C)^{0,1}$ are those of the form
\[
V(x) = \frac{1}{2} \norm{x}_2^2 + \epsilon f(u)
\]
where $\epsilon$ is a small positive parameter,
\[
u = (u_1,\dots,u_m), \qquad u_j = \frac{x_j + 4i}{x_j - 4i}.
\]
and $f$ is a real-valued trace polynomial in $u$ and $u^*$.  Computations similar to those of \S \ref{subsec:TPdiff} show that the normalized Hessian of $\Jac(Df(u(x)))$ with respect to $x$ is bounded uniformly in $N$.  Therefore, $V \in \mathcal{E}_m(1/2,3/2)^{0,1}$ for sufficiently small $\epsilon$.  Similar examples are described in the introduction of \cite{Dabrowski2017}.  More generally, we can replace the trace polynomial $f(u)$ by a power series where the individual terms are trace monomials in $u$.

The class $\mathcal{E}_m(c,C)^{0,1}$ does not include trace polynomials in $x$ because if $g$ is a trace polynomial of degree $\geq 3$, then we cannot have $g(x)$ convex and $g(x) - (C/2) \norm{x}_2^2$ concave (globally).  However, if we consider a potential which is a small perturbation of a quadratic (as considered in \cite{GMS2006},  \cite{GS2014}), we can fix this problem by introducing an operator-norm cut-off as follows.

Let $f$ be a scalar-valued trace polynomial and let us denote
\begin{equation} \label{eq:exampleV}
V^{(\epsilon)}(x) = \norm{x}_2^2 + \epsilon f(x).
\end{equation}
Let $\phi: \R \to \R$ be a $C_c^\infty$ function such that $\phi(t) = t$ for $|t| \leq R$ and $\phi(t) = 0$ for $|t| \geq 2R$.  Let $\Phi: M_N(\C)_{sa}^m \to M_N(\C)_{sa}^m$ be given by $\Phi_N(x) = (\phi(x_1),\dots,\phi(x_m))$.
\begin{equation} \label{eq:exampleV2}
\tilde{V}_N^{(\epsilon)}(x) = \norm{x}_2^2 + \epsilon f_N(\Phi_N(x)).
\end{equation}
We will prove the following.

\begin{proposition} \label{prop:perturbativeexample0}
Let $\tilde{V}_N^{(\epsilon)}$ be given as above.  Then $[\tilde{V}_\bullet^{(\epsilon)}] \in \mathcal{T}_m^{0,1}$.  Moreover, given $\delta > 0$, we have $[\tilde{V}_\bullet^{(\epsilon)}] \in \mathcal{E}_m(1-\delta,1+\delta)^{0,1}$ for sufficiently small $\epsilon$ (depending on $f$, $R$, and $\delta$).
\end{proposition}

As a consequence, we will deduce the following result about measures defined by $V^{(\epsilon)}$ restricted to an operator-norm ball (without the smooth cut-off $\Phi$).

\begin{proposition} \label{prop:perturbativeexample}
Let $2 < R' < R$, let $f$ be a trace polynomial, and let $V^{(\epsilon)}$ be as in \eqref{eq:exampleV}.  Let
\[
d\mu_N^{(\epsilon)}(x) = \frac{1}{Z_N} \exp(-N^2 V_N^{(\epsilon)}(x)) \mathbf{1}_{\norm{x} \leq R} \,dx.
\]
For sufficiently small $\epsilon$ (depending on $f$, $R$, and $R'$), we have the following.  The measure $\mu_N^{(\epsilon)}$ exhibits exponential concentration around a non-commutative law $\lambda^{(\epsilon)} \in \Sigma_{m,R'}$.  If $X \in (\mathcal{M},\tau)$ is a non-commutative $m$-tuple realizing the law $\lambda^{(\epsilon)}$, then the conjugate variable is given by $DV^{(\epsilon)}(X)$.  Moreover, we have
\[
\chi(\lambda^{(\epsilon)}) = \underline{\chi}(\lambda^{(\epsilon)}) = \chi^*(\lambda^{(\epsilon)}) = \lim_{N \to \infty} \left( \frac{1}{N^2} h(\mu_N^{(\epsilon)}) + \frac{m}{2} \log N \right).
\]
\end{proposition}

To fix notation for the remainder of this section, functions without a subscript, such as $f$, will denote elements of $\mathcal{T}_m^0$ or $\mathcal{T}_m^{0,1}$, and $Df$ will denote the ``gradient'' defined in the abstract space $\mathcal{T}_m^{0,1}$ as the limit of the ``gradients'' of trace polynomials approximating $f$.  However, $f_N$ will denote $f|_{M_N(\C)_{sa}^m}$, and $Df_N$ will denote the normalized gradient $N \nabla f_N$ defined in the usual sense of calculus with respect to $\ip{\cdot,\cdot}_2$.  Moreover, $Hf_N = \Jac(Df_N)$ will denote the Hessian of $f_N$ with respect to $\ip{\cdot,\cdot}_2$.

In order to prove Proposition \ref{prop:perturbativeexample0}, we must understand $D[f_N \circ \Phi_N]$ and $H[f_N \circ \Phi_N]$.  To this end, we recall some results of Peller \cite{Peller2006} on non-commutative derivatives of $\phi(x)$, where $\phi$ is a smooth function on the real line.

For a polynomial $\phi$ in one variable, the non-commutative derivative $\mathcal{D}\phi \in \C\ip{X} \otimes \C\ip{X}$ defined by Definition \ref{def:NCderivative} can be written as the difference quotient
\[
\mathcal{D}\phi(s,t) = \frac{\phi(s) - \phi(t)}{s - t},
\]
where we view $\C\ip{X} \otimes \C\ip{X}$ as a subset of functions on $\R^2$ with the variables $s$ and $t$.  However, the above difference quotient makes sense whenver $\phi: \R \to \C$ is smooth. Thus, it defines an extension of $\mathcal{D}$ to continuously differentiable functions $\phi$ of one variable.

Similarly, if $\phi$ is a polynomial, then the higher order non-commutative derivatives $\mathcal{D}^n \phi$ can be viewed as functions of $n+1$ variables, which are obtained through iterated difference quotients and thus their definition can be extended to smooth functions $\phi$.  (However, beware that we have \emph{not} defined $\mathcal{D}_j^n \phi$ if $\phi$ is a non-polynomial function of multiple variables.)

If $\phi$ is a polynomial, then to estimate $\phi(X) - \phi(Y)$ for operators $X$ and $Y$ with norm bounded by $R$, one seeks to control the norm of $\mathcal{D}\phi$ in the projective tensor product $L^\infty[-R,R] \widehat{\otimes} L^\infty[-R,R]$.  Similarly, if $\phi$ is a smooth function and $\phi(X)$ and $\phi(Y)$ are defined through functional calculus, one can estimate the operator norm $\norm{\phi(X) - \phi(Y)}$ by representing $\phi$ as an integral of simpler functions (e.g.\ by Fourier analysis) whose non-commutative derivatives are easier to analyze.  In this case, it is convenient to write $\mathcal{D}\phi$ as an \emph{integral} rather than a \emph{sum} of simple tensors.

We thus consider the integral projective tensor powers of the space of bounded Borel functions $\mathcal{B}(\R)$.  The \emph{integral projective tensor product} $\mathcal{B}(\R)^{\widehat{\otimes}_i n}$ consists of Borel functions $G$ on $\R^n$ which admit a representation
\begin{equation} \label{eq:integraldecomposition}
G(x_1,\dots,x_n) = \int_{\Omega} G_1(x_1,\omega) \dots G_n(x_n,\omega)\,d\mu(\omega)
\end{equation}
for some measure space $(\Omega,\mu)$ such that
\begin{equation} \label{eq:integraldecomposition2}
\int_{\Omega} \norm{G_1(\cdot,\omega)}_{\mathcal{B}(\R)} \dots \norm{G_1(\cdot,\omega)}_{\mathcal{B}(\R)} \,d\mu(\omega) <+\infty
\end{equation}
and we define $\norm{G}_{\mathcal{B}(\R)^{\widehat{\otimes}_i n}}$ to be the infimum of \eqref{eq:integraldecomposition2} over all representations \eqref{eq:integraldecomposition}.

Given $G \in \mathcal{B}(\R)^{\widehat{\omega}_i n}$, bounded self-adjoint operators $x_0$, \dots, $x_n$ and bounded operators $y_1$, \dots, $y_n$, we define
\begin{equation}
G(x_0,\dots,x_n) \# (y_1 \otimes \dots \otimes y_n) = \int_{\Omega} G_0(x_0,\omega) y_1 G_1(x_1,\omega) \dots y_n G_n(x_n,\omega)\,d\mu(\omega),
\end{equation}
where $G_0$, \dots, $G_n$ satisfy \eqref{eq:integraldecomposition}.  This is well-defined by \cite[Lemma 3.1]{Peller2006}.  If the $x_j$'s and $y_j$'s are elements of a tracial von Neumann algebra $(\mathcal{M},\tau)$, we have by the non-commutative H\"older's inequality (see \S \ref{subsec:NCLP}) that if $1/\alpha = 1/\alpha_1 + \dots + 1/\alpha_n$, then
\begin{equation} \label{eq:Holderestimate}
\norm{G(x_0,\dots,x_n) \# (y_1 \otimes \dots \otimes y_n)}_\alpha \leq \norm{G}_{\mathcal{B}(\R)^{\widehat{\otimes}_i (n+1)}} \norm{y_1}_{\alpha_1} \dots \norm{y_n}_{\alpha_n},
\end{equation}
Moreover, we have the following bounds on the non-commutative derivatives of $\phi$ as a corollary of the results of \cite{Peller2006}.

\begin{proposition} \label{prop:Peller}
There exists a constant $K_n$ such that for all $\phi \in C_c^\infty(\R)$,
\begin{equation} \label{eq:noncommutativederivativeestimate}
\norm{\mathcal{D}^n \phi}_{\mathcal{B}(\R)^{\widehat{\otimes}_i (n+1)}} \leq K_n \int_{\R} |\widehat{\phi}(\xi) \xi^n|\,d\xi.
\end{equation}
\end{proposition}

\begin{proof}
As in \cite[{\S 2}]{Peller2006}, choose $w \in C_c^\infty$ such that $0 \leq w \leq \chi_{[-1/2,2]}$ and $\sum_{k \in \Z} w(2^{-k} \xi) = 1$ for $\xi > 0$.  Let $W_k$ and $W_k^{\#}$ be given by $\widehat{W}_k(\xi) = w(2^{-n} \xi)$ and $\widehat{W}_k^{\#}(\xi) = w(-2^{-k}x)$ where $\widehat{\cdot}$ denotes the Fourier transform.  It is shown in \cite[Theorem 5.5]{Peller2006} that
\[
\norm{\mathcal{D}^n \phi}_{\mathcal{B}(\R)^{\widehat{\otimes}_i (n+1)}} \leq K_n \sum_{k \in \Z} 2^{nk} \left( \norm{W_k * \phi}_{L^\infty(\R)} + \norm{W_k^{\#} * \phi}_{L^\infty(\R)} \right).
\]
This can be estimated by the right hand side of \eqref{eq:noncommutativederivativeestimate} (for a possibly different constant) by a standard Fourier analysis computation.
\end{proof}

\begin{proof}[Proof of Proposition \ref{prop:perturbativeexample0}]
Recall that $\tilde{V}_N^{(\epsilon)}(x) = \frac{1}{2} \norm{x}_2^2 + \epsilon f_N \circ \Phi_N$.  Thus, to show that the sequence $V_N^{(\epsilon)}$ defines an element of $\mathcal{T}_0^m$, it suffices to prove this for $f_N \circ \Phi_N$.  To this end, it is sufficient to show that for each $r > 0$, there is a sequence of trace polynomials $\{g^{(k)}\}_{k\in\N}$ such that
\[
\lim_{k \to \infty} \sup_{N \in \N} \sup_{x \in M_N(\C)_{sa}^m: \norm{x}_\infty \leq r} |g^{(k)}(x) - f_N \circ \Phi_N(x)| = 0
\]
and
\[
\lim_{k \to \infty} \sup_{N \in \N} \sup_{x \in M_N(\C)_{sa}^m: \norm{x}_\infty \leq r} \norm{Dg^{(k)}(x) - D[f_N \circ \Phi_N(x)]}_2
\]
Fix $r > 0$.  By standard approximation techniques, there exist Schwarz functions $\phi^{(k)}: \R \to \R$ such that $\phi^{(k)}|_{[-r,r]}$ is a polynomial and $\phi^{(k)} \to \phi$ in the Schwarz space as $k \to \infty$.  By Proposition \ref{prop:Peller}, we have $\mathcal{D}^n \phi^{(k)} \to \mathcal{D}^n \phi$ in $\mathcal{B}(\R)^{\widehat{\otimes}_i (n+1)}$ as $k \to \infty$ for every $n$.

Let $\Phi_N^{(k)}(x_1,\dots,x_m) = (\phi^{(k)}(x_1),\dots,\phi^{(k)}(x_m))$.  Then $f_N \circ \Phi_N^{(k)}$ is given by a trace polynomial $g^{(k)}$ on $\{\norm{x}_\infty \leq r\}$.  Because of the spectral mapping theorem,
\[
\sup_{\norm{x} \leq r} \norm{\Phi_N^{(k)}(x) - \Phi_N(x)}_\infty \leq m \sup_{t \in [-r,r]} |\phi^{(k)}(t) - \phi(t)|
\]
which is independent of $N$ and vanishes as $k \to \infty$.  Thus, our trace polynomials $g^{(k)}$ approximate $f_N \circ \Phi_N$ uniformly on the operator norm ball $\{x: \norm{x}_\infty \leq r\}$.

Next, we must show that $D g^{(k)}$ approximates $D[f_N \circ \Phi_N]$ uniformly in $\norm{\cdot}_2$ on the operator norm ball $\{\norm{x}_\infty \leq r\}$.  By the chain rule, we have
\[
D_j[f_N \circ \Phi_N] = \Jac_j(\Phi_N)^t[D_j f_N],
\]
where $D_j$ and $\Jac_j$ are the normalized gradient and Jacobian with respect to the variable $x_j \in M_N(\C)_{sa}$.  Now
\[
\Jac_j(\Phi_N)(x) y = \mathcal{D} \phi(x_j) \# y.
\]
Now $\mathcal{D} \phi$ viewed as an element of the tensor product $\C[X] \otimes \C[X]$ is is invariant under the flip map that switches the order of the tensorands; this is because $\mathcal{D}\phi$ is represented as a difference quotient for one-variable functions.  Flip invariance implies that
\[
\tau_N[ (\mathcal{D}\phi(x_j) \# y)z] = \tau_N[y(\mathcal{D}\phi(x_j) \# z)],
\]
which means that the operator $\Jac_j(\Phi_N)(x)$ on $M_N(\C)_{sa}$ is self-adjoint.  Hence,
\[
D_j[f_N \circ \Phi_N](x) = \Jac_j(\Phi_N(x))[D_j f_N](x) = \mathcal{D}\phi(x_j) \# D_j f_N(\Phi_N(x)).
\]
This function is given by a trace polynomial on $\{\norm{x}_\infty \leq r\}$, and specifically it must equal $D_j g^{(k)}$ because $D_j g^{(k)}$ is uniquely determined as a trace polynomial by the fact that it is the gradient of $g^{(k)}|_{M_N(\C)_{sa}^m}$ for every $N$.  Moreover, for $\norm{x}_\infty \leq r$, we have
\[
\mathcal{D}\phi_k(x_j) \# D_j f(\Phi_k(x)) = \mathcal{D}\phi_k(x_j) \# D_j f(\Phi(x)) + \mathcal{D}\phi_k(x_j) \# [D_j f(\Phi_k(x)) - D_j f(\Phi(x))].
\]
The first term converges to $\mathcal{D} \phi(x_j) \# D_j f(\Phi(x))$ in $\norm{\cdot}_2$ uniformly on $\{\norm{x}_\infty \leq r\}$ using \eqref{eq:Holderestimate} with estimates independent of $N$.  Similarly, because the images of $\Phi_k$ and $\Phi$ are contained in an operator norm ball and $D_j f$ is $K$-Lipschitz in $\norm{\cdot}_2$ on this ball for some $K > 0$, we have $D_j f(\Phi_k(x)) - D_j f(\Phi(x)) \to 0$ uniformly.  This in turn implies that the second term goes to zero because $\mathcal{D} \phi_k(x_j)$ is uniformly bounded in $\mathcal{B}(\R) \widehat{\otimes}_i \mathcal{B}(\R)$.  Thus, for every $r > 0$, there is a sequence of trace polynomials $g^{(k)}$ such that $g_k \to f \circ \Phi$ and $Dg^{(k)} \to D(f \circ \Phi)$ uniformly on $\{\norm{x}_\infty \leq r\}$.   This means that $f \circ \Phi \in \mathcal{T}_m^{1,0}$.

It follows that the sequence $V_N^{(\epsilon)}$ defines a function in $\mathcal{T}_m^{0,1}$ for every $\epsilon$.  It remains to show that this function is in $\mathcal{E}_m(1-\delta,1+\delta)^{0,1}$ for sufficiently small $\epsilon$.  To this end, it suffices to show that $f_N \circ \Phi_N$ defines a function in $\mathcal{E}_m(-a,a)^{0,1}$ for some real $a > 0$.  Thus, we only need to obtain some uniform upper and lower bounds on the operator norm of $H[f_N \circ \Phi_N]$ that are independent of $N$.  However, this is equivalent to showing that $D_j(f_N \circ \Phi_N) = \mathcal{D}\phi_N(x_j) \# D_j f_N(\Phi+N(x))$ is Lipschitz in $\norm{\cdot}_2$ for each $j$ (uniformly in $N$).  Because $\mathcal{D}^2 \phi$ is bounded in $\mathcal{B}(\R) \widehat{\otimes}_i \mathcal{B}(\R) \widehat{\otimes}_i \mathcal{B}(\R)$, we see that
\[
\norm{\mathcal{D} \phi(x_j) \# y - \mathcal{D} \phi(x_j') \# y}_2 \leq K \norm{x_j - x_j'}_2 \norm{y}_2
\]
for some constant $K$.  Together with the fact that $D_j f_N(\Phi_N(x))$ is Lipschitz in $\norm{\cdot}_2$, this implies that $D_j(f_N \circ \Phi_N)$ is Lipschitz in $\norm{\cdot}_2$ as desired.
\end{proof}

\begin{proof}[Proof of Proposition \ref{prop:perturbativeexample}]
Let $\tilde{\mu}_N$ be the measure on $M_N(\C)_{sa}^m$ given by the potential $\tilde{V}$.  Let $\delta$ be a number in $(0,1)$ to be chosen later.  By Proposition \ref{prop:perturbativeexample0}, we have that $\tilde{V}^{(\epsilon)} \in \mathcal{E}_m(1-\delta,1+\delta)^{0,1}$ for sufficiently small $\epsilon$.  By Theorem \ref{thm:convergenceandconcentration}, the laws $\tilde{\mu}_N$ concentrate around a non-commutative law $\lambda$.  Furthermore, in Theorem \ref{thm:convergenceandconcentration} (1), we can take $M = 0$ and $c = 1 - \delta$ and $C = 1 + \delta$, so that
\[
\limsup_{N \to \infty} R_N \leq \frac{2}{(1 - \delta)^{1/2}} + \frac{\norm{D\tilde{V}(0)}_2}{1 - \delta} + \frac{\delta}{(1 - \delta)^{3/2}}.
\]
Note that $D\tilde{V}(0) = DV(0) = \epsilon Df(0)$ is a scalar multiple of the identity matrix since $f$ is a trace polynomial.  Because $R' > 2$, we may choose $\delta$ sufficiently small that
\[
\frac{2}{(1 - \delta)^{1/2}} + \frac{\delta}{(1 - \delta)^{3/2}} < R'.
\]
Then by choosing $\epsilon$ (and hence $\norm{D\tilde{V}(0)}_2$) sufficiently small, we can arrange that $R_* = \limsup_{N \to \infty} R_N < R'$.  This implies that the measures $\tilde{\mu}_N$ concentrate on the ball $\{\norm{x}_\infty \leq R'\}$.  For $\norm{x}_\infty \leq R$, we have $\tilde{V}(x) = V(x)$, and therefore $\mu_N$ is the (normalized) restriction of $\tilde{\mu}_N$ to $\{\norm{x}_\infty \leq R\}$.  It follows that $\mu_N$ concentrates around the law $\lambda$ as well.

If $X \in (\mathcal{M},\tau)$ realizes the law $\lambda$, then $\norm{X}_\infty \leq R'$ since $\lambda \in \Sigma_{m,R_*} \subseteq \Sigma_{m,R'}$ by Theorem \ref{thm:convergenceandconcentration} (2).  Moreover, by Proposition \ref{prop:convergenceofFisherinfo}, the conjugate variables for $\lambda$ are given by $D\tilde{V}(X) = DV(X)$.  Moreover, by Theorem \ref{thm:mainthm} applied to $\tilde{\mu}_N$, we have
\[
\chi(\lambda) = \underline{\chi}(\lambda) = \chi^*(\lambda) = \lim_{N \to \infty} \left( \frac{1}{N^2} h(\tilde{\mu}_N) + \frac{m}{2} \log N \right).
\]
In the last equality, we can replace $\tilde{\mu}_N$ by $\mu_N$ as in the proof of Proposition \ref{prop:convergenceofentropy} because $\tilde{\mu}_N$ is concentrated on $\{\norm{x}_\infty \leq R'\}$.
\end{proof}

\begin{remark}
The approach given here probably does not give the optimal range of $\epsilon$ for Proposition \ref{prop:perturbativeexample}.  To get the best result, one would want a more direct way to extend the potential $V: \{\norm{x}_\infty \leq R\} \to \R$ to a potential $\tilde{V}$ defined everywhere.  This leads us to ask the following question.

Suppose that $V$ is a real-valued function in the closure of trace polynomials with respect to the norm $\norm{f}_{\mathcal{T}_m^0,R} + \norm{Df}_{\mathcal{T}_m^1,R}$, and hence $V$ defines a function $\{x: \norm{x}_\infty \leq R\} \to \R$ for $x \in M_N(\C)_{sa}^m$.  If $V(x) - (c/2) \norm{x}_2^2$ is convex and $V(x) - (C/2) \norm{x}_2^2$ is concave on $\{\norm{x} \leq R\}$, then does $V$ extend to a potential $\tilde{V} \in \mathcal{E}_m(c,C)^{0,1}$?  What if we allow $\tilde{V}$ to have slightly worse constants $c$ and $C$?

The construction of extensions that preserve the convexity properties is not difficult, but it is less obvious how to construct an extension that one can verify preserves the approximability by trace polynomials.
\end{remark}

\bibliographystyle{siam}
\bibliography{free_entropy}

\end{document}